\documentclass[12pt]{article}
\newcommand{\Vtwocolumns}{false}
\newcommand{\Vlongreport}{false} 
\newenvironment{proof}      
{\par\noindent\textbf{Proof}}{\eopp\smallskip\vskip 3 pt}            
\newcommand{\eopp}{\hspace*{\fill}{$\blacksquare$} } 
\usepackage{hslTRArXiv}
\usepackage{fullpage}

\usepackage[crop=pdfcrop]{pstool}
\usepackage{amssymb}
\usepackage{amsmath}
\usepackage{graphicx}
\usepackage{color,soul}
\usepackage{cases}
\usepackage{psfrag}
\usepackage{balance}
\usepackage{multirow}
\usepackage{subfigure}
\usepackage{enumerate}
\usepackage{lipsum}
\usepackage{ifthen}

\newcommand{\comment}[1]{\null}  
\usepackage{rgsEnvironments}				
\usepackage{rgsMacros}					

\setboolean{Conference}{false}

\newboolean{JournalTwoC}
\setboolean{JournalTwoC}{\Vtwocolumns}
\newcommand{\NotForJtwoC}[1]{\ifthenelse{\boolean{JournalTwoC}}{}{#1}}
\newcommand{\IfJtwoC}[2]{\ifthenelse{\boolean{JournalTwoC}}{#1}{#2}}

\newboolean{JournalSub}
\setboolean{JournalSub}{\Vlongreport}
\newcommand{\NotForJournal}[1]{\ifthenelse{\boolean{JournalSub}}{}{#1}}
\newcommand{\IfJournal}[2]{\ifthenelse{\boolean{JournalSub}}{#1}{#2}}

\def\startmodif{\color{black}}
\def\stopmodif{\color{black}\normalcolor}

\newcommand{\He}{{\mbox{\rm He}}}

\begin{document}


\IfJtwoC{

\title{Interconnected Observers for Robust Decentralized Estimation with Performance Guarantees and Optimized Connectivity Graph}

\author{Yuchun Li and
        Ricardo G. Sanfelice~\IEEEmembership{}
\thanks{Y. Li and R. G. Sanfelice are with the Department
of Computer Engineering, University of California, Santa Cruz,
CA, 95064, USA. E-mail: {\tt\small yuchunli,ricardo@ucsc.edu}. 
} \vspace{-2.5ex}}

\maketitle
}{

\ititle{Technical Report: Interconnected Observers for Robust Decentralized Estimation with Performance Guarantees and Optimized Connectivity Graph}
\iauthor{
  Yuchun Li \\
  {\normalsize yuchunli@ucsc.edu} \\
  Ricardo G. Sanfelice \\
  {\normalsize ricardo@ucsc.edu}}
\idate{\today{}} 
\iyear{2015}
\irefnr{02}
\makeititle


\title{Interconnected Observers for Robust Decentralized Estimation with Performance Guarantees and Optimized Connectivity Graph}

\author{Yuchun Li and Ricardo G. Sanfelice
\thanks{Y. Li and R. G. Sanfelice are with the Department
of Computer Engineering, University of California, Santa Cruz,
CA, 95064, USA. E-mail: {\tt\small yuchunli,ricardo@ucsc.edu}.
This research has been partially supported by the National Science Foundation under CAREER Grant no. ECS-1150306 and by the Air Force Office of Scientific Research under Grant no. FA9550-12-1-0366. 
} }
\date{}
\maketitle

}


%

\begin{abstract}
Motivated by the need of
observers that are both robust to disturbances
and guarantee fast convergence to zero of the estimation error,
we propose an observer for linear time-invariant systems
with noisy output
that consists of the combination of
$N$ coupled observers over a connectivity graph.
At each
node of the graph, the output of these {\em interconnected observers}
is defined
as the average of the estimates obtained using
local information.
The convergence rate and the robustness to measurement
noise of the proposed observer's output are characterized
in terms of
$\cal KL$ bounds.
Several optimization problems are formulated
to design the proposed observer so as to satisfy
a given rate of convergence specification while minimizing
the $H_{\infty}$ gain from noise to estimates
or the size of the connectivity graph.
It is shown that that the interconnected observers
relax the well-known tradeoff between rate of convergence
and noise amplification, which is a property attributed to the
proposed innovation term that, over the graph, couples the estimates between
the individual observers.
Sufficient conditions involving information
of the plant only, assuring that
the estimate
obtained at each node of the graph
outperforms the one obtained with a single, standard
Luenberger observer are given.
The results are illustrated in
several examples throughout the paper.
\end{abstract}
\IfJtwoC{\vspace{-10pt}}{}
\section{Introduction}
\label{sec:introduction}
We consider linear time-invariant systems
of the form
\begin{flalign}
      \begin{split}
              \dot x & =Ax,\quad
               y = Cx+m(t),
      \end{split}\label{eq:plant}
\end{flalign}
where $x\in \mathbb{R}^n$, $y\in \mathbb{R}^p$, and $t \mapsto m(t)$ denotes measurement noise,
for which there exists a Luenberger observer
      \begin{flalign}
      \begin{split}
              {\dot{\hat x}}_L & =A{\hat x}_L - K_L({\hat y}_L -y), \quad
              {\hat y}_L = C{\hat x}_L
      \end{split}\label{eq:singleob}
      \end{flalign}
leading to the exponentially stable error system
\IfJtwoC{
      \begin{align}
      \!\!{\dot e}_L  \!=\! (A \!\!-\!\! K_L C) e_L \!\!+\!\! K_L m(t) \!:=\! \tilde{A}_L e_L \!+\! K_L m(t)
      \label{eq:singleob_error}
      \end{align}
}{
      \begin{flalign}
      \begin{split}
              {\dot e}_L  = (A - K_L C) e_L + K_L m(t) := \tilde{A}_L e_L + K_L m(t)
      \end{split}\label{eq:singleob_error}
      \end{flalign}}
with estimation error given by $e_L:= {\hat x}_L - x$.
It is well-known that,
under observability conditions of \eqref{eq:plant},
the matrix gain $K_L$
can be chosen to make the convergence rate of \eqref{eq:singleob_error}
arbitrarily fast. However, due to the fast convergence speed requiring a large gain,
the price to pay is that the effect of measurement noise $m$ is amplified.
Indeed, the design of observers, such as those in the form \eqref{eq:singleob},
involves a tradeoff between convergence rate and robustness to measurement noise \cite{Luenberger1971,72tradeoff}. 
In fact, in \cite[page 597]{Luenberger1971},
D. G. Luenberger makes the following remark about the error system \eqref{eq:singleob_error} when $(C,A)$ is observable: ``Theoretically, the eigenvalues can be moved arbitrarily toward minus infinity, yielding extremely rapid convergence. This tends, however, to make the observer act like a differentiator and thereby become highly sensitive to noise, and to introduce other difficulties.'' 
Along the same lines, the authors of \cite{72tradeoff} recognize the compromise between performance and robustness in the design of \eqref{eq:singleob}:
``At this point we can only offer some intuitive guidelines for a choice of $K$ to obtain satisfactory performance of the observer. To obtain fast convergence of the reconstruction error to zero, $K$ should be chosen so that the observer poles are quite deep in the left-half complex plane. This, however, generally must be achieved by making the gain matrix $K$ large, which in turn makes the observer very sensitive to any observation noise that may be present, added to the observed variable $y(t)$. A compromise must be found,'' see \cite[page 332]{72tradeoff}. Unfortunately, this issue is also at the core of every estimation algorithm for multi-agent systems.
\IfJtwoC{\vspace{-9pt}}{}
\subsection{Related work}
Several
observer architectures and design methods
with the goal of conferring good performance
and robustness to the error system have been proposed in the literature.
In particular, $H_{\infty}$ tools have been employed to formulate sets of Linear Matrix Inequalities (LMIs)
that, when feasible, guarantee that the
${\cal L}_2$ gain from disturbance
to the estimation error is below a pre-established upper bound;
see, e.g., \IfJtwoC{\cite{94.boyd.book.lmi,Li.Fu.97.TAC,Marquez2003},}{\cite{94.boyd.book.lmi,Li.Fu.97.TAC,06.Jung.adaptive.observer,Marquez2003},} to just list a few.
Following ideas from adaptive control \IfJtwoC{\cite{Ioannou.Sun.96},}{\cite{Egardt.79,Ioannou.Sun.96},}
observers with a gain that adapts to the plant measurements
have been proposed in  \cite{Astolfi.Praly.06, Andrieu2009},
though the presence of measurement noise can lead
to large values of the gains. 
\IfJtwoC{Such issues also emerge in the design of high-gain observers, where the
use of high gain can significantly amplify the effect of measurement noise, as in \cite{Khalil.09.switched-gain}. More recently, observers using adaptive gains \cite{05.lei.global.observer}, two gains designed with different objectives \cite{11.khalil.nonlinear.hign.gain.noise,11.Sanfelice.high.gain.adaption}, and switching between two observers \cite{12.Shim.Observer} have been found successful in certain settings.
}{
Such issues also emerge in the design of high-gain observers, where the use of high gain can significantly amplify the effect of measurement noise. Indeed, in \IfJtwoC{\cite{Khalil.09.switched-gain},}{\cite{06.khalil.highgain.noise,Khalil.09.switched-gain},} 
it is  shown that  measurement noise introduces an upper limit for the gain of a high-gain observer with constant gain when good performance is desired. More recently, observers using essentially two set of gains, one set optimized for convergence
and the other for robustness, have been found successful in certain settings.
Such approaches include the piecewise-linear gain approach in \cite{11.khalil.nonlinear.hign.gain.noise} for
simultaneously satisfying steady-state and transient bounds,
the high gain observer with nonlinear adaptive gain in \cite{11.Sanfelice.high.gain.adaption}, and the high gain observer with on-line gain tuning in \cite{05.lei.global.observer}. Also recently, a switching algorithm combining two observers for performance improvement was proposed in \cite{12.Shim.Observer}.}

Recent research efforts in multi-agent systems have lead to enlightening results in distributed estimation and consensus. Distributed Kalman filtering are employed for achieving spatially-distributed estimation tasks in \cite{09.Cortes.DistributedKalman} and for sensor networks in \cite{05.Olfati.DistributedKalmanConsensus, 05.spanos.distributedKalman, 06.Alriksson.DistributedKalman, 07.Olfati.DistriibutedKalman, 08.Carli.DistributedKalman}. 
To characterize the effect of unmodeled dynamics on the consensus multi-agent system, in \cite{12.Zhao.Consensus.Hinfity}, a region-based approach is used for distributed $H_\infty$-based consensus of multi-agent systems with an undirected graph. For dynamic average consensus, \cite{13.Cortes.Consensus} proposes a decentralized algorithm that guarantees asymptotic agreement of a signal over strongly connected and weight-balanced graphs. In \cite{08.Hong.distributedobservers}, switching inter-agent topologies are used to design distributed observers for a leader-follower problem in multi-agent systems. 
 For estimating the trajectory of a moving target with perturbed dynamics, nonlinear filters based on networked sensors are proposed in \cite{Hu20102041,Hu2014}. 
However, distributed estimation algorithms that systematically meet specifications of performance and robustness to measurement noise are not available.
\IfJtwoC{\vspace{-9pt}}{}
\subsection{Contributions}
We propose a new observer, called {\em interconnected observers}, with improved  convergence rate of the estimation error and robustness to measurement noise, when compared with the observer in \eqref{eq:singleob}. The proposed observer consists of $N$ linear time-invariant observers interconnected over a graph. The local estimate at each node is provided by an observer featuring an innovation term that appropriately injects the estimate obtained from its neighbors, which can be computed in a decentralized manner. The global estimate of the state of the plant is given by the average of the local estimates.

The main contributions of this paper are threefold. 
\begin{enumerate}
\item [1)] We establish that, under certain conditions involving its parameters,
and when compared to the Luenberger observer in \eqref{eq:singleob}, the proposed observer significantly improves the rate of convergence and the gain from measurement noise to estimation error, with improvements of more than  $50 \%$ at times (see Table~\ref{tab:global_numberofagent}).
\item [2)] We characterize the convergence rate and the robustness to measurement noise of the proposed observer in terms of $\cal KL$ bounds, which provide useful information on how the parameters of the observers affect such properties.
\item [3)] We formulate optimization problems for the purpose of the design of interconnected observers. 
\begin{enumerate}
\item [\romannumeral1)] For a fixed rate of convergence, optimization problems are proposed for the design of interconnected observers with optimized gain from measurement noise to estimation error (local and global). 
\item [\romannumeral2)] For a fixed rate of convergence and a desired $H_\infty$ gain, optimization problems that minimize the number of edges of the connectivity graph are also formulated. 
\IfJtwoC{}{Using appropriate coordinates and conditions, we show that these problems can be converted into convex optimization problems that can be used to efficiently design interconnected observers. }
\item [\romannumeral3)] An LMI condition involving only information about the plant is provided to guarantee that the estimate obtained at each node of the graph outperforms the one obtained with a single, standard Luenberger observer\IfJtwoC{.}{, which uniquely relaxes the well-known trade off between rate of convergence and noise amplification.}
\end{enumerate}   
\end{enumerate}
Examples throughout the paper illustrate our results and their applicability to estimation in multi-agent systems, such as mobile and sensor networks. To the best of our knowledge, we are not aware of an observer in the literature that guarantees such properties simultaneously.
\IfJtwoC{\vspace{-9pt}}{}
\subsection{Organization of the Paper}
The remainder of this paper is organized as follows.
In Section~\ref{sec:MotivationalExample}, the idea and benefits behind interconnected observers are presented in a motivational example.
Section~\ref{sec:coupledpairofLobserver} introduces the proposed observer in general form, the $\cal KL$ bounds, and the design methods in terms of optimization problems.  \NotForJtwoC{Section~\ref{sec:discussion} discusses the optimization
of the number of nodes in the graph, a decentralized method to compute a global estimate, and a relationship with the optimal observer.}
\IfJtwoC{\vspace{-7pt}}{}
\section{Motivational Example}
\label{sec:MotivationalExample}
Consider the scalar plant 
     \begin{flalign}
     \begin{split}
             \dot x &=ax, \quad 
             y = x+m,
     \end{split}\label{eq:scalarplant}
     \end{flalign} 
where $m$ denotes measurement noise and $a<0$.
Suppose we want to estimate the state $x$ from measurements of $y$. Following \eqref{eq:singleob}, a Luenberger observer for \eqref{eq:scalarplant} is {given by}
     \begin{flalign}
     \begin{split}
             {\dot{\hat x}}_L &=a{\hat x}_L - {K}_L ({\hat y}_L - y), \quad 
             {\hat y}_L = {\hat x}_L.
     \end{split}\label{eg:scalarplantob}
     \end{flalign}
The resulting estimation error system is given by \eqref{eq:singleob_error}
with $\tilde{A}_L = a - K_L$.
Its rate of convergence  is $a-K_L$ and, when $m$ is constant, its  steady-state error is $e_L^\star : = \frac{K_L}{K_L-a}m$. It is apparent that to get fast convergence rate, the constant $K_L$ needs to be positive and  large.
However, as argued in the introduction, with $K_L$ large, the effect of measurement noise is amplified. In light of recent popularity of multi-agent systems, it is natural to explore the advantages of using more than one measurement of the plant's output so as to overcome such a tradeoff.   
            
In this paper, we show that it is possible to design interconnected observers that are capable of relaxing the said tradeoff. 
\IfJtwoC{}{More precisely, interconnected observers are proposed to improve the rate of convergence and the robustness to measurement noise, when compared to a  single Luenberger observer.}
To illustrate the idea behind the proposed observer, consider the estimation  of the state of the scalar plant \eqref{eq:scalarplant} with two agents, each taking its own measurement of $y$. The two agents can communicate with each other according to the following directed graph: agent $1$ can transmit information to agent $2$, but agent $2$ cannot send data to agent $1$. This is shown in Figure~\ref{fig:2nodes_to_neighbor_only}.
         \begin{figure}[!h]
         \centering
         {
         \psfragfig[trim=0.1cm 0.5cm 0.5cm 0.5cm,clip,width=0.18\textwidth]{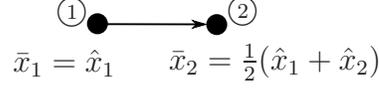}
         {
         \psfrag{[a]}[][][1]{\textcircled{\scriptsize$1$}}
         \psfrag{[b]}[][][1]{\hspace{-0.28in}\textcircled{\scriptsize$2$}}
         \psfrag{[c]}[][][1]{$\bar x_1 = \hat x_1$}
         \psfrag{[d]}[][][1]{\hspace{0.4in}$\bar x_2 = \frac{1}{2}(\hat x_1 + \hat x_2)$}
         }
         \caption{Two agents connected as a direct graph.}
         \label{fig:2nodes_to_neighbor_only}
         }
     \end{figure}
     
Following the approach in this paper, an interconnected observer would take the form
\IfJtwoC{
     \begin{flalign}
     \begin{split}
             \dot{\hat x}_1 &=a{\hat x}_1 - K_{11}({\hat y}_1 - y_1),\\
             \dot{\hat x}_2 &=a{\hat x}_2 - K_{22}({\hat y}_2 - y_2)- K_{21} ({\hat y}_1 - y_1),  \\
             {\hat y}_1 &= {\hat x}_1, \
             {\hat y}_2 = {\hat x}_2, \
             {\bar x}_1  = {\hat x}_1, \ {\bar x}_2 = (1/2)({\hat x}_1+{\hat x}_2),
     \end{split}\label{eq:coupled_obs}
     \end{flalign}
}{
     \begin{flalign}
     \begin{split}
             \dot{\hat x}_1 &=a{\hat x}_1 - K_{11}({\hat y}_1 - y_1),  \quad 
             \dot{\hat x}_2 =a{\hat x}_2 - K_{22}({\hat y}_2 - y_2)- K_{21} ({\hat y}_1 - y_1),  \\
             {\hat y}_1 &= {\hat x}_1, \quad
             {\hat y}_2 = {\hat x}_2, \quad
             {\bar x}_1  = {\hat x}_1, \quad {\bar x}_2 = \frac{{\hat x}_1+{\hat x}_2}{2},
     \end{split}\label{eq:coupled_obs}
     \end{flalign}}
where $\hat x_i$ and $\bar x_i$ are associated with agent $i$, each measured plant output $y_i$ is corrupted by measurement noise $m_i$, that is $y_1= x + m_1$ and  $y_2 = x + m_2$, respectively, where $m_i$'s are independent. The term ``$- K_{21} ({\hat y}_1 - y_1)$''
defines an innovation term exploiting the information shared by agent $1$ with agent $2$. 
The output $\bar{x}_i$ of agent $i$ defines the local 
estimate (at agent $i$) of $x$. Since agent $1$ only has access to its own information, we have $\bar x_1 = \hat{x}_1$, while since agent $2$ has also information from its neighbor, agent $2$'s output $\bar x_2$ can be taken as the average of the states $\hat{x}_1$ and $\hat{x}_2$.\footnote{In general, ${\bar x}_2$ could be the convex combination of $\hat x_1$ and $\hat x_2$, i.e., ${\bar x}_2 = s_1 \hat x_1 + s_2 \hat x_2, s_1 + s_2 = 1, s_1,s_2 \in \mathbb{R}$. }

To analyze the estimation error induced by the interconnected observer in \eqref{eq:coupled_obs}, define error variables $e_i := {\hat x}_i-x, i\in\{1,2\}$. Then, the error system is given by
\IfJtwoC{
     \begin{flalign}
     \begin{split}
             \dot{ e}_1 &=(a - K_{11}){e}_1 + K_{11} m_1,  \\
             \dot{ e}_2 &=- K_{21}{e}_1 \!+\! (a \!-\! K_{22}){e}_2 \!+\! K_{21} m_1 \!+\! K_{22} m_2,  \\
     \end{split}\label{eq:coupled_obs_error}
     \end{flalign}
}{   \begin{flalign}
     \begin{split}
             \dot{ e}_1 &=(a - K_{11}){e}_1 + K_{11} m_1,  \quad 
             \dot{ e}_2 =- K_{21}{e}_1 + (a-K_{22}){e}_2 + K_{21} m_1 + K_{22} m_2,  \\
     \end{split}\label{eq:coupled_obs_error}
     \end{flalign}}
which can be written in matrix form as 
     \begin{flalign}
     \begin{split}
             \dot{ e} &=\tilde{A}{e}+\tilde{K}m,
     \end{split}\label{eq:coupled_obs_error_matrix_oneway}
     \end{flalign}
where $e=[e_1\ e_2]^\top$, $m =[m_1\ m_2]^\top$,  
\IfJtwoC{
     \begin{flalign}
     \begin{split}
             \tilde{A}\!=\!\left[\!\!
             \begin{array}{cc}
             a-K_{11} &  0 \\
             -K_{21}  & a-K_{22}
             \end{array}
             \!\!\right],
             \tilde{K} \!=\! \left[\!\!
             \begin{array}{cc}
             K_{11} & 0  \\
             K_{21} & K_{22}
             \end{array}
             \!\!\right].
     \end{split}\label{eq:coupled_closed_A_matrix_oneway}
     \end{flalign}
}{
     \begin{flalign}
     \begin{split}
             \tilde{A}=\left[
             \begin{array}{cc}
             a-K_{11} &  0 \\
             -K_{21}  & a-K_{22}
             \end{array}
             \right],\
             \tilde{K}=\left[
             \begin{array}{cc}
             K_{11} & 0  \\
             K_{21} & K_{22}
             \end{array}
             \right].
     \end{split}\label{eq:coupled_closed_A_matrix_oneway}
     \end{flalign}}
Then, when $K_{11}, K_{21}$, and $K_{22}$ are chosen such that $\tilde{A}$ is Hurwitz and when $m$ is constant, the steady-state value of \eqref{eq:coupled_obs_error_matrix_oneway} is given by
\IfJtwoC{
     \begin{flalign}
     \begin{split}
             \!\!\!e_1^\star \!\!=\!\!\frac{K_{11}}{K_{11}\!-\!a}m_1,
             e_2^\star \!\!=\!\!\frac{-a K_{21}}{(K_{11}\!\!-\!\!a)(K_{22}\!\!-\!\!a)}m_1 \!\!+\!\! \frac{K_{22}}{K_{22}\!\!-\!\!a}m_2.
     \end{split}
     \end{flalign}
}{
     \begin{flalign}
     \begin{split}
             e_1^\star &=\frac{K_{11}}{K_{11}-a}m_1,\quad 
             e_2^\star =\frac{-a K_{21}}{(K_{11}-a)(K_{22}-a)}m_1 + \frac{K_{22}}{K_{22}-a}m_2.
     \end{split}
     \end{flalign} } 
Furthermore, the local estimation error resulting from each agent  is given by the quantity
$\bar{e}_i:=\bar{x}_i - x,\ i\in \{1,2\}$,
and has a steady-state value given by
\IfJtwoC{
     \begin{align*}
         \bar{e}_1^\star \!=\! e_1^\star, \ 
         \bar{e}_2^\star \!=\! \frac{K_{11}(K_{22} \!-\! a) \!-\! aK_{21}}{2(K_{11} \!-\! a)(K_{22} \!-\! a)}m_1 \!+\! \frac{K_{22}}{2(K_{22} \!-\! a)}m_2.
     \end{align*}
}{
     \begin{align*}
         \bar{e}_1^\star = e_1^\star, \quad
         \bar{e}_2^\star  = \frac{K_{11}(K_{22} \!-\! a) \!-\! aK_{21}}{2(K_{11} \!-\! a)(K_{22} \!-\! a)}m_1 \!+\! \frac{K_{22}}{2(K_{22} \!-\! a)}m_2.
     \end{align*} }
Let $K_{11} = K_{22} = K_L$. Because of the structure of $\tilde{A}$, it can be verified that the rate of convergence for the estimation error \eqref{eq:coupled_obs_error_matrix_oneway} is $a-K_L$, which is the same as that of the Luenberger observer \eqref{eg:scalarplantob}.  Moreover, 
assuming that constant noise $m_1$ and $m_2$ are equal, i.e., $m_1 = m_2 = m_0$, then 
\IfJtwoC{
$\bar{e}_2^\star  = \frac{2K_L(K_L-a) - aK_{21}}{2(K_L - a)^2}  m_0.$
}{
     \begin{align*}
     \begin{split}
           \bar{e}_2^\star  = \frac{2K_L(K_L-a) - aK_{21}}{2(K_L - a)^2}  m_0.
     \end{split}
     \end{align*} }
Interestingly, picking $K_{21} = \frac{2K_L(K_L-a)}{a} $, we obtain $\bar{e}_2^\star = 0$ for any unknown constant $m_0$, namely, the measurement noise can be completely rejected. When constant noise $m_1$ and $m_2$ are not equal, the choice  $K_{21} = \frac{K_L(K_L-a)}{a}$ leads to $\bar{e}_2^\star =  \frac{K_L}{2(K_L-a)}m_2$, 
 which is a significant improvement ($50\%$) over the case that agent $2$ only has access to its own measurement (in which case $\bar e_2^\star = \frac{K_L}{K_L - a}m_2$). These properties cannot be achieved by using the Luenberger observer in \eqref{eg:scalarplantob}. 

For general measurement noises $m_1$ and $m_2$ (not necessarily constant), the $H_\infty$ norm\footnote{By ``$H_\infty$ norm'' we mean the ${\cal L}_2$ gain from $m$ to $e$, which is the induced $2$-norm of the complex matrix transfer function from $m$ to $e$. } from noise to the estimation error can be employed to study the noise effect. \IfJtwoC{In fact, when $K_{21} \approx -4.75$, }{As shown in Figure~\ref{fig:hinfnorm_k21_local}, when $K_{21} \approx -4.75$,} the $H_\infty$ gain from noise $m$ to the local estimate $\bar{e}_2$ achieves a minimum equal to $0.45$, which is much smaller than that of the Luenberger observer in \eqref{eg:scalarplantob}, which is $0.8$, with equal rate of convergence ($K_L = 2,a=-0.5$).
\IfJtwoC{
}{
\begin{figure}[!h]
         \centering
         \subfigure[$H_\infty$ norm from noise $m$ to estimation error ${e}_L$ with respect to the parameter $K_L$.]{
         \label{fig:tradeoff}
         \psfragfig[width=0.45\textwidth]{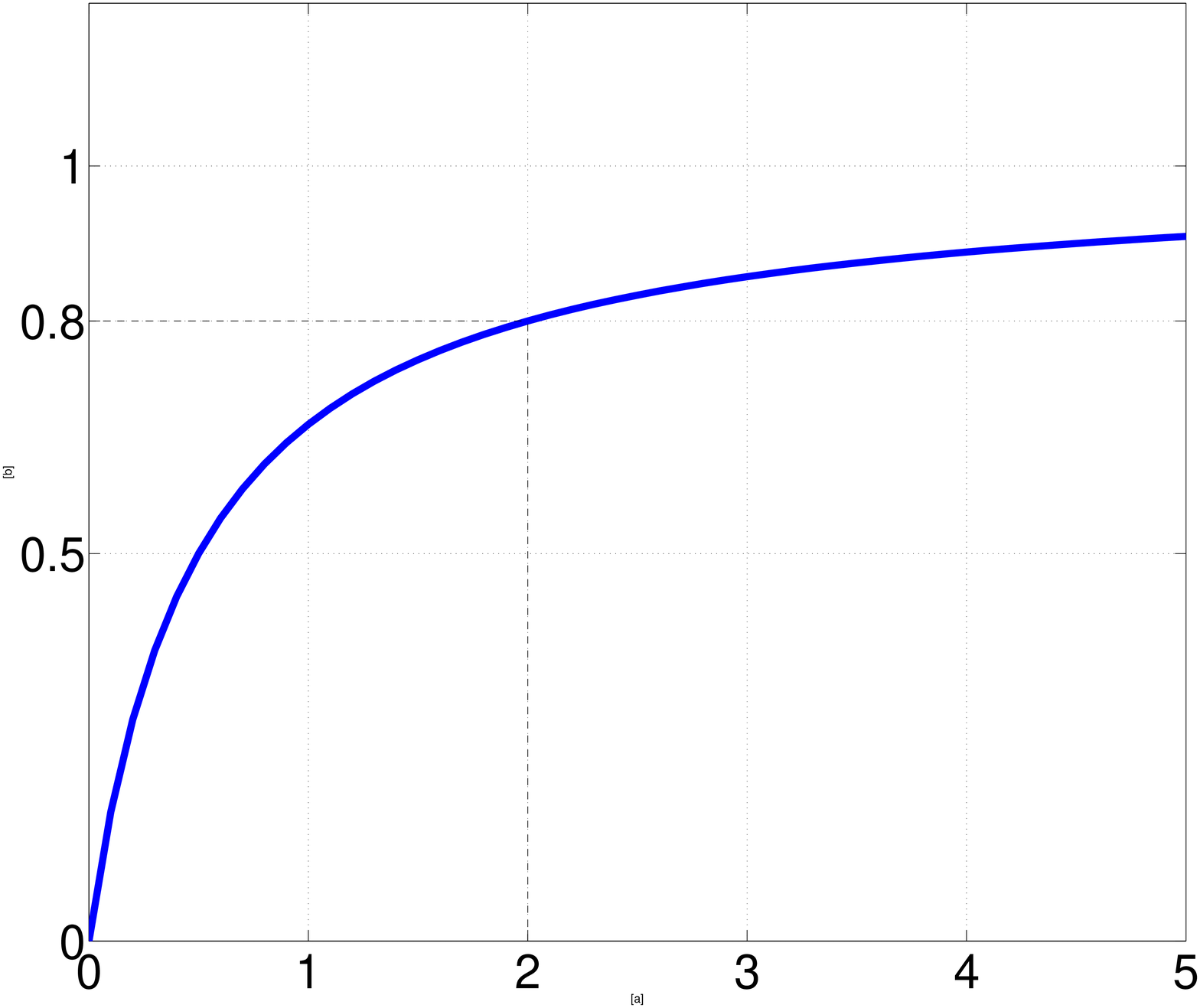}
         {
         \psfrag{[a]}[][][0.8]{\hspace{0.2in}$K_L$}
         \psfrag{[b]}[][][0.8][-90]{\hspace{0.0in}$H_\infty$}
         }
         }
         \quad
         \subfigure[The local $H_\infty$ norm from noise $m$ to estimation error $\bar{e}_2$ with respect to the parameter $K_{21}$.]{
         \label{fig:hinfnorm_k21_local}
         \psfragfig[width=0.45\textwidth]{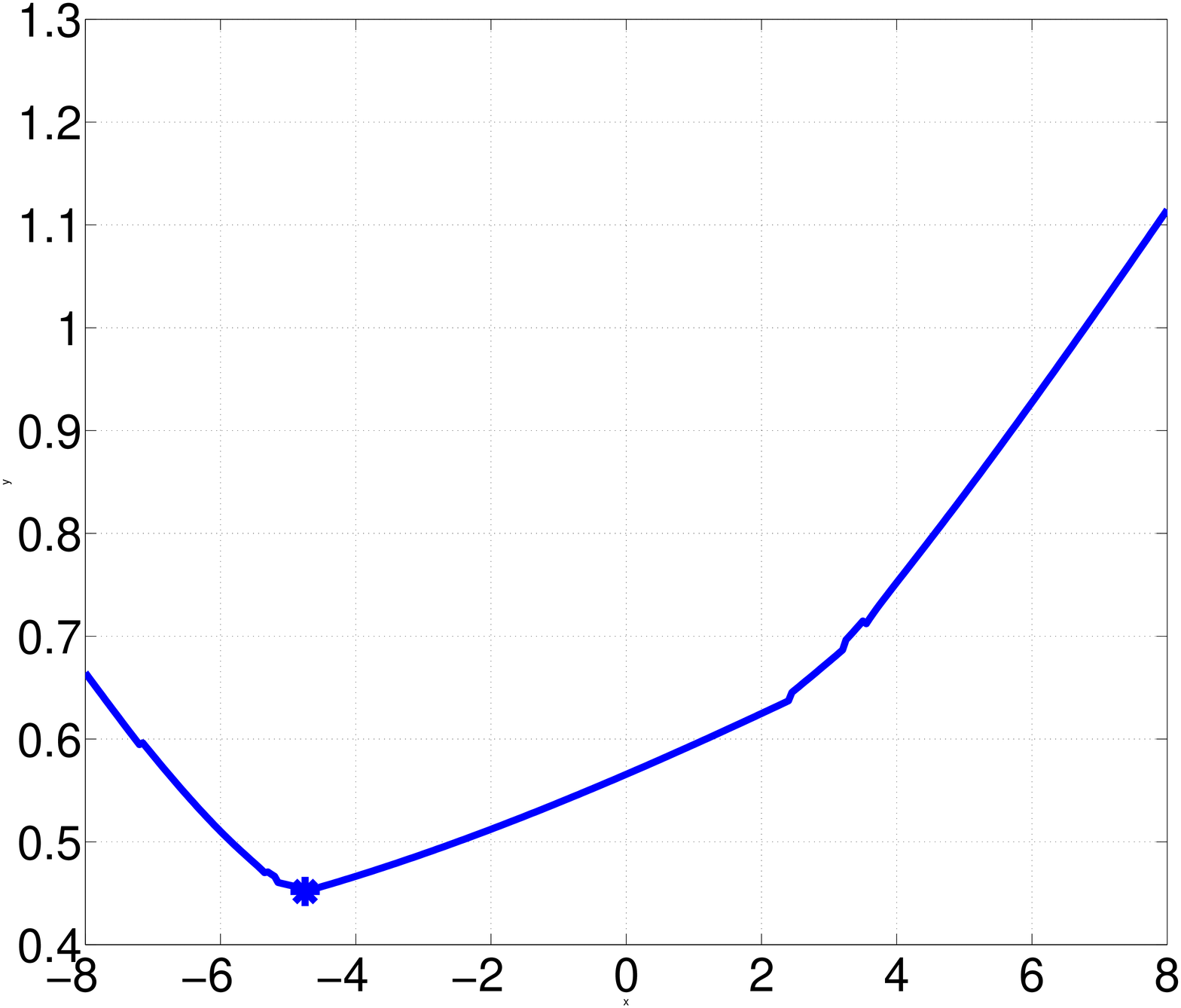}
         {
         \psfrag{x}[][][0.8]{\hspace{0.4in}$K_{21}$}
         \psfrag{y}[][][0.8][-90]{\hspace{0.0in}$H_\infty$}
         }
         }         
         \caption{Comparison between the $H_\infty$ norms for the proposed observer and the Luenberger observer, with fixed parameters $K_{11}= K_{22}=2$ and $a =-0.5$ (improved by approximately $43.8\%$).}
         \label{fig:hinfnorm_k21_global_local}
\end{figure}    }

It is important to point out that the observer proposed in this paper will also outperform the Luenberger observer in \eqref{eg:scalarplantob} when, in addition, agent 2 can transmit information to agent 1, i.e., the link between the two agents is bidirectional. Such an improvement is unique for the following two reasons. When the two agents are connected by a bidirectional link, our observer can be considered to be a bank of two observers providing a global estimate that averages the estimate of each individual observer.  When the innovation terms ``$-K_{21}(\hat y_1 - y_1)$'' and ``$-K_{12}(\hat y_2 - y_2)$'' are missing, it can be shown that the effect of noise in the global estimate cannot be reduced -- bank of observers currently available in the literature suffer from this shortcoming \IfJournal{{(see \cite[Appendix D]{Li.Sanfelice.13.TR.Interconnected}}}{(see Appendix~\ref{app:argumenton2uncoupledobs}} for a proof of this claim).  This suggests that the innovation terms in our interconnected observer are key.  The second reason stems from the fact that our observer can be viewed as an ``augmented-dimension observer'' since, in general, it would have dimension $Nn$ for a plant of dimension $n$.  This property would contradict the well-known fact that an observer in the form \eqref{eg:scalarplantob} (or, in general, of the form \eqref{eq:singleob}) minimizing the mean square estimation error under perturbations has necessarily the same dimension as the plant (see, e.g., \cite[Section 4.2, Definition 4.3, and Theorem 4.5]{72tradeoff} and  \IfJtwoC{{\color{black}\cite[Section~\uppercase\expandafter{\romannumeral4}.C]{Li.Sanfelice.13.TR.Interconnected}}}{Section~\ref{subsec:comparison_obs}}).  However, when performance specifications (relative to the optimal observer) are added, which, in this paper, are formulated in terms of eigenvalue constraints, an $n$-dimensional observer may not be optimal.  The augmented dimension (larger than the plant) is the key feature that enables our observer to outperform observers of the form \eqref{eg:scalarplantob}, in particular, by mitigating the typical amplification of noise due to large gain required to speed up convergence.

As we show next, the idea behind the proposed interconnected observer illustrated in the example above generalizes to the case where $N$ agents can measure the plant's output and share information over a graph.
\IfJtwoC{\vspace{-9pt}}{}
\section{Interconnected Observers}
\label{sec:coupledpairofLobserver}
\subsection{Notation and basic definitions}
Given a matrix $A$ with Jordan form $A = X J X^{-1}$,
$\alpha(A) := \max\{Re(\lambda):\ \lambda\in \mbox{eig}(A)\}$, where $\mbox{eig}(A)$ denotes the eigenspace of $A$;
$\mu(A): = \max\{Re(\lambda)/2:\ \lambda\in \mbox{eig}(A+A^\top)\}$;
$|A| := \max\{|\lambda|^\frac{1}{2}:\ \lambda\in \mbox{eig}(A^\top A)\}$;
$\kappa(A):=\min\{|X| |X^{-1}|:\ A = X J X^{-1} \}$;
$A$ is dissipative if $A+A^\top<0$.
Given a vector $u\in \mathbb{R}^n$, $|u| := \sqrt{u^\top u}$. 
\IfJtwoC{}{Given a Lebesgue measurable function $t\mapsto G(t)$, the norm $||G||_1$ is defined by $||G||_1 : = \int_0^\infty || G(t) || dt$, where $|| G(t) || = \sup \{ |G(t)u |: u \in \mathbb{R}^n \,  \textrm{and}\, | u | \leq 1 \}$ for all $t\geq 0$. }
Given a function $m: \mathbb{R}_{\geq 0 } \to \mathbb{R}^n$, $ |m|_\infty : = \sup_{t \geq 0} | m(t) |$.
\IfJtwoC{}{Given a function $\nu: \mathbb{R}_{\geq 0 } \to \mathbb{R}$, $ D^+ \nu(t) : = \lim \sup_{h \to 0^+} \frac{\nu(t+h)-\nu(t)}{h}$.}
The set of complex numbers is denoted by $\mathbb{C}$. 
The set of natural numbers is denoted by $\mathbb{N}:= \{1,2,3, \cdots\}$. 
Given a symmetric matrix $P$, $\lambda_{\max}(P) : =\max\{\lambda : \ \lambda\in \mbox{eig}(P) \} $ and $\lambda_{\min}(P) : =\min\{\lambda : \ \lambda\in \mbox{eig}(P) \} $.
For a continuous transfer function $\mathbb{C}\ni s \mapsto T(s)\in \mathbb{C}$, the $H_\infty$ norm is defined as $||T||_\infty = \sup_{\omega\in\mathbb{R}} ||T(j\omega)||$, $T$ is called stable if all its poles have negative real part, the dominant pole for a stable transfer function is the pole with largest real part, the rate of convergence of a closed-loop system with stable transfer function is defined by the absolute value of real part of the dominant pole. 
\IfJtwoC{}{Given a function $f:[0,\infty)\to \mathbb{R}$, $f$ is square integrable if $\int_{0}^{\infty} f(t)dt$ exists and is finite. 
For a continuous differentiable function $\mathbb{R}\ni x \mapsto f(x)\in \mathbb{R}$, $f$ is pseudo-convex if for any $x_1,x_2\in\mathbb{R}$ such that $\nabla f(x_1)^\top (x_2-x_1)\geq0$, then $f(x_2)\geq f(x_1)$; furthermore, $f$ is pseudo-concave if $-f$ is pseudo-convex. } 
Given matrices $A, B$ with proper dimensions, we define the operator $\mbox{He}(A,B): = A^\top B + B^\top A$;
$A\otimes B$ defines the Kronecker product; 
and $A*B$ defines the Khatri-Rao product.
\IfJtwoC{}{ Given $N\in\mathbb{N}$, $I_N\in \mathbb{R}^{N\times N}$ defines the identity matrix, 
$1_N$ is the vector of $N$ ones,  
 and  $\Pi_N := I_N - \frac{1}{N}1_N 1_N^\top$. }
Given a set $S$, the function $\mbox{card}(S)$ defines the cardinality of the set $S$.
A function $\alpha:\reals_{\geq 0} \to \reals_{\geq 0}$ is a class-${\cal K}_\infty$ function, also written $\alpha\in {\cal K}_\infty$, if $\alpha$ is zero at zero, continuous, strictly increasing, and unbounded.
A function $\beta:\reals_{\geq 0}\times \reals_{\geq 0} \to \reals_{\geq 0}$ is a class-${\cal KL}$ function, also written $\beta \in {\cal KL}$, if it is nondecreasing in its first argument, nonincreasing in its second argument, $\lim_{r\to 0^+} \beta(r,s) = 0$ for each $s\in\reals_{\geq 0}$, and $\lim_{s\to \infty}\beta(r,s) = 0$ for each $r\in\reals_{\geq 0}$.
\IfJtwoC{\vspace{-9pt}}{}
\subsection{Preliminaries on graph theory}
A directed graph (digraph) is defined as $\Gamma = ({\cal V}, {\cal E}, G)$. The set of nodes of the digraph are indexed by the elements of ${\cal V} = \{1,2,\dots, N\}$, and the edges are the pairs in the set ${\cal E} \subset {\cal V} \times {\cal V}$. Each edge directly links two nodes, i.e., an edge from $i$ to $j$, denoted by $(i,j)$, implies that agent $i$ can send information to agent $j$. The adjacency matrix of the digraph $\Gamma$ is denoted by $G = (g_{ij})\in \mathbb{R}^{N\times N}$, where $g_{ij} =1$ if $(i,j)\in {\cal E}$, and $g_{ij}=0$ otherwise. A digraph is undirected if $g_{ij}=g_{ji}$ for all $i,j\in {\cal V}$. The in-degree and out-degree of agent $i$ are defined by $d^{in}(i) = \sum_{j=1}^N g_{ji}$ and $d^{out}(i) = \sum_{j=1}^N g_{ij}$. 
\IfJtwoC{}{A digraph is weight-balanced if, for each node $i\in{\cal V}$, the in-degree equals its out-degree. }
The  in-degree matrix $D$ is the diagonal matrix with entries $D_{ii}=d^{in}(i)$, for all $i\in{\cal V}$. 
\IfJtwoC{}{The Laplacian matrix of the graph $\Gamma$, denoted by ${\cal L}$, is defined as ${\cal L}=D - G$. The Laplacian has the property that ${\cal L}1_N = 0$. 
Therefore, $0$ is an eigenvalue of the matrix ${\cal L}$. Furthermore, the rest of the eigenvalues of ${\cal L}$ have nonnegative real parts.}
\IfJtwoC{}{Denote by ${\cal M}_i$ the set containing all edges that connected to the $i$-th agent, i.e., ${\cal M}_i:= \{(j,i): j\in{\cal V}, (j,i)
\in {\cal E}\}$.}
The set of indices corresponding to the neighbors that can send information to the $i$-th agent is denoted by ${\cal I}(i):=\{j\in{\cal V}: (j,i)
\in {\cal E} \}$.
\IfJtwoC{\vspace{-15pt}}{}
\subsection{Observer structure and basic properties}
The general form of the proposed observer consists of $N$ {\em interconnected observers} with output given by the average over a graph of the states of the individual observers.\footnote{More general linear combinations defining $\bar{x}_i$ are possible, {\it i.e.}, $\bar{x}_i = \sum_{j \in {\cal I}(i)} \eta_j \hat x_j$ with $\eta_j \in \mathbb{R}$ for all $j$ and $\sum_{j \in {\cal I}(i)}{\eta_j} =1$.} 
Specifically, consider a network of $N$ agents defined by a graph $\Gamma=({\cal V},{\cal E},G)$. 
For the estimation of the plant's state, a local state observer using information from its neighbors is attached to each agent. More precisely, for each $i\in{\cal V}$, the agent $i$ runs a local state observer given by 
\IfJtwoC{
      \begin{align}  
      \begin{split}
      \dot {\hat x}_i &= A {\hat x}_i - \sum_{j \in {\cal I}(i)}K_{ij}( \hat y_j -  y_j),\\[-5pt]
      \hat y_i & = C \hat x_i,\quad
      \bar{x}_i  = \frac{1}{\mbox{card}({\cal I}(i))} \sum_{j \in {\cal I}(i)} \hat x_j,
      \end{split}
      \label{eq:graph_individual}
      \end{align}
}{
      \begin{align}  
      \dot {\hat x}_i &= A {\hat x}_i - \sum_{j \in {\cal I}(i)}K_{ij}( \hat y_j -  y_j),\quad 
      \hat y_i  = C \hat x_i,\quad
      \bar{x}_i  = \frac{1}{\mbox{card}({\cal I}(i))} \sum_{j \in {\cal I}(i)} \hat x_j, \label{eq:graph_individual}
      \end{align} }
where $\hat x_i $ denotes the state variable of the observer, $\bar x_i$ is the local estimate of the plant's state $x$, and $y_i$ denotes the measurement of $y$ in \eqref{eq:plant} taken by the $i$-th agent under measurement noise $m_i$, that is 
$
y_i  = Cx + m_i.
$ 
The information that the $i$-th agent obtains from its neighbors are the values of $\hat x_j$'s and $y_j$'s for each $j\in {\cal I}(i)$. 
The collection of local state observers in \eqref{eq:graph_individual} connected via the graph $\Gamma$ defines the proposed {\em interconnected observer}.

 To analyze the properties of interconnected observers, define for each  $i \in {\cal V}$, $e_i: = \hat x_i - x$ and the associated vector $e: = (e_1, \dots, e_N)$. Furthermore, define the local estimation error $\bar e_i : = \bar x_i - x$,  the global estimation error vector $\bar e: = (\bar e_1, \dots, \bar e_N)$, and the noise vector $m:= (m_1,  \dots, m_N)$. 
Then, it follows that 
\IfJtwoC{
\begin{align}
\begin{split}
\dot e_i &= A  e_i - \sum_{j \in {\cal I}(i)}K_{ij}C e_j + \sum_{j \in {\cal I}(i)}K_{ij}m_j,\\[-4pt]
\bar{e}_i &= \frac{1}{\mbox{card}({\cal I}(i))} \sum_{j \in {\cal I}(i)} e_j, 
\end{split}
\label{eq:error_individual}
\end{align}
}{
\begin{align}
\dot e_i &= A  e_i - \sum_{j \in {\cal I}(i)}K_{ij}C e_j + \sum_{j \in {\cal I}(i)}K_{ij}m_j,\quad 
\bar{e}_i  = \frac{1}{\mbox{card}({\cal I}(i))} \sum_{j \in {\cal I}(i)} e_j, 
\label{eq:error_individual}
\end{align} }
which can be rewritten in the compact form
\IfJtwoC{
\begin{align}
\begin{split}
\!\!\! \dot e   & = (I_N \!\otimes\! A \!-\! ({\cal K}\!*\!G^\top)(I_N \!\otimes\! C) ) e \!+\! ({\cal K}\!*\!G^\top) m,\\
\!\!\! \bar{e}  &= (D^{-1}\!\otimes\! I_n)(G^\top \!\otimes\! I_n) e,
\end{split}
\label{eq:general_error_compact_graph}
\end{align}
}{
\begin{align}
\dot e   & = (I_N \!\otimes\! A \!-\! ({\cal K}\!*\!G^\top)(I_N \!\otimes\! C) ) e \!+\! ({\cal K}\!*\!G^\top) m,\quad
\bar{e}     = (D^{-1}\!\otimes\! I_n)(G^\top \!\otimes\! I_n) e,
\label{eq:general_error_compact_graph}
\end{align} }
where $G$ is the adjacency matrix, $D$ is the in-degree matrix, 
\begin{align}
{\cal K} = 
\left[\begin{array}{cccc}
K_{11}  & K_{12} & \cdots   & K_{1N}\\
K_{21}  & K_{22} & \cdots   & K_{2N}\\
\vdots   & \vdots   & \ddots   & \vdots\\
K_{N1} & K_{N2} & \cdots   & K_{NN}
\end{array}\right],
\end{align}
and the Khatri-Rao product ${\cal K}*G^\top$ is such that ${\cal K}$ is treated as $N\times N$ block matrices with $K_{ij}$'s as blocks. Define 
\IfJtwoC{
\begin{align}
\begin{split}
{\cal A}&:= I_N \otimes A - ({\cal K}*G^\top)(I_N \otimes C),\\
{\cal B}&:= {\cal K}*G^\top, {\cal C}:=  (D^{-1}\!\otimes\! I_n)(G^\top \otimes I_n).
\end{split}
\label{eq:matrix_graph}
\end{align}
}{
\begin{align}
\begin{split}
{\cal A}&:= I_N \otimes A - ({\cal K}*G^\top)(I_N \otimes C),\quad 
{\cal B}:= {\cal K}*G^\top,  \quad {\cal C}:=  (D^{-1}\!\otimes\! I_n)(G^\top \otimes I_n).
\end{split}
\label{eq:matrix_graph}
\end{align} }
Then, the transfer function from measurement noise $m$ to error $\bar{e}$ is given by $T(s) = {\cal C} (sI - {\cal A})^{-1} {\cal B}$. For the purpose of designing the proposed interconnected observer, each agent is self-connected, i.e., $(i,i)\in {\cal E}$. Therefore, we have $\mbox{\rm tr}(D) \geq N$.
\begin{remark}
The matrix $I_N \otimes A$ is a block diagonal matrix with matrix $A$ in each of the $N$ diagonal blocks (of dimension $n \times n$).  The matrix ${\cal K}*G^\top$ defines the gain matrix for the graph, while $ (D^{-1} \otimes I_n)(G^\top \otimes I_n)$ generates the estimation matrix for each agent by averaging the local estimates\IfJtwoC{.}{ obtained from its neighbors.} 
\end{remark}
It can be verified that, under a detectability condition, interconnected observers can be designed so that the origin of the error system in \eqref{eq:general_error_compact_graph} is (exponentially) stable.

\begin{proposition}
For the plant \eqref{eq:plant} with measurement noise $m_i\equiv 0$ for each agent $i$, if the pair $(A,C)$ is detectable, then, for any $N \in \mathbb{N}$, there exists a digraph $\Gamma$ with adjacency matrix $G$ and a gain ${\cal K}$ such that the matrix ${\cal A}$ is Hurwitz and the resulting system \eqref{eq:general_error_compact_graph} has its origin exponentially stable.
\end{proposition} 
\begin{proof}
For any $N\in\mathbb{N}$, consider $G = I_N$. Then it follows that $\dot e_i = (A-K_{ii}C)e_i$ for each $i\in{\cal V}$. Under the assumption that the pair $(A,C)$ is detectable, immediately we know that, for each $i\in{\cal V}$, there exists $K_{ii}$ such that $A-K_{ii}C$ is Hurwitz. Therefore, the resulting ${\cal A}$ is Hurwitz. 
\end{proof}
\IfJtwoC{\vspace{-14pt}}{}
\subsection{$\cal KL$ characterization of performance and robustness}
In this section, the performance and robustness properties of observers are characterized in terms of ${\cal KL}$ bounds. More precisely, given an observer with estimation error $e$, we are interested in bounds of the form
      \begin{flalign*}
      \begin{split}
          |e(t)|\leq \beta(|e(0)|,t)+ \varphi(|m|_\infty) \qquad \forall t \geq 0,
      \end{split}
      \end{flalign*}
where $t\mapsto e(t)$ is a solution to the error system, $\beta$ is a class-${\cal KL}$ function, and $\varphi$ is a class-${\cal K}_\infty$ function.
To establish and compare this property with that of the interconnected observers, the next result characterizes such bounds for the proposed observer so that it can be designed to 
outperform those due to a Luenberger observers.
\begin{proposition}
\label{prop:issbound}
For the plant \eqref{eq:plant}, assume the pair $(A,C)$ is detectable. Let  $N\in\mathbb{N}$ and a digraph $\Gamma=({\cal V},{\cal E},G)$ be given. If there exists a gain ${\cal K}$ such that at least one of the following conditions are satisfied:
\begin{enumerate}
\item [1)] The matrix ${\cal A}$ is Hurwitz with distinct eigenvalues; 
\item [2)] The matrix ${\cal A}$ is dissipative, i.e.,  for some $\bar{\alpha} > 0$, ${\cal A}^\top + {\cal A} \leq - 2 \bar{\alpha} I$;
\item [3)] There exists $P=P^\top>0$ such that $\mbox{\rm He}({\cal A}, P) \leq - 2 \bar{\alpha} P$ for some $\bar{\alpha} > 0$;
\end{enumerate}
then, there exist a class-${\cal KL}$ function $\beta: \mathbb{R}_{\geq 0}\times \mathbb{R}_{\geq 0} \to \mathbb{R}_{\geq 0}$ and a class-${\cal K}$ function $\varphi: \mathbb{R}_{\geq 0} \to \mathbb{R}_{\geq 0}$ such that the solution $\bar{e}$ of \eqref{eq:general_error_compact_graph} from any $e(0)\in\mathbb{R}^{nN}$ satisfies
\begin{align} 
| \overline{e}(t) | \leq \beta(|{e}(0)|,t) + \varphi (|m|_\infty) \quad \forall t\in \mathbb{R}_{\geq 0}.
\label{neq:KLbounds_graph}
\end{align}
In particular, the functions $\beta$ and $\varphi$ can be chosen, for all $s,t\geq 0$, as follows: if $1)$ holds, then,
$         \beta (s, t) \!=\! \kappa(\!{\cal A}) |{\cal C}| \! \exp( {\alpha}({\cal A}) t) s$, 
$         \varphi (s) \!=\!  \kappa({\cal A}) \! \frac{|{\cal B}| |{\cal C}|}{| {\alpha} ({\cal A}) |}\!s$; 
 if $2)$ holds, then, 
$         \beta (s, t)=  |{\cal C}| \exp( \mu({\cal A}) t) s$,
$         \varphi(s) = \frac{|{\cal B}| |{\cal C}| }{| \mu ({\cal A}) |}s$;
 if $3)$ holds, then,
$         \beta (s, t) = \sqrt{c_p}|{\cal C}|\exp(-\lambda t) s$,
$         \varphi(s) =  c_p \frac{|{\cal B}||{\cal C}|}{|\lambda|} s, $
with $\lambda = \frac{ \bar{\alpha} \lambda_{\min} (P)}{\lambda_{\max}(P)}$ and $c_p= \frac{\lambda_{\max}(P)}{\lambda_{\min}(P)}$.           
\end{proposition}
\begin{proof}
The proof can be found in \IfJtwoC{{\color{black}\cite[Appendix A]{Li.Sanfelice.13.TR.Interconnected}}}{Appendix~\ref{app:proof_prop_KL_bounds}}.
\end{proof}

Proposition~\ref{prop:issbound} provides a way to design parameters for the interconnected observer as follows. Recall that $\tilde A_L$ and $K_L$ are defined in \eqref{eq:singleob_error}. 
\begin{theorem}
\label{thm:basedonprops}
For the plant \eqref{eq:plant} with the Luenberger observer \eqref{eq:singleob} and the interconnected observers \eqref{eq:graph_individual}, let $N\in\mathbb{N}$ and a digraph $\Gamma$ be given.  If $K_L$ is such that at least one of the following conditions are satisfied: 
\begin{itemize}
\item [1)] $\tilde A_L$ is Hurwitz with distinct eigenvalues, and there exists ${\cal K}$ such that 
$\alpha({\cal A}) < \alpha(\tilde A_L)$ 
and $\kappa({\cal A})\frac{|{\cal B}| |{\cal C}|}{|\alpha({\cal A})|} < \kappa({\tilde A}_L)\frac{|K_L|}{|\alpha({\tilde A}_L)|}$; 
\item [2)] $\tilde A_L$ is dissipative, and there exists ${\cal K}$ such that $\mu({\cal A}) < \mu(\tilde A_L)$ (or $\alpha({\cal A}) < \alpha(\tilde A_L)$, respectively -- see below $c)$) and $\frac{|{\cal B}| |{\cal C}|}{|\mu({\cal A})|} < \frac{|K_L|}{|\mu({\tilde A}_L)|}$; 
\item [3)] $\tilde A_L$ satisfies $\He({\tilde A}_L, P_L)\leq -2\overline\alpha_L P_L$ for some $\overline \alpha_L >0$ and $P_L = P_L^\top>0$, and there exists ${\cal K}$ such that 
\begin{enumerate}
\item [3.1)] item $3)$ of Proposition~\ref{prop:issbound} holds with $\overline \alpha>0, P=P^\top>0$,
\item [3.2)] $\lambda:=\frac{\overline \alpha \lambda_{\min}(P)}{\lambda_{\max}(P)} < \frac{\overline \alpha_L \lambda_{\min}(P_L)}{\lambda_{\max}(P_L)}=:\lambda_L$ and $c_p \frac{|{\cal B}| |{\cal C}|}{|\lambda|} < c_{pL} \frac{|K_L|}{|\lambda_L|}$, with $c_p =  \frac{\lambda_{\max}(P)}{\lambda_{\min}(P)}$ and $c_{pL} = \frac{\lambda_{\max}(P_L)}{\lambda_{\min}(P_L)}$;  
\end{enumerate}
\end{itemize}
then, there exist $\beta\in {\cal KL}$ and $\varphi \in {\cal K}_\infty$ such that the solution $\bar{e}$ of \eqref{eq:general_error_compact_graph} from any $e(0)\in\mathbb{R}^{nN}$ satisfies
\begin{enumerate}
  \item [a)] $|\bar{e}(t)| \leq {\beta}(|e(0)|,t) + {\varphi}(|m|_\infty)$ for all $t \geq 0$; 
  \item [b)] Given nonzero $e(0)$ and $e_L(0)$, $\exists t^\star\geq 0$ such that $ {\beta}(|e(0)|,t) < \beta_L(|e_L(0)|,t)$ for all $t > t^\star$;
  \item [c)] ${\varphi}(s) < \varphi_L(s)$, for all $s \neq 0$, $s\in \mathbb{R}_{\geq 0}$.
\end{enumerate}
In particular, the functions $\beta\in {\cal KL}$ and $\varphi \in {\cal K}_\infty$ can be chosen accordingly as in Proposition~\ref{prop:issbound} while $\beta_L\in {\cal KL}$ and $\varphi_L \in {\cal K}_\infty$ can be  chosen, for all $s,t\geq 0$, as follows: if $1)$ holds, then 
$         \beta_L (s, t) \!=\! \kappa(\!{\tilde A}_L) \! \exp( {\alpha}({\tilde A}_L) t) s$, 
$         \varphi_L (s) \!=\!  \kappa({\tilde A}_L) \! \frac{|{K_L}| }{| {\alpha} ({\tilde A}_L) |}\!s$; 
 if $2)$ holds, then 
$         \beta_L (s, t)= \exp( \mu({\tilde A}_L) t) s$ 
(or $\beta_L (s, t) \!=\! \kappa(\!{\tilde A}_L) \! \exp( {\alpha}({\tilde A}_L) t) s$, respectively),
$         \varphi_L(s) = \frac{|K_L| }{| \mu ({\tilde A}_L) |}s$;
 if $3)$ holds, then 
$         \beta_L (s, t) = \sqrt{c_{pL}} \exp(-\lambda_L t) s$,
$         \varphi_L(s) =  c_{pL} \frac{|{K_L}|}{|\lambda_L|} s$. 
\end{theorem}

\begin{proof}
The proof follows from Proposition~\ref{prop:issbound}. Note that the Luenberger observer is a special case of the interconnected observer with $N = 1$.
\end{proof}
\stopmodif
\IfJtwoC{\vspace{-9pt}}{}
\begin{remark}
\label{remark:specialcase}
\IfJtwoC{}{Assumption $2)$ of Proposition~\ref{prop:issbound} is a special case of assumption $3)$ with matrices $P=I$ and $P_0=I$.}
Note that the boundedness property in item $2)$ in Theorem~\ref{thm:basedonprops} guarantees that the rate of convergence of the interconnected observers is faster than or equal to that of a Luenberger observer by comparing the ${\cal KL}$ estimates they induce (which is a reasonable measure of performance when the ${\cal KL}$ functions are derived using similar bounding techniques).
\end{remark} 
The ${\cal KL}$ bounds established in Proposition~\ref{prop:issbound} characterize a worse case property of the estimation error of the proposed observer, which can be compared to that of a Luenberger observer via Theorem~\ref{thm:basedonprops}. The following example illustrates this point. 
\begin{example}
\label{subset:firstorderplant}
We revisit the motivational example in Section~\ref{sec:MotivationalExample} and design an interconnected observer with $N=2$ with an all-to-all graph as shown in Figure~\ref{fig:2nodes_alltoall}.
\IfJtwoC{
         \begin{figure}[!h]
         \centering
         \psfragfig[trim=0.1cm 0.7cm 0.5cm 0.1cm,clip,width=0.16\textwidth]{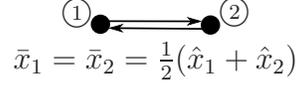}
         {
         \psfrag{[a]}[][][1]{}
         \psfrag{[b]}[][][1]{}
         \psfrag{[d]}[][][1]{\hspace{0.45in}$\bar x_1 = \bar x_2 = \frac{1}{2}(\hat x_1 + \hat x_2)$}
	}
         \put(-83,20){\textcircled{\scriptsize$1$}}
         \put(-16,20){\textcircled{\scriptsize$2$}}         
         \caption{Two agents connected as a direct graph.}
         \label{fig:2nodes_alltoall}
     \end{figure}
}{
         \begin{figure}[!h]
         \centering
         \psfragfig[trim=0.1cm 0.7cm 0.5cm 0.1cm,clip,width=0.16\textwidth]{J2agentslinkalltoall}
         {
         \psfrag{[a]}[][][1]{\textcircled{\scriptsize$1$}}
         \psfrag{[b]}[][][1]{\hspace{-0.26in}\textcircled{\scriptsize$2$}}
         \psfrag{[d]}[][][1]{\hspace{0.45in}$\bar x_1 = \bar x_2 = \frac{1}{2}(\hat x_1 + \hat x_2)$}
	}
         \caption{Two agents connected as a direct graph.}
         \label{fig:2nodes_alltoall}
     \end{figure}
}
Consider the case when two agents are experiencing common noises $m_1 = m_2 = m$.  The transfer functions 
from $m$ to $e_L$ and from $m$ to $\overline e$ (global) are given by\footnote{For the particular choice of parameters $K_{11}=K_{22}=K_L$ and $K_{12}=K_{21}=0$, $||T||_\infty = ||T_L||_\infty$.} $T_L(s) = \frac{K_L}{s-a+K_L}$ and $T(s) = {\cal C}(sI - {\cal A})^{-1} {\cal B}$. 
In particular, the proposed observer takes the form
\IfJtwoC{
     \begin{flalign}
     \begin{split}
             \dot{\hat x}_1 &=a{\hat x}_1 - K_{11}({\hat y}_1 - y) - K_{12} ({\hat y}_2 - y),  \\
             \dot{\hat x}_2 &=a{\hat x}_2 - K_{22}({\hat y}_2 - y) - K_{21} ({\hat y}_1 - y),  \\
             {\hat y}_1 &= {\hat x}_1, \quad
             {\hat y}_2 = {\hat x}_2, \quad
             {\bar x}_1 = {\bar x}_2  = \frac{{\hat x}_1+{\hat x}_2}{2}.
     \end{split}\label{eq:coupled_obs_full}
     \end{flalign}
}{
     \begin{flalign}
     \begin{split}
             \dot{\hat x}_1 &=a{\hat x}_1 - K_{11}({\hat y}_1 - y) - K_{12} ({\hat y}_2 - y),  \quad 
             \dot{\hat x}_2 =a{\hat x}_2 - K_{22}({\hat y}_2 - y) - K_{21} ({\hat y}_1 - y),  \\
             {\hat y}_1 &= {\hat x}_1, \quad
             {\hat y}_2 = {\hat x}_2, \quad
             {\bar x}_1 = {\bar x}_2  = \frac{{\hat x}_1+{\hat x}_2}{2}.
     \end{split}\label{eq:coupled_obs_full}
     \end{flalign} }
Then, we have the following result. 
\begin{proposition}
\label{prop:example_Hinf}
Given $a,K_L\in\mathbb{R}$ such that $a\neq 0$ and $a-K_L<0$, then there exist $K_{11},K_{22},K_{12},K_{21}\in \mathbb{R}$ such that the rate of convergence of the observer \eqref{eq:coupled_obs_full} is no smaller than that of the one induced by the Luenberger observer and the $H_\infty$ norm of $T$ is smaller than the $H_\infty$ norm of $T_L$, $i.e.$, $||T||_\infty < ||T_L||_\infty$. 
\end{proposition}
\begin{proof}
The proof can be found in \IfJournal{{\color{black}\cite[Appendix  B]{Li.Sanfelice.13.TR.Interconnected}}}{{Appendix~\ref{app:proof_prop_exmaple_Hinf}}}.
\end{proof}
It should be noted that averaging the estimates of two uncoupled single Luenberger observers (one at each agent) does not lead to both faster convergence rate and smaller steady state error \IfJournal{{\color{black}(see \cite[Appendix  D]{Li.Sanfelice.13.TR.Interconnected})}}{(see Appendix~\ref{app:argumenton2uncoupledobs})}.
\IfJtwoC{To perform a numerical comparison,}{The error dynamics of \eqref{eq:coupled_obs} can be written as 
\IfJtwoC{
     \begin{flalign}
     \begin{split}
             \dot{ e} &\!=\!\tilde{A}{e}+\tilde{K}m, \\ 
             \tilde{A}&\!=\!\left[\!\!
             \begin{array}{ll}
             a-K_{11} & -K_{12} \\
             -K_{21}  & a-K_{22}
             \end{array}
             \!\!\right],
             \tilde{K}\!=\!\left[\!\!
             \begin{array}{l}
             K_{11} \!+\! K_{12}  \\
             K_{21} \!+\! K_{22}
             \end{array}
             \!\!\right]\!\!.\!\! 
     \end{split}\label{eq:coupled_obs_error_matrix}
     \end{flalign}
}{
     \begin{flalign}
     \begin{split}
             \dot{ e} &=\tilde{A}{e}+\tilde{K}m, \ 
             \tilde{A}=\left[
             \begin{array}{ll}
             a-K_{11} & -K_{12} \\
             -K_{21}  & a-K_{22}
             \end{array}
             \right],\
             \tilde{K}=\left[
             \begin{array}{l}
             K_{11} + K_{12}  \\
             K_{21} + K_{22}
             \end{array}
             \right]. 
     \end{split}\label{eq:coupled_obs_error_matrix}
     \end{flalign} }
where $e=[e_1\ e_2]^\top$. 
The steady-state value of \eqref{eq:coupled_obs_error_matrix} is given by $e^\star = [e_1^\star\ e_2^\star]^\top$, where
\IfJtwoC{
     \begin{align*}
             e_1^\star &=\frac{K_{11} K_{22} - K_{12} K_{21} - K_{11} a - K_{12} a}{K_{11} K_{22} - K_{12} K_{21} - K_{22} a - K_{11} a+a^2}m,\\
             e_2^\star &=\frac{K_{11} K_{22} - K_{12} K_{21} - K_{22} a - K_{21} a}{K_{11} K_{22} - K_{12} K_{21} - K_{22} a - K_{11} a+a^2}m.
     \end{align*}
}{
     \begin{align*}
            \scalebox{0.9}{ $\displaystyle e_1^\star =\frac{K_{11} K_{22} - K_{12} K_{21} - K_{11} a - K_{12} a}{K_{11} K_{22} - K_{12} K_{21} - K_{22} a - K_{11} a+a^2}m,\quad
             e_2^\star =\frac{K_{11} K_{22} - K_{12} K_{21} - K_{22} a - K_{21} a}{K_{11} K_{22} - K_{12} K_{21} - K_{22} a - K_{11} a+a^2}m$.}
     \end{align*} }
The
estimation error of the proposed observer is given by
 the quantity
 \begin{equation}\label{eqn:averageE}
 \bar{e}_i:=\bar{x}_i - x = \frac{e_1 + e_2}{2}, \quad i\in\{1,2\},
 \end{equation}
and has a steady-state value given by
\IfJtwoC{
     \begin{flalign}
     \begin{split}
           \bar{e}_i^\star \! =\!
           \frac{K_{11}\! K_{22} \!\!-\!\! K_{12} K_{21} \!\!-\!\! (\!1\!/2)\!(\! K_{11} \!\!+\!\! K_{22} \!\!+\!\! K_{12} \!\!+\!\! K_{21} \!) a}{K_{11} K_{22} \!\!-\!\! K_{12} K_{21} \!\!-\!\! K_{22} a \!\!-\!\! K_{11} a \!\!+\!\! a^2}m.
     \end{split}
     \end{flalign}
}{
     \begin{flalign}
     \begin{split}
           \bar{e}_1^\star  =  \bar{e}_2^\star  =
           \frac{K_{11} K_{22} - K_{12} K_{21} - (1/2)(K_{11}+K_{22} + K_{12} + K_{21}) a}{K_{11} K_{22} - K_{12} K_{21} - K_{22} a - K_{11} a+a^2}m.
     \end{split}
     \end{flalign} }
Under the condition that the matrix $\tilde{A}$ is Hurwitz, its eigenvalues can be written
in the general form
 $\lambda_{1,2}=-\sigma\pm j\omega,$
 where $\sigma$ is positive
 and $\omega\in\mathbb{R}$.
 Then,  referring to Theorem~\ref{thm:basedonprops}, by comparing the ${\cal KL}$ bounds, the following
 conditions guarantee  a faster convergence rate of the proposed observer:
 \IfJtwoC{
 \begin{align}
 \begin{split}
        \!\! -\sigma \!=\! \frac{-(K_{11} \!-\! a) \!-\! (K_{22} \!-\! a)}{2}< a\!-\!K_L<0,\\
         \left((K_{11}-a)+(K_{22}-a)\right)^2<4\det{\tilde{A}},
 \end{split}
 \label{neq1a:example0}
 \end{align}
 }{
\begin{eqnarray}
        -\sigma = \frac{-(K_{11}-a)-(K_{22}-a)}{2}< a-K_L<0,\quad
        \left((K_{11}-a)+(K_{22}-a)\right)^2<4\det{\tilde{A}}\label{neq1a:example0},
\end{eqnarray} }
where $\det(\tilde{A})= K_{11} K_{22} - K_{12} K_{21} - K_{22} a - K_{11} a+a^2$.
Furthermore, in order to assure an improvement on
the effect of measurement noise, the steady-state values $e^\star$ and $e_L^\star$ should satisfy
$\left|{\bar{e}}_i^\star\right|< \left|e_L^\star\right|,$
which leads to the following condition:     
\IfJtwoC{
     \begin{flalign}
     \begin{split}
            \left|\! \frac{K_{11}\! K_{22} \!\!-\!\! K_{12} K_{21}  \!\!-\!\! (\!1/2)(K_{11} \!\!+\!\! K_{22} \!\!+\!\! K_{12} \!\!+\!\! K_{21}) a }{K_{11} K_{22} \!\!-\!\! K_{12} K_{21} \!\!-\!\! K_{22} a \!\!-\!\! K_{11} a \!\!+\!\! a^2}\right| \\
           <\left|\frac{K_L}{K_L-a}\right|.
     \end{split}\label{neq2:example0}
     \end{flalign}
}{
     \begin{flalign}
     \begin{split}
            \left|\frac{K_{11} K_{22} - K_{12} K_{21} - (1/2)(K_{11}+K_{22}+K_{12}+K_{21}) a }{K_{11} K_{22} - K_{12} K_{21} - K_{22} a - K_{11} a+a^2}\right| 
           <\left|\frac{K_L}{K_L-a}\right|.
     \end{split}\label{neq2:example0}
     \end{flalign} } 
Now, to perform a numerical comparison,} we consider the case where $a=-0.5$ and $m: \mathbb{R}_{\geq 0} \to \mathbb{R}$ is a continuous bounded function. A Luenberger observer is designed following \eqref{eg:scalarplantob} to achieve a convergence rate of $2.5$ and an $H_\infty$ gain from $m$ to $e_L$ equal to $0.8$, which leads to $K_L = 2$. For the interconnected observers  \eqref{eq:coupled_obs_full},  
using Theorem~\ref{thm:basedonprops}, conditions $2)$ can be rewritten as 
\IfJtwoC{
      \begin{flalign}
      \begin{split}
           \alpha({\cal A})  \leq a - K_L, \\
           \frac{\sqrt{2}}{2} \frac{\sqrt{(K_{11}+K_{12})^2+(K_{22}+K_{21})^2}}{|{\mu}({\cal A})|} < \left| \frac{a}{a-K_L} \right|.
      \end{split}\label{neq:sim_prop3}
      \end{flalign}
}{
      \begin{flalign}
      \begin{split}
           \alpha({\cal A})  \leq a - K_L,  \quad 
           \frac{\sqrt{2}}{2} \frac{\sqrt{(K_{11}+K_{12})^2+(K_{22}+K_{21})^2}}{|{\mu}({\cal A})|} < \left| \frac{a}{a-K_L} \right|.
      \end{split}\label{neq:sim_prop3}
      \end{flalign}     }
From solving \eqref{neq:sim_prop3}, we pick parameters $K_{11} = 1.7896,\ K_{22} = 2.2278,\ K_{12} = 0.0538,\ K_{21} = -1.1633$. It can be verified that the eigenvalues of ${\cal A}$ according to this set of parameters are $-2.5087\pm 0.1208i$. Moreover, $\mu({\cal A}) = -1.9123$. 

Now we perform simulations using these parameters and different measurement noises.
With initial conditions $x(0)=3$, $\overline{x}_1(0) = \overline{x}_2(0) = \overline{x}_L(0) = 5$, the first simulation is ran for measurement noise $m(t)\equiv 0$ and the resulting trajectories are shown in Figure~\ref{fig:m0_trajs}. This figure shows that the interconnected observers converge at a faster rate compared to the Luenberger observer. In fact, item $2)$ of Theorem~\ref{thm:basedonprops} holds with $t^\star \approx 6.7s$.   
\IfJtwoC{
    \begin{figure}[!h]
         \centering
         \subfigure[$m(t)\equiv 0$.]{
         \label{fig:m0_trajs}
         \psfragfig[trim=1.8cm 0.8cm 1.5cm 0.7cm,clip, width=0.21\textwidth]{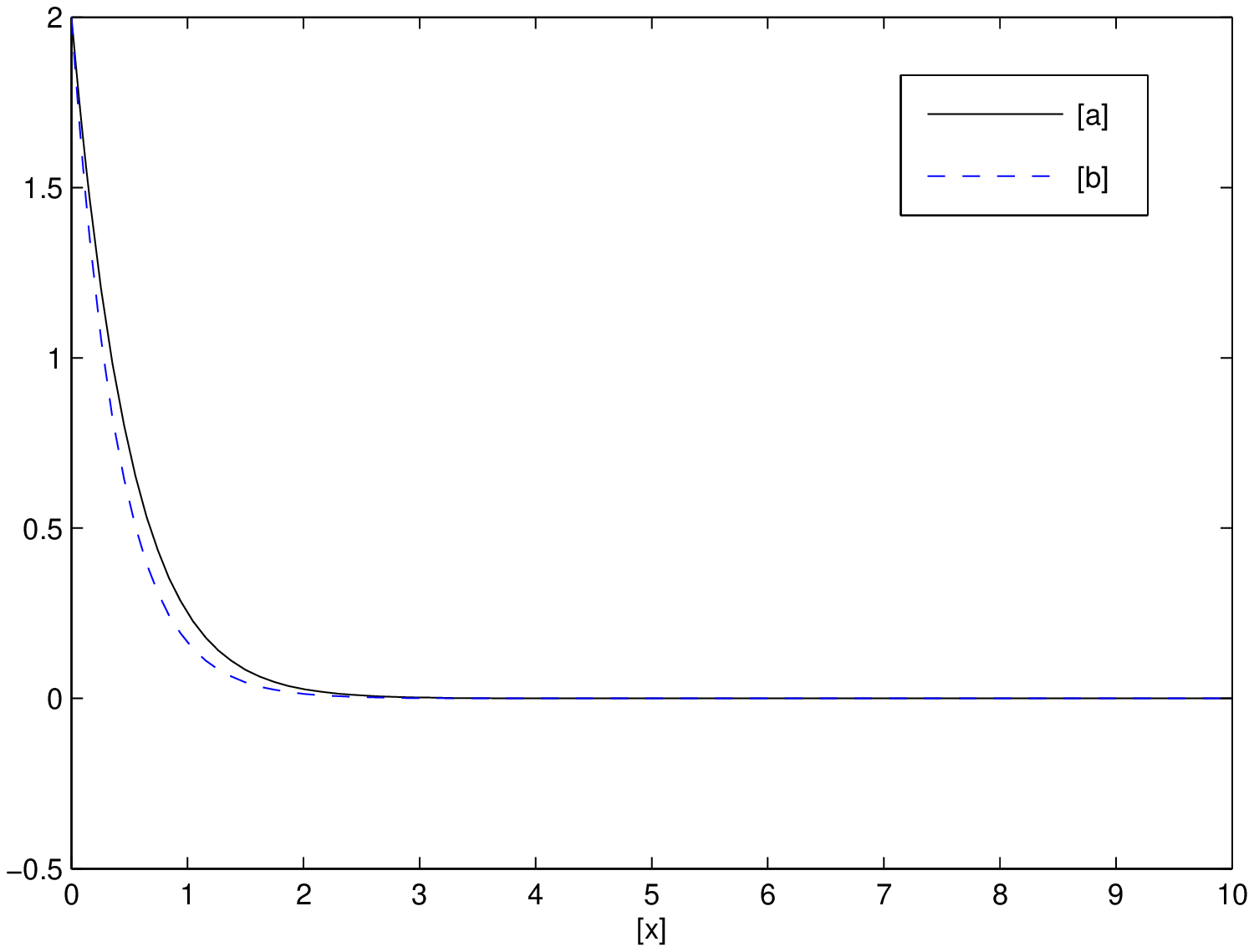}
         {
         \psfrag{[x]}{}
         \psfrag{[y]}{$$}
         \psfrag{[a]}[][][0.6]{\hspace{-0.02in}$\overline{e}$}
         \psfrag{[b]}[][][0.6]{\hspace{0.03in}$e_L$}
         }
         \put(-62,-4){\scalebox{0.6}{$t[s]$}}
         }
         \ 
         \subfigure[$m(t)\equiv 0.3$.]{
         \label{fig:m03_trajs}
         \psfragfig[trim=1.8cm 0.8cm 1.5cm 0.7cm,clip,width=0.21\textwidth]{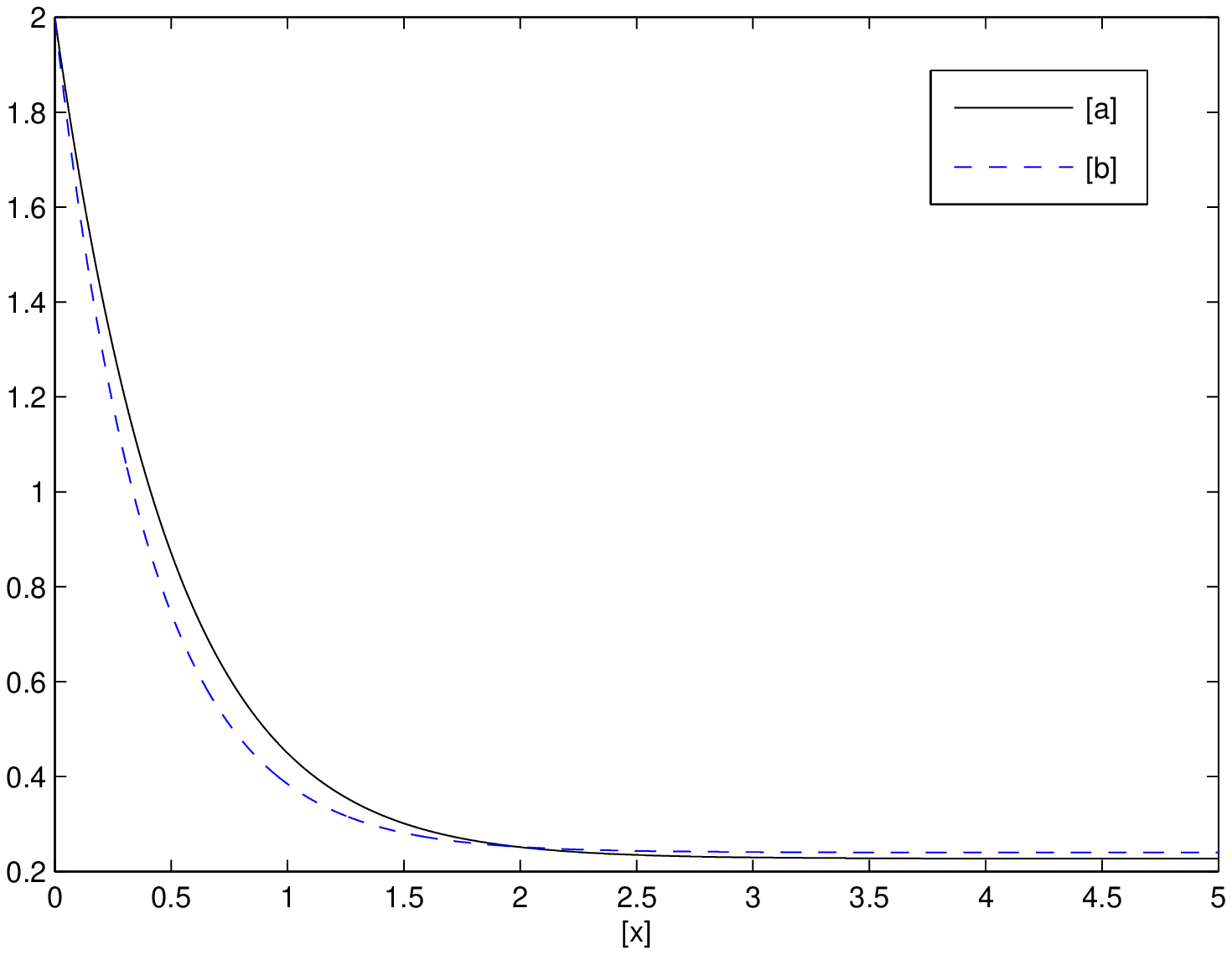}
         {
         \psfrag{[a]}[][][0.6]{\hspace{-0.02in}$\overline{e}$}
         \psfrag{[b]}[][][0.6]{\hspace{0.03in}${e}_L$}
         \psfrag{[x]}{}
         \psfrag{[y]}{\tiny$$}
         }
         \put(-62,-4){\scalebox{0.6}{$t[s]$}}
         }
         \caption{Comparisons of estimation errors of the proposed observer and that of a  Luenberger observer for different measurement noises with $N=2$.}\label{fig:prop2_m_0}
     \end{figure}
}{
    \begin{figure}[!h]
         \centering
         \subfigure[Estimation errors ($m(t)\equiv 0$) for Luenberger observer and the interconnected observers ($N=2$).]{
         \label{fig:m0_trajs}
         \psfragfig[trim=1.6cm 0.8cm 1.3cm 0.7cm,clip, width=0.3\textwidth]{JFigprop2m0}
         {
         \psfrag{[x]}{\hspace{0.02in}\tiny$t[s]$}
         \psfrag{[y]}{$$}
         \psfrag{[a]}{\hspace{-0.02in}\tiny$\overline{e}$}
         \psfrag{[b]}{\hspace{-0.02in}\tiny$e_L$}
         }
         }
         \quad
         \subfigure[Difference between estimation errors of Luenberger observer and that of the interconnected observers ($N=2$ and $m(t)\equiv 0$).]{
         \psfragfig[trim=1.5cm 0.8cm 1.3cm 0.7cm,clip,width=0.3\textwidth]{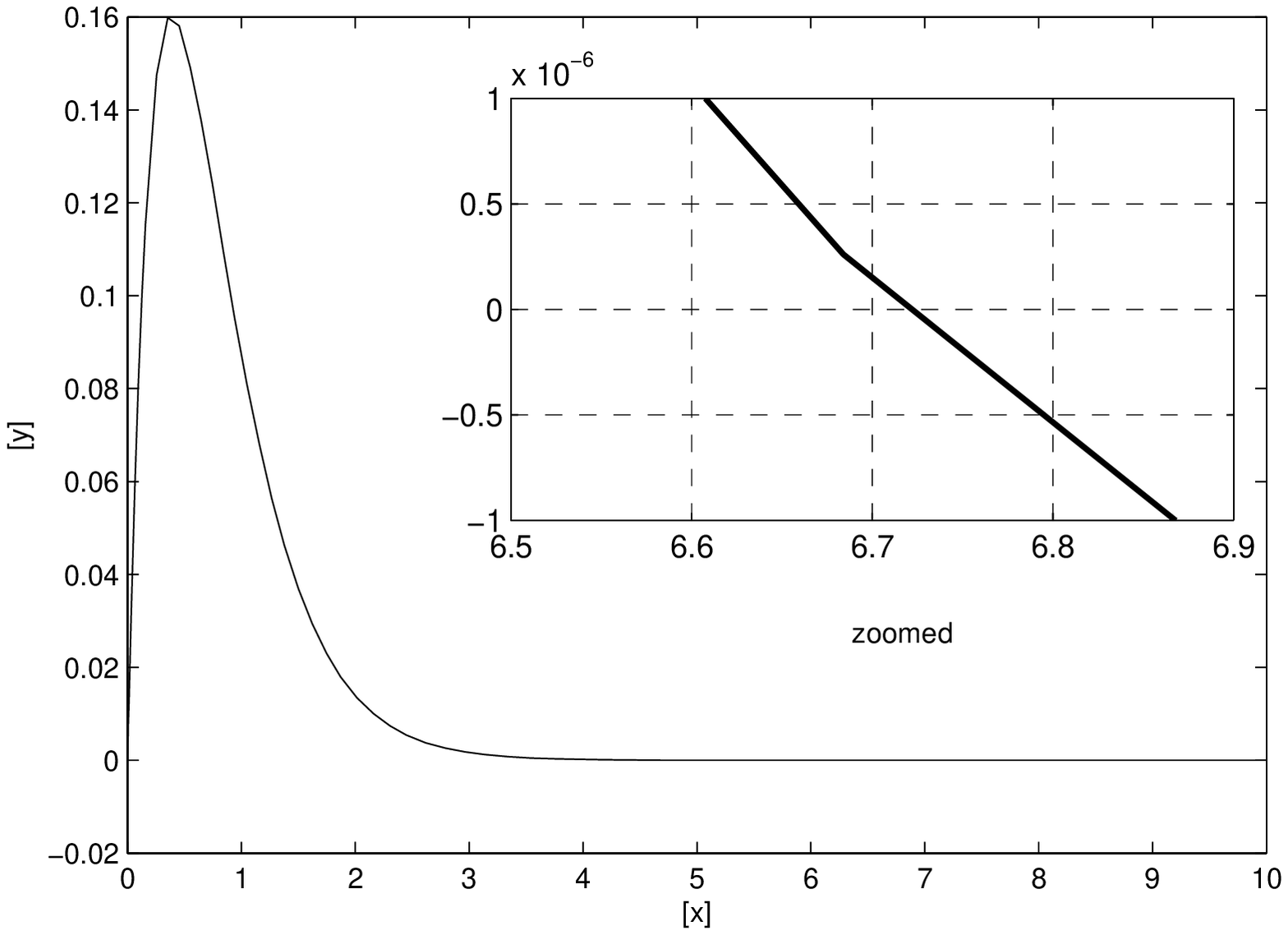}
         {
         \psfrag{[x]}{\hspace{0.02in}\tiny$t[s]$}
         \psfrag{[y]}{\tiny$\overline{e}-{e}_L$}
         \psfrag{zoomed}{\tiny{zoomed}}
         }
         }
\quad
         \subfigure[Estimation errors ($m(t)\equiv 0.3$) for Luenberger observer and the interconnected observers ($N=2$).]{
         \label{fig:m03_trajs}
         \psfragfig[trim=1.6cm 0.8cm 1.3cm 0.7cm,clip,width=0.3\textwidth]{JFigprop2m03}
         {
         \psfrag{[a]}{\hspace{-0.02in}\tiny$\overline{e}$}
         \psfrag{[b]}{\hspace{-0.02in}\tiny${e}_L$}
         \psfrag{[x]}{\hspace{0.015in}\tiny$t[s]$}
         \psfrag{[y]}{\tiny$$}
         }
         }
         \quad
         \subfigure[Difference between estimation errors of Luenberger observer and that of the interconnected observers ($N=2$, $m(t)\equiv 0.3$).]{
         \label{fig_prop2_m03b}
	 \psfragfig[trim=1.5cm 0.8cm 1.3cm 0.7cm,clip, width=0.3\textwidth]{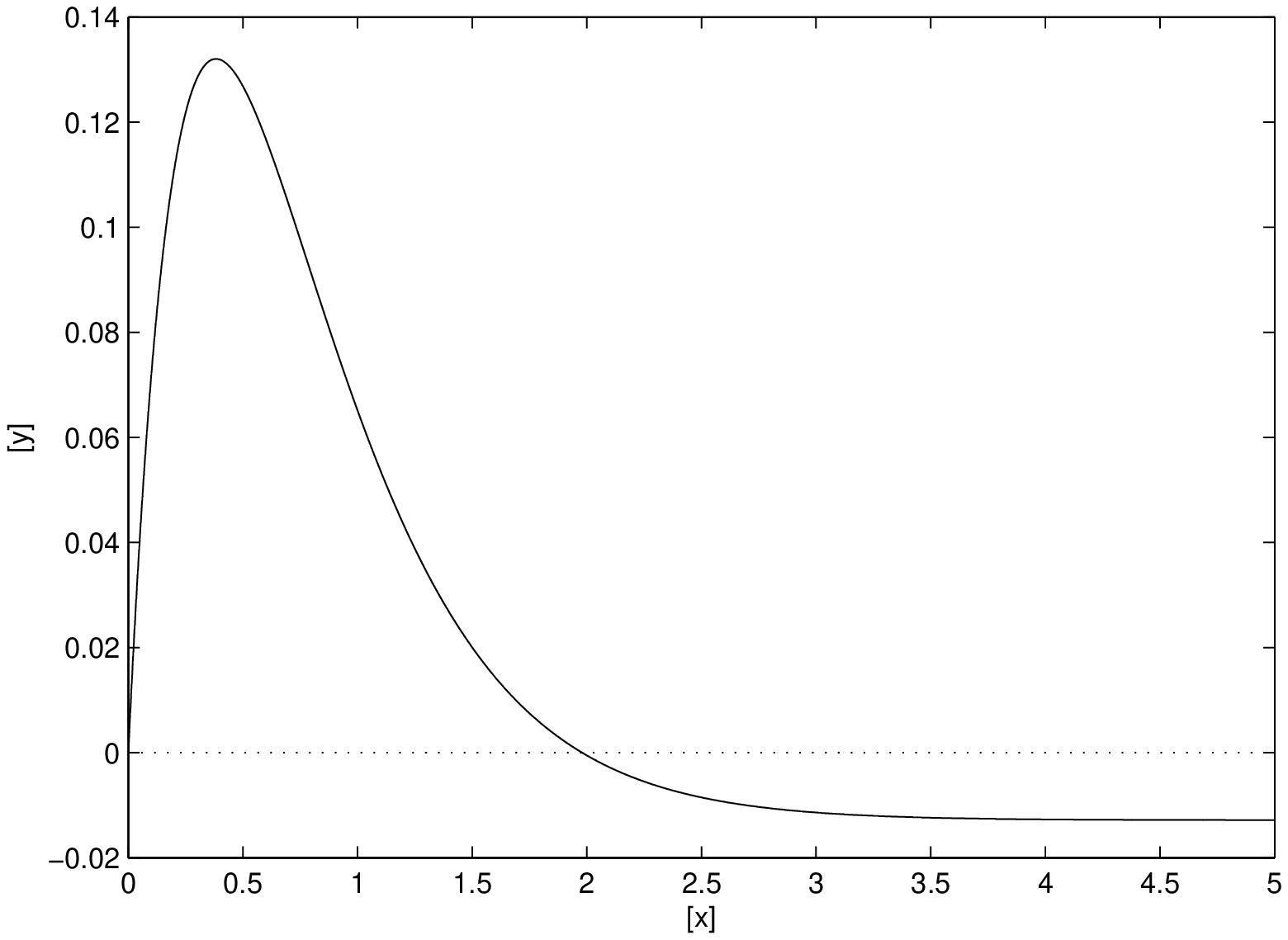}
	 {
         \psfrag{[x]}{\hspace{0.015in}\tiny$t[s]$}
         \psfrag{[y]}{\tiny$\overline{e}-\overline{e}_L$}
         }
         }   
         \quad
         \subfigure[Tail of estimation errors with low frequency sinusoid noise.]{
         \psfragfig[trim=1.6cm 0.5cm 1.3cm 0.7cm,clip, width=0.3\textwidth]{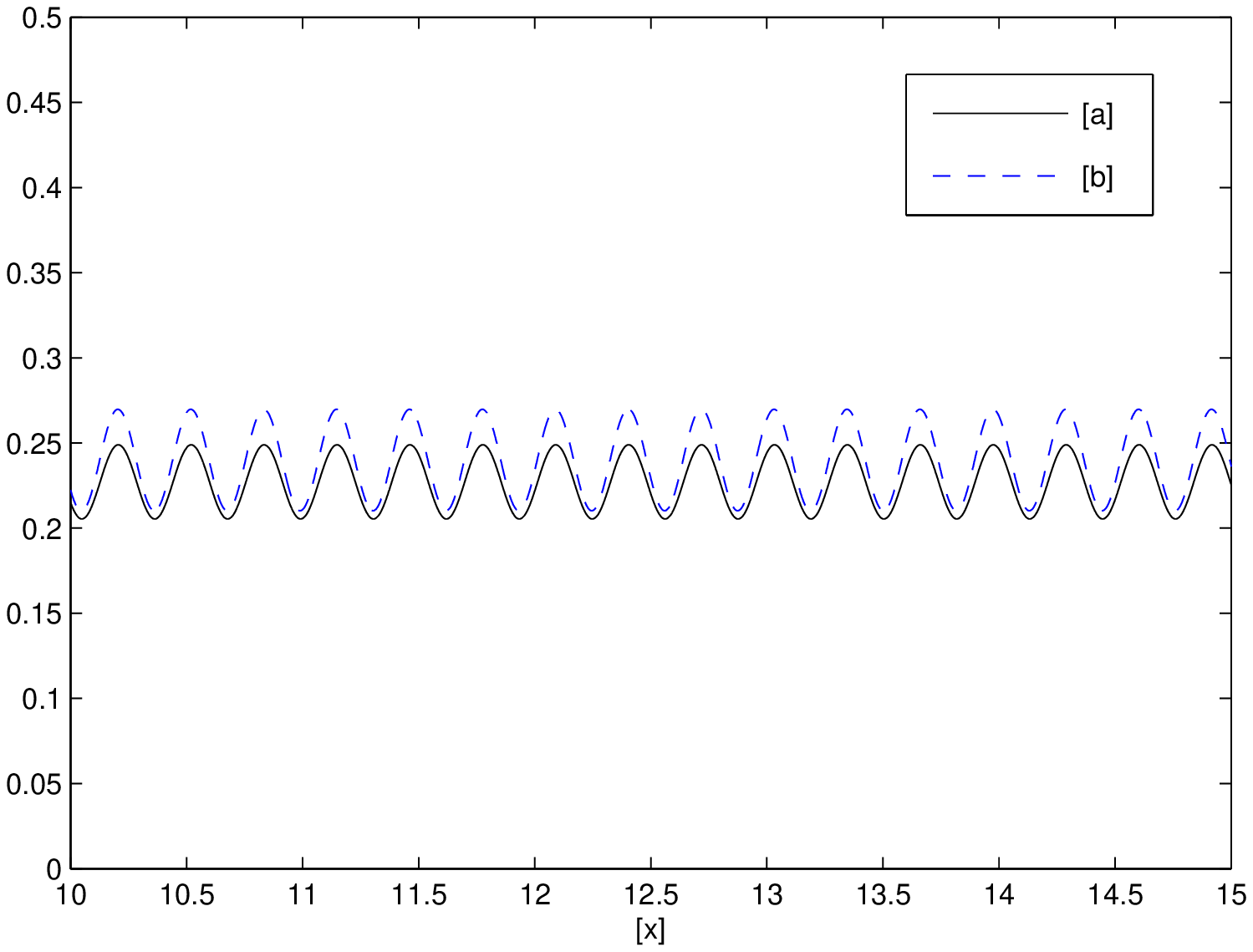}
         {
         \psfrag{[a]}{\hspace{-0.02in}\tiny$\overline{e}$}
         \psfrag{[b]}{\hspace{-0.02in}\tiny${e}_L$}
         \psfrag{[x]}{\hspace{0.01in}\tiny$t[s]$}
         \psfrag{[y]}{\tiny$$}
         }
         \label{fig:prop2_msin_lf}
         }
         \quad
         \subfigure[Tail of estimation errors with high frequency sinusoid noise.]{
         \label{fig:prop2_msin_hf}
         \psfragfig[trim=1.6cm 0.5cm 1.3cm 0.7cm,clip, width=0.3\textwidth]{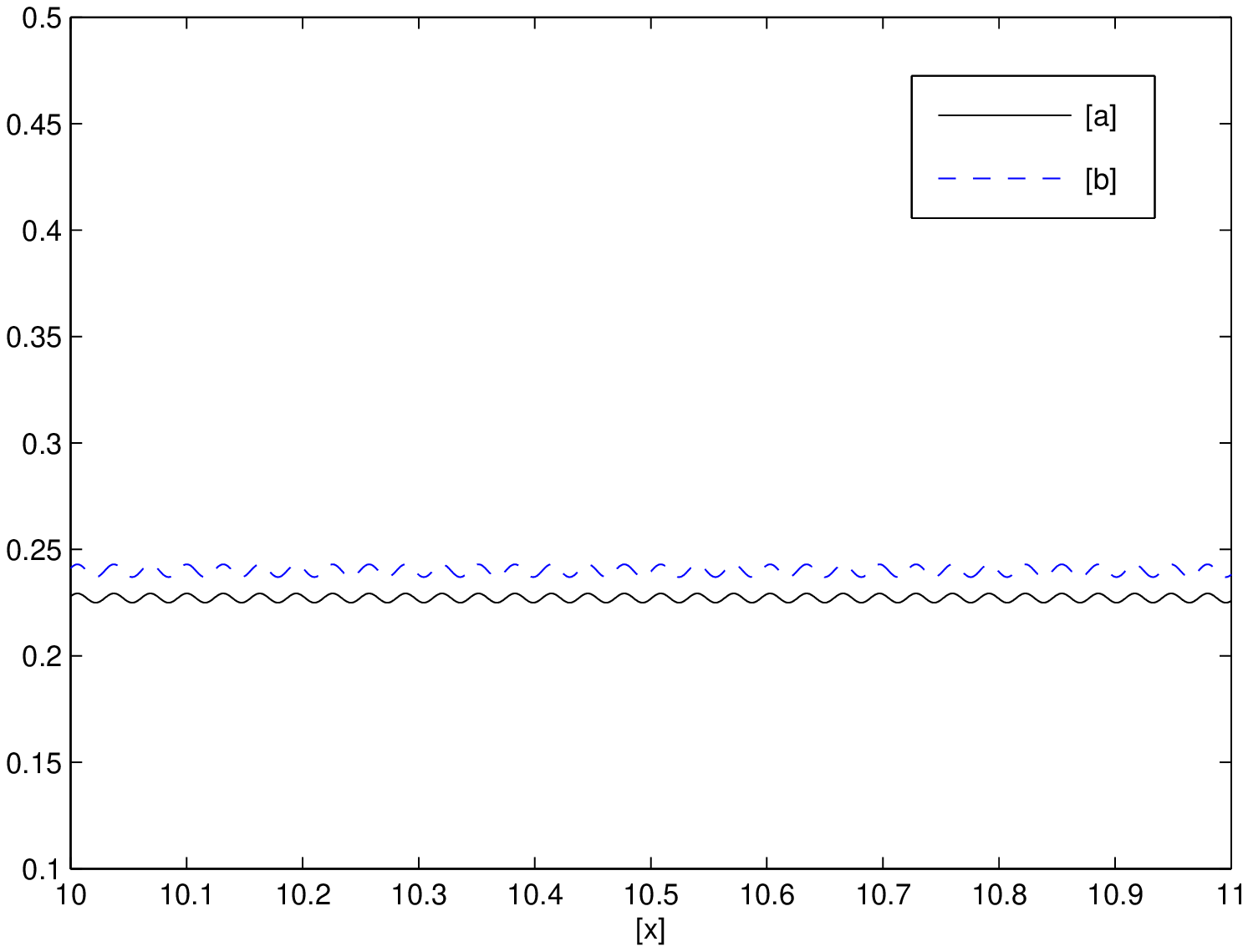}
         {
         \psfrag{[a]}{\hspace{-0.02in}\tiny$\overline{e}$}
         \psfrag{[b]}{\hspace{-0.02in}\tiny${e}_L$}
         \psfrag{[x]}{\hspace{-0.2in}\tiny$t[s]$}
         \psfrag{[y]}{\tiny$$}
         }
         }         
         \caption{Comparisons of estimation errors of the proposed observer and that of a  Luenberger observer for different measurement noises.}\label{fig:prop2_m_0}
     \end{figure} }
Simulation results for $m(t) \equiv 0.3$ are shown in Figure~\ref{fig:m03_trajs}. The behavior of the interconnected observers with constant noise is similar to that of with zero noise. It is worth to note that there is an improvement of the steady-state error by the interconnected observers since $\overline{e}^\star = 0.2272$, while the Luenberger observer gives $e_L^\star = 0.2400$. As shown in Figure~\ref{fig:m03_trajs}, at around $t \approx 2s$, $\overline{e}$ becomes closer to $0$ than $\overline{e}_L$ thereafter.  
To further explore the performance  of the interconnected observers, {we}
also consider  measurement noise with different frequencies,  {\itshape i.e.}, a low frequency noise $m(t)= 0.3 + 0.3\sin(20t)$ and a high frequency noise $m(t)= 0.3 + 0.3\sin(200t)$. \NotForJtwoC{In order to clearly  determine the properties of the interconnected observers, the tails of the simulations are shown in Figure~\ref{fig:prop2_msin_lf} and Figure~\ref{fig:prop2_msin_hf}, respectively.} The advantage of the interconnected observers lies on the properties of damped oscillatory behavior and smaller mean value of estimation error. Specifically, a numerical  comparison of the estimation errors after transient is reported in the first two columns of Table~\ref{tab:frequency}, which confirm the improvements guaranteed by the interconnected observers. 
\IfJtwoC{
\begin{table}[!ht]
\setlength\tabcolsep{2.8pt}
\centering
\caption{Comparison of estimation error ($\bar{e}$) of the observers with measurement noise of different frequencies.}
\begin{tabular}{|l|c|c|c|c|c|c|}
\hline
\multirow{2}{*}{\scriptsize observer type} & \multicolumn{2}{|c|} {\scriptsize low freq. noise}  &\multicolumn{2}{|c|} {\scriptsize high freq. noise}  & \multicolumn{2}{|c|} {\scriptsize $H_\infty$ from $m$ to $\bar e$} \\ \cline{2-7}
 & {\scriptsize mean $\bar{e}$} & {\scriptsize std $\bar{e}$} & {\scriptsize mean $\bar{e}$} & {\scriptsize std $\bar e$} & {\scriptsize Thm.~\ref{thm:basedonprops}}  & {\scriptsize Thm.~\ref{thm:optimization_Hinf_fixedstructure}} \\ \hline
{\scriptsize Luenberger's}  & 0.2419 & 0.0211 & 0.2395 & 0.0022  & 0.8000  & 0.8000 \\ \hline
{\scriptsize Interconnected}& 0.2286 & 0.0154 & 0.2268 & 0.0016  & 0.7572  & 0.4953 \\ \hline
{\scriptsize Improved (\%)} & 5.5    & 27.0   & 5.3    & 27.3    & 5.4   & 38.1   \\ \hline
\end{tabular}
\label{tab:frequency}
\end{table}
}{
\begin{table}
\centering
\caption{Comparison of estimation error ($\bar{e}$) of the observers with measurement noise of different frequencies.}
\begin{tabular}{|l|c|c|c|c|c|c|}
\hline
\multirow{2}{*}{observer type} & \multicolumn{2}{|c|} {low freq. noise}  &\multicolumn{2}{|c|} {high freq. noise}  & \multicolumn{2}{|c|} {$H_\infty$ gain from $m$ to $\bar e$} \\ \cline{2-7}
 & {\scriptsize mean $\bar{e}$} & {\scriptsize std $\bar{e}$} & {\scriptsize mean $\bar{e}$} & {\scriptsize std $\bar e$} & {\scriptsize Thm.~\ref{thm:basedonprops}}  & {\scriptsize Thm.~\ref{thm:optimization_Hinf_fixedstructure}} \\ \hline
Luenberger       & 0.2419 & 0.0211 & 0.2395 & 0.0022  & 0.8000  & 0.8000 \\ \hline
interconnected   & 0.2286 & 0.0154 & 0.2268 & 0.0016  & 0.7572  & 0.4953 \\ \hline
improvement (\%) & 5.5    & 27.0   & 5.3    & 27.3    & 5.4     & 38.1   \\ \hline
\end{tabular}
\label{tab:frequency}
\end{table} }
\end{example}
\IfJtwoC{\vspace{-9pt}}{}
\subsection{Design via feasibility/optimization problems}
\label{sec:Design}
The interconnected observers in \eqref{eq:graph_individual} can be designed by solving feasibility and optimization problems that minimize the $H_\infty$ gain of the transfer function from measurement noise $m$ to estimation error $\bar{e}$ (global) or $\bar e_i$ (local) under the rate of convergence constraint. 
To formulate such problems, following \cite{Scherer.1997.multiobjectiveLMI}, the error system in \eqref{eq:general_error_compact_graph}
 is rewritten as
      \begin{flalign}
      \begin{split}
               \dot{ e} = A_e {e} + u, \quad
               y_e = C_e e + m, \quad
               z_\infty = {\cal X} e,
       \end{split}\label{eq:synthesis_model}
       \end{flalign}
where $A_e= I_N \otimes A,\ C_e = -I_N \otimes C$, and the ``input'' $u$ is assigned via $u = M_u y_e$ with $M_u = {\cal K}*G^\top$. 
Note that $z_\infty$ denotes the overall estimation error (or the local estimation error) of the interconnected observers, {\itshape i.e.}, $z_\infty=\overline{e}$ with ${\cal X} = {\cal C}$ (or $z_\infty=\overline{e}_i$ with ${\cal X} = {\cal C}_i$). In the $s$-domain, the transfer function from $m$ to $z_\infty$ for \eqref{eq:synthesis_model} can be written as 
\begin{align}
T(s) = {\cal X} \big( sI - {\cal A}\big)^{-1} {\cal B} + {\cal D} \label{eq:tf_closedsys},
\end{align}
where ${\cal A} = A_e+M_uC_e, {\cal B} = M_u$, and ${\cal D} = 0$. 
Within this setting, feasibility ({\it i.e.}, inequalities) and optimization problems
associated with the design of the interconnected observers are formulated in the following sections.

\subsubsection{Rate of convergence and $H_\infty$ gain in terms of matrix inequalities} 
To guarantee that the rate of convergence of the interconnected observers is better (or no worse) than that of a Luenberger observer, the eigenvalues of the error system in \eqref{eq:general_error_compact_graph} will be assigned to the left of the vertical line at $-\sigma$ in the $s$-plane, where $\sigma$ is the rate of convergence for the Luenberger observer.
Following \cite{Chilali.1996.HwithPole}, the error system \eqref{eq:general_error_compact_graph} has all eigenvalues located to the left of $-\sigma$ on the $s$-plane if and only if there exists a matrix $P_S$ such that
      \begin{flalign}
      \begin{split}
               &\mbox{\rm He}({\cal A}, P_S) + 2 \sigma P_S < 0,\  
               P_S = P_S^\top>0.
      \end{split}\label{cons_region}
      \end{flalign}
{Note} that \eqref{cons_region} is nonlinear because of the cross term $P_S({\cal K}*G^\top)$ obtained when expanding $P_S {\cal A}$. 
The following theorem provides an equivalent linear formulation and a sufficient condition for \eqref{cons_region}.
\begin{proposition}
\label{prop:separation}
Condition \eqref{cons_region} is satisfied if 
\begin{enumerate}[a)]
\item and only if \
$
                  \mbox{\rm He}(A_{e}, P_S) + C_e^\top M_p^\top + M_p C_e + 2\sigma P_S \!<\! 0,\\ 
                  P_S= P_S^\top >0,
$
              in which case $M_u = P_S^{-1}M_p$;
\item the graph is all-to-all connected and there exists $h_1,h_2\in \mathbb{R}$ such that the following conditions hold:
\begin{enumerate}[{b.}1)]
\item  $h_1+h_2\geq  \sigma $;
\item  $P_i=P_i^\top > 0$ for each $i\in {\cal V}$ 
\item  $\mbox{\rm He}((A-K_{ii}C), P_i)+2h_1 P_i<0$ for each $i\in{\cal V}$;
\item  $\!\! \displaystyle 
                    \left[ \begin{array}{cccc}
                     2h_2P_1  &  S_{12}  &  \cdots  &  S_{1N} \\ 
                     S_{12}^\top  &  2h_2P_2  & \cdots & S_{2N} \\ 
                     \vdots & \vdots & \ddots & \vdots \\ 
                      S_{1N}^\top & S_{2N}^\top & \cdots & 2h_2P_N 
                    \end{array}\right]
                    <  0$, 
          where $S_{ij} = -(K_{ji}C)^\top P_j-P_iK_{ij}C$.
\end{enumerate}
\end{enumerate} 
\end{proposition}
\begin{proof}
Letting $M_p=P_S M_u$, and using the definition of $\cal A$, inequality \eqref{cons_region} can be written as
               \begin{align*}
                  & \mbox{\rm He}(A_{e}, P_S) + C_e^\top M_p^\top +M_p C_e +2\sigma P_S < 0,
               \end{align*} 
with $P_S= P_S^\top >0$. This proves item $a)$.
Now, assuming $b.1)$-$b.4)$ with $h_1,h_2\in \mathbb{R}$, note that the inequalities in $b.3)$ can be rewritten as 
\IfJtwoC{
\begin{align}
\setlength{\arraycolsep}{2.5pt}
\mbox{\rm diag}(Q_1,\dots,Q_N) + \mbox{\rm diag}(2h_1P_1,\dots,  2h_1 P_N) < 0,  \label{neq:region_part_Q}               
\end{align}
}{
\begin{align}
                    \left[ \begin{array}{cccc}
                     Q_1 & 0      & \cdots & 0\\
                     0      & Q_2 & \cdots & 0\\
                     \vdots  & \vdots & \ddots & \vdots\\
                     0      & 0 & \cdots & Q_N 
                    \end{array} \right] +
                     \left[ \begin{array}{cccc}
                     2h_1 P_1 & 0 & \cdots & 0\\
                     0  & 2h_1 P_2 & \cdots & 0\\
                     \vdots  & \vdots & \ddots & \vdots\\
                     0  & 0 & \cdots & 2h_1 P_N 
                    \end{array} \right]  <0,  \label{neq:region_part_Q}               
\end{align}    }    
with $Q_i = \mbox{\rm He}((A-K_{ii}C), P_i)$ for each $i\in{\cal V}$. By $b.2)$, symmetry of the inequalities \eqref{neq:region_part_Q}  and $b.4)$, and the definition of negative symmetric matrices, the sum of the left terms of \eqref{neq:region_part_Q} and $b.4)$ satisfies 
\IfJtwoC{
\begin{align}
\setlength{\arraycolsep}{2.5pt}
                    \!\!\!\!\!\left[\!\! \begin{array}{cccc}
                     Q_1 & S_{12} & \cdots & S_{1N}\\
                     S_{12}^\top  & Q_2 & \cdots & S_{2N} \\
                     \vdots & \vdots & \ddots & \vdots \\
                     S_{1N}^\top & S_{2N}^\top & \cdots & Q_N
                    \end{array} \right] \!\!+\!\!
                     \left[ \begin{array}{cccc}
                     2 \bar h P_1 & 0 & \cdots & 0\\
                     0  & 2\bar{h} P_2 & \cdots & 0\\
                     \vdots & \vdots & \ddots & \vdots \\
                     0 & 0 & \cdots & 2\bar{h} P_N
                    \end{array} \right]  \!\!<\!\!0,  \label{neq:region_part_QR}               
\end{align}
with 
$\bar{h} = h_1 + h_2$.
}{
\begin{align}
                    \left[ \begin{array}{cccc}
                     Q_1 & S_{12} & \cdots & S_{1N}\\
                     S_{12}^\top  & Q_2 & \cdots & S_{2N} \\
                     \vdots & \vdots & \ddots & \vdots \\
                     S_{1N}^\top & S_{2N}^\top & \cdots & Q_N
                    \end{array} \right] +
                     \left[ \begin{array}{cccc}
                     2(h_1+h_2) P_1 & 0 & \cdots & 0\\
                     0  & 2(h_1+h_2) P_2 & \cdots & 0\\
                     \vdots & \vdots & \ddots & \vdots \\
                     0 & 0 & \cdots & 2(h_1+h_2) P_N
                    \end{array} \right]  <0,  \label{neq:region_part_QR}               
\end{align}             
with $S_{ij} = -(K_{ji}C)^\top P_j -P_iK_{ij}C $ for $i,j\in {\cal V}$ and $j\neq i$.} Since $h_1+h_2\geq  \sigma $, \eqref{cons_region} is satisfied with 
$P_D=\mbox{\rm diag}(P_1,\dots,P_N)$.
\end{proof}
\begin{proposition}
\label{prop:linearized_two_neqs}
Conditions $b.1)$-$b.4)$ in Proposition~\ref{prop:separation} hold if and only if there exist $h_1,\ h_2,\in\mathbb{R}, \ Y_i,\ W_{ij},\ P_i$ for $i,j\in {\cal V}$ and $j\neq i$ such that:
\begin{enumerate}[a)]
\item  $h_1+h_2 \geq \sigma$,
\item  $P_i=P_i^\top > 0$, for each $i\in{\cal V}$,
\item  $\mbox{\rm He}(A, P_i) - C^\top Y_i^\top-Y_i C+2h_1 P_i<0$, for each $i\in{\cal V}$,
\item  $\displaystyle 
                     \left[ \begin{array}{cccc}
                     2h_2P_1  & R_{12} & \cdots & R_{1N}\\
                     R_{21} & 2h_2P_2 & \cdots & R_{2N}\\
                     \vdots & \vdots & \ddots & \vdots \\
                     R_{N1}  &  R_{N2} & \cdots & 2h_2P_N
                    \end{array} \right] < 0$, 
       \IfJtwoC{\\}{}
       where $R_{ij} = -C^\top W_{ji}^\top - W_{ij}C$. 
\end{enumerate}
 The conditions $b.3)$-$b.4)$ in Proposition~\ref{prop:separation} hold  with $K_{ii} = P_i^{-1} Y_i$ and $K_{ij} = P_i^{-1}W_{ij}$ for $i,j\in {\cal V}$, $j \neq i$. 
\end{proposition}
\begin{proof}
Let $Y_i = P_i K_{ii}$ and $W_{ij} = P_i K_{ij}$ for $i,j\in{\cal V}$ and $j\neq i$, then, $b.3)$-$b.4)$ in Proposition~\ref{prop:separation} can be rewritten as 
\begin{align*}
\mbox{\rm He}(A, P_i) - C^\top Y_i^\top-Y_i C+2h_1 P_i<0 
\end{align*}
for each $i\in {\cal V} $ and 
\begin{align*}
\left[\begin{array}{cccc}
2h_2P_1  &  R_{12} & \cdots & R_{1N}\\
R_{21} & 2h_2P_2 & \cdots & R_{2N}\\
\vdots & \vdots & \ddots & \vdots \\
R_{N1}  &  R_{N2} & \cdots & 2h_2P_N
\end{array}\right]<0,
\end{align*}
respectively. Therefore, $c)$ and $d)$ of Proposition~\ref{prop:linearized_two_neqs} hold.
\end{proof}

\subsubsection{Minimization of $H_\infty$ norm under rate of convergence constraint with fixed connectivity graph}
\IfJtwoC{W}{In this section, w}e consider the design of interconnected observer over a fixed digraph $\Gamma = ({\cal V}, {\cal E}, G)$. The design specifications of our interest are the rate of convergence and the $H_\infty$ gain from noise $m$ to estimation errors $\bar{e}$ or $e_i$\IfJtwoC{.}{, i.e., the ${\cal L}_2$ gain. In particular, to guarantee that the rate of convergence of the system \eqref{eq:general_error_compact_graph} is better (or no worse) than that of a single Luenberger observer as in \eqref{eq:singleob}, the eigenvalues of the error system \eqref{eq:general_error_compact_graph} will be assigned to the left of the vertical line at $-\sigma$ in the $s$-plane, where $\sigma$ is the convergence rate for the Luenberger observer. }
\begin{theorem}
\label{thm:optimization_Hinf_fixedstructure}
Given a plant as in \eqref{eq:plant} and a digraph $\Gamma$,  the rate of convergence is larger than or equal to $\sigma$ and the global $H_\infty$ gain (respectively, the local $H_\infty$ gain) from $m$ to estimation error $\bar{e}$ in \eqref{eq:general_error_compact_graph} (respectively, $\bar{e}_i$ in \eqref{eq:error_individual}) is minimized if and only if there exist matrices ${\cal K}$,  $P_S$, and $P_H$ such that the following optimization problem has a solution:  
\begin{subequations}
\begin{align}
&\inf \gamma \label{eq:synthesis_obj_fixedgraph}\\
 \mbox{s.t.}\ &\mbox{\rm He}({\cal A}, P_S) + 2 \sigma P_S <0, \label{eq:synthesis_cons_eigen_fixed} \\
&\left[\begin{array}{ccc}
\text{$\mbox{\rm He}$}({\cal A}, P_H) & P_H {\cal B} & {\cal X}^\top \\
{\cal B}^\top P_H             & -\gamma I         & 0 \\
{\cal X}                             & 0                      & -\gamma I
\end{array}\right] <0, \label{eq:synthesis_cons_hinf_fixed} \\
& P_S = P_S^\top >0,\ P_H = P_H^\top >0,
\end{align}
\label{eq:synthesis_cons}%
\end{subequations} 
where ${\cal X} = {\cal C}$ (respectively, ${\cal X} = {\cal C}_i$ and ${\cal C}_i$ is the sub-matrix of ${\cal C}$ from the $(in-n+1)$-th row to the $(in)$-th row).
\end{theorem}
\begin{proof}
From \cite[Theorem 2.41]{90.scherer.dssertation}, the $H_\infty$ gain for a system from input to output with realization $T_1(s) = C_1(sI-A_1)^{-1}B_1$ is less than or equal to $\gamma$ if and only if there exists some $P_H = P_H^\top>0$ such that
\begin{align}
\left[\begin{array}{ccc}
\mbox{\rm He}(A_1,P_H) & P_H B_1 & C_1^\top \\
B_1^\top P_H & -\gamma I & 0\\
C_1 & 0 & -\gamma I 
\end{array}\right] < 0,
\label{eq:general_Hinf}
\end{align}
Then, for system \eqref{eq:general_error_compact_graph} with $T(s) = {\cal C}(sI-{\cal A})^{-1}{\cal B}$, we have that the global $H_\infty$ gain from $m$ to $\bar{e}$ is less than or equal to $\gamma$ if and only if \eqref{eq:general_Hinf} holds with $A_1 = {\cal A}$, $B_1 = {\cal B}$ and $C_1={\cal C}$, which leads to \eqref{eq:synthesis_cons_hinf_fixed} with ${\cal X} = {\cal C}$. The same argument applies for $T_i(s) = {\cal C}_i(sI-{\cal A})^{-1}{\cal B}$ which leads to \eqref{eq:synthesis_cons_hinf_fixed} with ${\cal X} = {\cal C}_i$. Then, the proof finishes by adding constraint \eqref{cons_region}.
\end{proof}
\begin{remark} 
For a fixed connectivity graph, the optimization problem in \eqref{eq:synthesis_cons} can be solved offline. Moreover, due to the form of the observer at each node as in \eqref{eq:graph_individual}, the information needed by each agent is what the neighbors provide through the connectivity graph.
Therefore, the resulting observers for each agent are decentralized.
\end{remark}

Note that the optimization problem \eqref{eq:synthesis_cons} is not jointly convex over the variables ($P_S$, $P_H$, $M_u$).  Moreover, it is nonlinear because of the existence of cross terms $P_H M_u$ and $P_S M_u$. In order to remove the nonlinearities and make the two constraints jointly convex, following \cite{Scherer.1997.multiobjectiveLMI}, we reformulate the problem by seeking common solutions of $P_S$ and $P_H$, and changing variables to  $M_p := P M_u$.
Using item $a)$ of Proposition~\ref{prop:separation} to rewrite the terms $\mbox{\rm He}({\cal A},P_S)$ and $\mbox{\rm He}({\cal A},P_H)$ in \eqref{eq:synthesis_cons}, we have the following result.
\begin{theorem}
\label{tim:common_P_linearized}
Given a plant as in \eqref{eq:plant} and a graph $\Gamma$,  the rate of convergence is larger than or equal to $\sigma$ and the global $H_\infty$ gain (respectively, the local $H_\infty$ gain) from $m$ to estimation error $\bar{e}$ in \eqref{eq:general_error_compact_graph} (respectively, $\bar{e}_i$ in \eqref{eq:error_individual}) is minimized if there exist $M_p$ and $P$ such that the following optimization problem (LMI) is feasible:
      {\setlength\arraycolsep{2.5pt}
      \begin{flalign*}
      \begin{split}
            &\quad \inf \gamma \\
            &\textrm{s.t.:}
            \ \mbox{\rm He}(A_{e},P) + C_e^\top M_p^\top +  M_p  C_e +2\sigma P < 0, \\
            & \left[
              \begin{array}{ccc}
               \mbox{\rm He}(A_{e},P) + C_e^\top M_p +M_p^\top C_e    &  M_p  & {\cal X}^\top \\
                M_p^\top & -\gamma I & 0 \\
               {\cal X}           & 0 & -\gamma I
              \end{array}
              \right]
              <0,\\
             & P=P^\top>0,
      \end{split}
      \end{flalign*}   }%
where ${\cal X} = {\cal C}$ (respectively, ${\cal X} = {\cal C}_i$ and ${\cal C}_i$ is the sub-matrix of ${\cal C}$ from the $(in-n+1)$-th row to the $(in)$-th row).       
\end{theorem}

\begin{remark}
The resulting observer gain matrix from Theorem~\ref{tim:common_P_linearized} is given by $M_u = P^{-1}M_p$. By making the optimization problem linear and convex, a global optimizer is guaranteed. However, asking for common solution of $P_H=P_D$ may eliminate a better feasible solution to the original optimization problem in \eqref{eq:synthesis_cons}. 
\end{remark}

Following \IfJtwoC{\cite{05.Ebihara.dilatedLMI},}{\cite{04.Ebihara.DilatedLMI,05.Ebihara.dilatedLMI},} it is possible to formulate an equivalent convex optimization problem to the one in 
Theorem~\ref{tim:common_P_linearized} but with noncommon $P_D$ and $P_H$ matrices, see \IfJournal{{\color{black}\cite[Appendix  F]{Li.Sanfelice.13.TR.Interconnected}}}{{Appendix~\ref{app:dilatedLMI}}}.

\IfJtwoC{}{Next, we provide an example to illustrate the results above. }
\begin{example}
\label{ex:saturation}
We revisit the motivational example with connectivity graph as in  Figure~\ref{fig:2nodes_to_neighbor_only}. To further indicate the improvement obtained by the proposed observer,
we choose $K_{11} = K_{22} = K_L = 2$,  and $K_{12} = -0.5 K_L=-1$. The resulting local $H_\infty$ gain from $m$ to $\bar e_1$ is $0.55$, which is smaller than that of the Luenberger observer, which is $0.8$. If instead the connectivity graph in Figure~\ref{fig:2nodes_alltoall} is considered, we can further optimize the parameters by solving the optimization problem  \eqref{eq:synthesis_cons}.
Feasible parameters for \eqref{eq:synthesis_cons} are found using the solver PENBMI \cite{03Michal.PENNON}. For $K_{11} \approx 3.5198$, $K_{22} \approx 0.4802$, $K_{12}\approx -8.0142$, $K_{21} \approx 0.2883$, the resulting global $H_\infty$ gain is $\approx 0.4953$, which is $\approx 38.09\%$ smaller than that of Luenberger observer (which is $0.8$ with $K_L=2$). This improvement and the improvement obtained when using 
Theorem~\ref{thm:basedonprops}
 are listed in the last two columns of Table~\ref{tab:frequency}. 
%

 In fact, when the rate of convergence specification is $\sigma = 2.5$, and the $H_\infty$ gain from $m$ to $\bar{e}$ is restricted to be less than or equal to $0.8$, then, by letting $K_{11}=2$ and $K_{22}=2$, we can find the feasible region for $K_{12}$ and $K_{21}$ as shown in Figure~\ref{fig:global_hinf_2}. Moreover, if the rate of convergence is required to be $\sigma = 3.0$ with the same $H_\infty$ constraint, then, by letting $K_{11}=2.5$ and $K_{22}=2.5$, we obtain the feasible region for $K_{12}$ and $K_{21}$ as shown in Figure~\ref{fig:global_hinf_2.5}. As the figure suggests, faster rate of convergence leads to a smaller feasible region for the observer parameters. More importantly, for a single Luenberger observer, there is no feasible solution for rate of convergence larger than or equal to $3.0$ and global $H_\infty$ gain less than $0.8$.
\IfJtwoC{ 
\begin{figure}[!h]
         \centering
         \subfigure[Regions for rate of convergence equal $2.5$ ($K_{12}=K_{21}=2$).]{
         \label{fig:global_hinf_2}
         \psfragfig[trim=5.3cm 2cm 1.8cm 2cm,clip,width=0.2\textwidth]{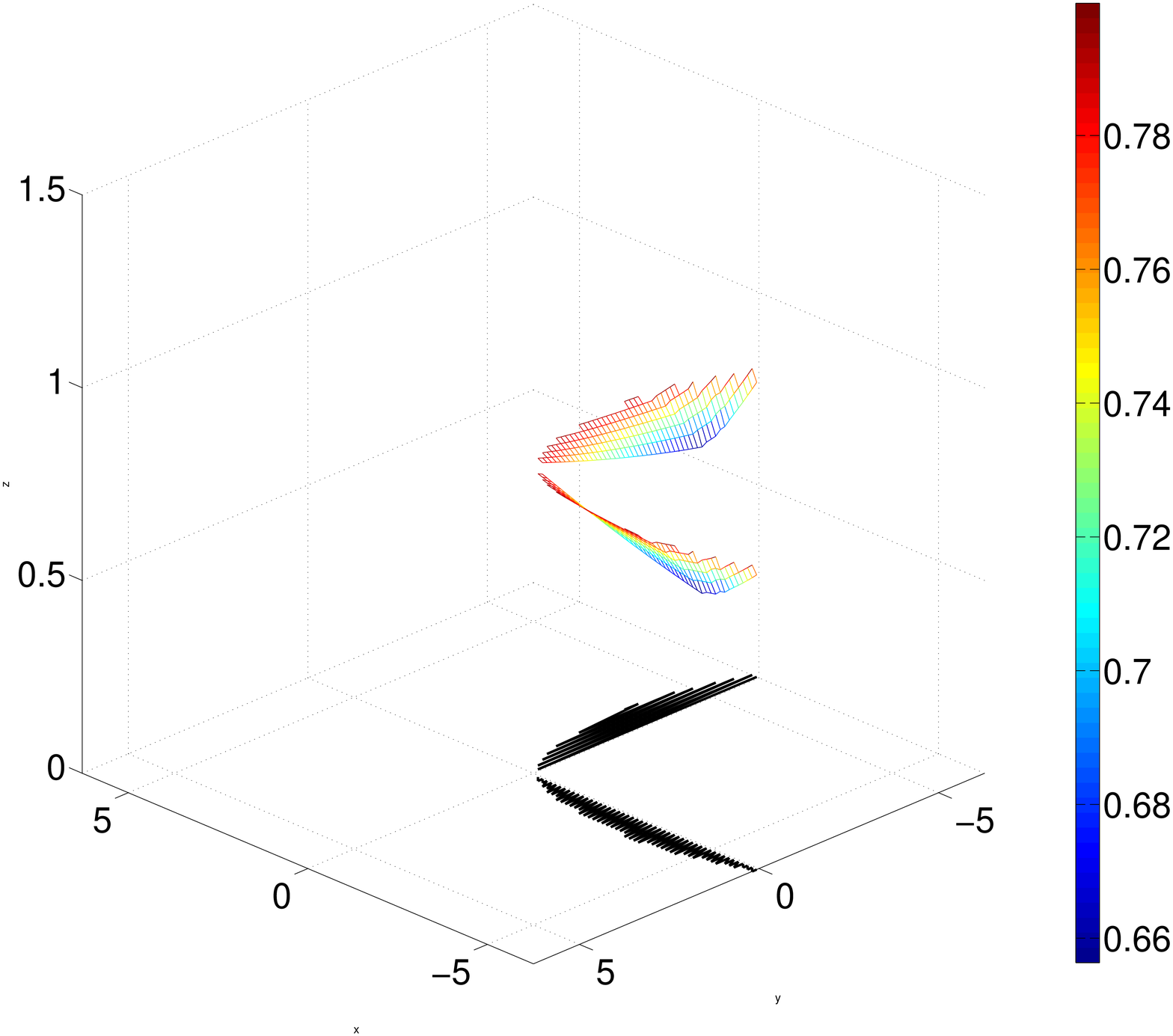}
         {
         \psfrag{x}[][][0.6]{$K_{12}$}
         \psfrag{y}[][][0.6]{$K_{21}$}  
         \psfrag{z}[][][0.6][-90]{$H_\infty$}         
         }
         }
         \ 
         \subfigure[Regions for rate of convergence equal $3.0$ ($K_{12}=K_{21}=2.5$).]{
         \label{fig:global_hinf_2.5}
         \psfragfig[trim=5.3cm 2cm 1.8cm 2cm,clip,width=0.2\textwidth]{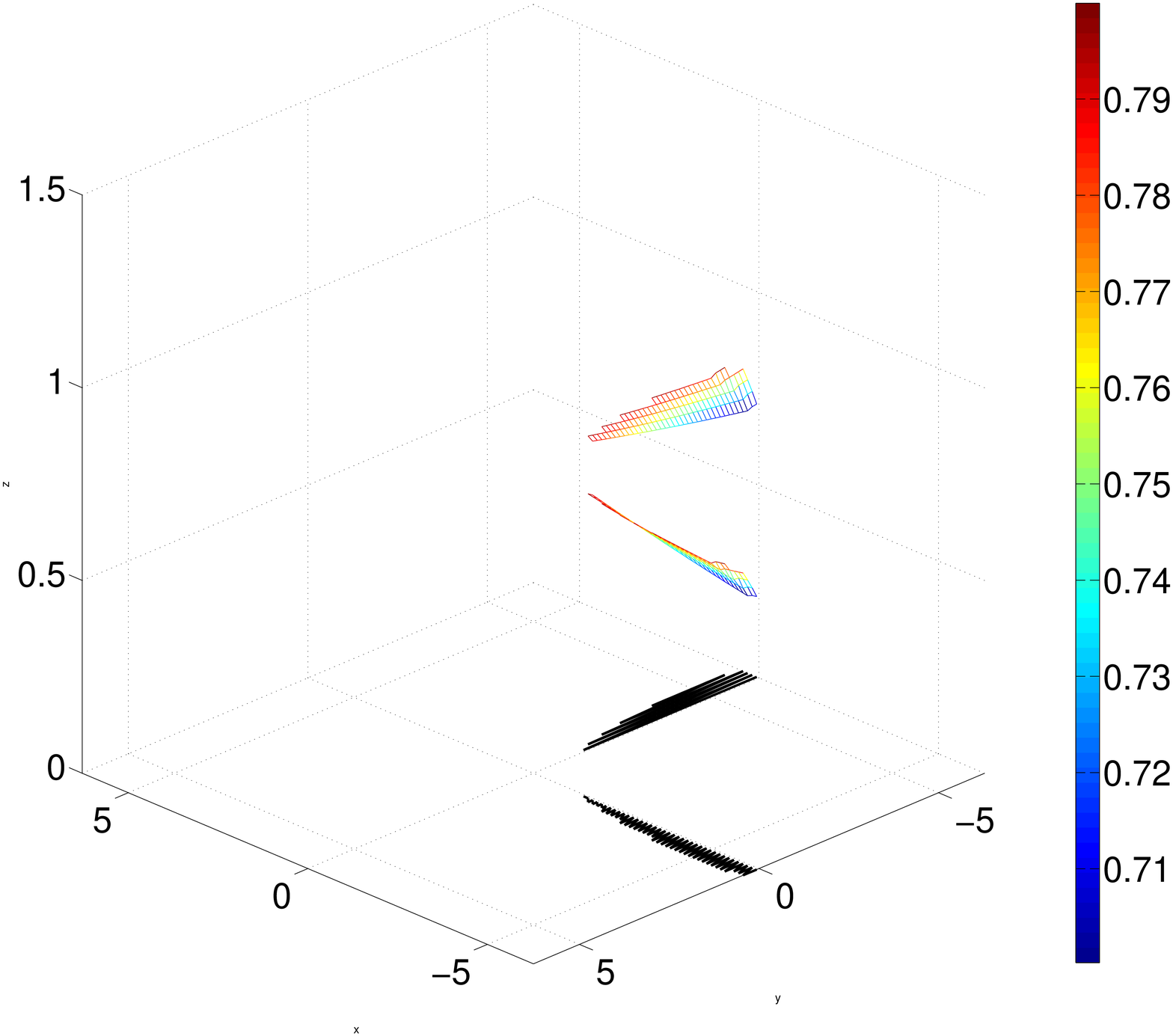}
         {
         \psfrag{x}[][][0.6]{$K_{12}$}
         \psfrag{y}[][][0.6]{$K_{21}$}     
         \psfrag{z}[][][0.6][-90]{\hspace{0in}$H_\infty$}  
         }
         }         
         \caption{Feasible regions for observer parameters subject to different rate of convergence specification and global $H_\infty$ gain less than $0.8$.}
         \label{}
\end{figure}
}{
\begin{figure}[!h]
         \centering
         \subfigure[Regions for rate of convergence equal $2.5$ ($K_{12}=K_{21}=2$).]{
         \label{fig:global_hinf_2}
         \psfragfig[trim=5.3cm 2cm 1.8cm 2cm,clip,width=0.4\textwidth]{JFiginterLplane20global}
         {
         \psfrag{x}[][][0.6]{$K_{12}$}
         \psfrag{y}[][][0.6]{$K_{21}$}  
         \psfrag{z}[][][0.6][-90]{$H_\infty$}         
         }
         }
         \ 
         \subfigure[Regions for rate of convergence equal $3.0$ ($K_{12}=K_{21}=2.5$).]{
         \label{fig:global_hinf_2.5}
         \psfragfig[trim=5.3cm 2cm 1.8cm 2cm,clip,width=0.4\textwidth]{JFiginterLplane25global}
         {
         \psfrag{x}[][][0.6]{$K_{12}$}
         \psfrag{y}[][][0.6]{$K_{21}$}     
         \psfrag{z}[][][0.6][-90]{\hspace{0in}$H_\infty$}  
         }
         }         
         \caption{Feasible regions for observer parameters subject to different rate of convergence specification and global $H_\infty$ gain less than $0.8$.}
         \label{}
\end{figure}
}
\IfJtwoC{
\begin{figure}[!h]
         \centering
	\psfragfig[trim=1cm 0.5cm 0.0cm 0cm,clip,width=0.45\textwidth]{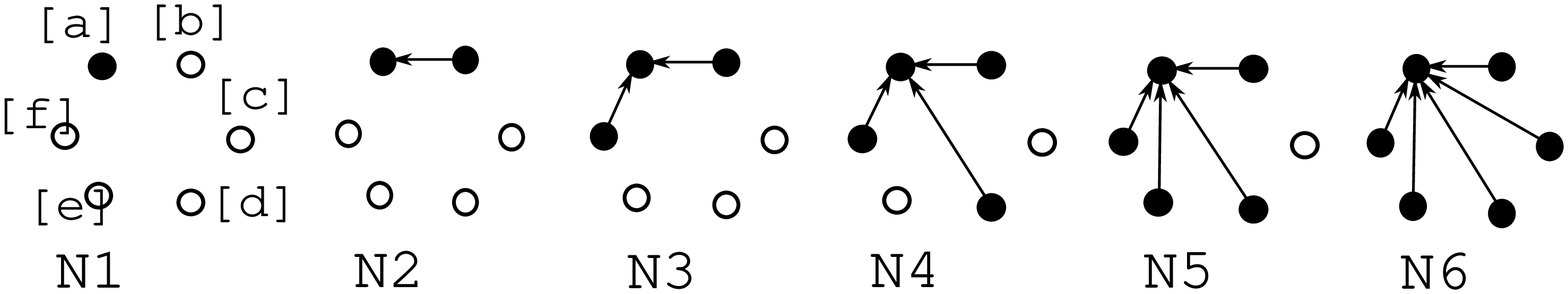}
	{      
         \psfrag{[a]}[][][0.8]{}
         \psfrag{[b]}[][][0.8]{}
         \psfrag{[c]}[][][0.8]{}
         \psfrag{[d]}[][][0.8]{}
         \psfrag{[e]}[][][0.8]{}
         \psfrag{[f]}[][][0.8]{}
         \psfrag{N1}[][][0.6]{\hspace{0.4in}$M_1=0$}
         \psfrag{N2}[][][0.6]{\hspace{0.3in}$M_1=1$}
         \psfrag{N3}[][][0.6]{\hspace{0.2in}$M_1=2$}
         \psfrag{N4}[][][0.6]{\hspace{0.3in}$M_1=3$}
         \psfrag{N5}[][][0.6]{\hspace{0.2in}$M_1=4$}
         \psfrag{N6}[][][0.6]{\hspace{0.2in}$M_1=5$}
         }
	 \put(-236,35){\scalebox{0.8}{\textcircled{\scriptsize$1$}}}  
	 \put(-208,35){\scalebox{0.8}{\textcircled{\scriptsize$2$}}}
	 \put(-200,22){\scalebox{0.8}{\textcircled{\scriptsize$3$}}}   
	 \put(-208,8){\scalebox{0.8}{\textcircled{\scriptsize$4$}}}
	 \put(-235,8){\scalebox{0.8}{\textcircled{\scriptsize$5$}}}
	 \put(-241,22){\scalebox{0.8}{\textcircled{\scriptsize$6$}}}
         \caption{Different graph structures for agent $1$ with $N=6$.}
         \label{fig:6agents_links}
\end{figure}
}{
\begin{figure}[!h]
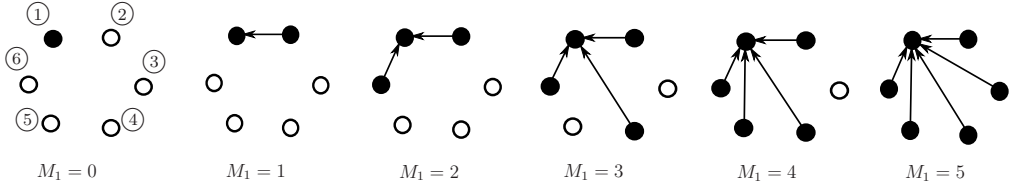

         \centering
	\psfragfig[trim=1cm 0.5cm 0.0cm 0cm,clip,width=0.8\textwidth]{J6agentslinkoneline}
	{      
         \psfrag{[a]}[][][0.8]{}
         \psfrag{[b]}[][][0.8]{}
         \psfrag{[c]}[][][0.8]{}
         \psfrag{[d]}[][][0.8]{}
         \psfrag{[e]}[][][0.8]{}
         \psfrag{[f]}[][][0.8]{}
         \psfrag{N1}[][][0.6]{\hspace{0.4in}$M_1=0$}
         \psfrag{N2}[][][0.6]{\hspace{0.3in}$M_1=1$}
         \psfrag{N3}[][][0.6]{\hspace{0.2in}$M_1=2$}
         \psfrag{N4}[][][0.6]{\hspace{0.3in}$M_1=3$}
         \psfrag{N5}[][][0.6]{\hspace{0.2in}$M_1=4$}
         \psfrag{N6}[][][0.6]{\hspace{0.2in}$M_1=5$}
         }
	 \put(-372,56){\scalebox{0.8}{\textcircled{\scriptsize$1$}}}  
	 \put(-340,56){\scalebox{0.8}{\textcircled{\scriptsize$2$}}}
	 \put(-328,38){\scalebox{0.8}{\textcircled{\scriptsize$3$}}}   
	 \put(-336,16){\scalebox{0.8}{\textcircled{\scriptsize$4$}}}
	 \put(-376,16){\scalebox{0.8}{\textcircled{\scriptsize$5$}}}
	 \put(-380,39){\scalebox{0.8}{\textcircled{\scriptsize$6$}}}
         \caption{Different graph structures for agent $1$ with $N=6$.}
         \label{fig:6agents_links}
\end{figure}
}
\begin{table}[!ht]
\setlength\tabcolsep{4.5pt}
\centering
\caption{Comparison of local $H_\infty$ norms from noise $m$ to $\bar{e}_1$ with different number of incoming edges for agent $1$.}
\begin{tabular}{|l|c|c|c|c|c|c|}
\hline
\multirow{2}{*}{} & \multicolumn{6}{|c|} {number of non-self edges ($M_1$)} \\ \cline{2-7}
& 0 & 1 & 2 & 3 & 4 & 5 \\ \hline 
local $H_\infty$    & 0.80    & 0.45   & 0.34    & 0.28    & 0.25    & 0.22         \\ \hline
improv. (\%) & 0.00    & 43.8   & 57.5    & 65.0    & 68.8    & 72.5         \\ \hline
\end{tabular}
\label{tab:numberoflinks}
\end{table}

Now, for the same plant, consider digraphs with $6$ agents where the edges are defined as in Figure~\ref{fig:6agents_links}. In all cases, each agent is self connected. Let $M_1$ denote the number of non-self edges for agent $1$, e.g., when $M_1 =0$ as shown in Figure~\ref{fig:6agents_links}, it is implied that $G=I_6$, while when $M_1 = 5$, $G =\left[\begin{array}{ll} 
\!g_1 &  g_2\!
\end{array}\right]$, $g_1 = [1\ 1_5^\top]^\top$ and $g_2 = [0\ I_5]^\top$.
Let the rate of convergence specification be $\sigma = 2.5$. Then, the local $H_\infty$ norms from noise $m=(m_1,\dots,m_6)$ to estimation error $\bar e_1$ at agent $1$ for the cases in Figure~\ref{fig:6agents_links} are shown in Table~\ref{tab:numberoflinks}. From case $M_1 =0$ to case $M_1=1$, the improvement is significant; in fact, when an incoming edge is added to agent $1$, the local $H_\infty$ is improved by $43.8\%$ when compared to the case where a single Luenberger observer is used at agent $1$.  When two agents provide information to agent $1$ ($M_1=2$), the improvement is approximately $57.5\%$, while when three and four agents communicate to agent $1$, the improvement grows to approximately $65\%$ and $69\%$ ($M_1=4$), respectively. 
\end{example}
\begin{example}[second order plant]
\label{ex:second_order_plants}
First, we consider a second-order plant given as in \eqref{eq:plant} with
$
             A=\left[ \begin{array}{cc}
               -5/2 &  1/10 \\
               4/100 & -3
             \end{array} \right],\
             C=\left[ \begin{array}{cc}
               1 & 2 
             \end{array} \right].
$
For a given Luenberger observer with $K_L = [1.5\quad -0.16]^\top$, its rate of convergence is $-3.34$ and its $H_\infty$ norm from measurement noise $m$ to estimation error $e_L$ is equal $0.34$. With the interconnected observers for $N=2$ connected via an all-to-all connectivity graph, we obtain that the optimal global $H_\infty$ norm from measurement noise $m$ to estimation error $\bar e$ is approximately $0.05$ and the optimal local $H_\infty$ norm from $m$ to $\bar e_1$ (or $\bar{e}_2$) is $0.03$ with 
$M_u = [v_1\ v_2]$, where $v_1 = [10.3834\ -1.6019\ -10.7581\ 1.5963]^\top$ and $v_2 = [7.1992\ -1.2410\ -7.3028\ 1.2426]^\top$.
This is $\approx85.29\%$ smaller than that of Luenberger observers. 

Then, we consider a second-order plant with oscillatory behavior given as in \eqref{eq:plant} with
$
             A=\left[ \begin{array}{cc}
               0 & -1 \\
               1 & 0
             \end{array} \right]$, 
$             C=\left[ \begin{array}{cc}
               1 & 0 
             \end{array} \right].
$
For a given Luenberger observer with $K_L = [2\quad 0]^\top$, its rate of convergence is $-1$ and its $H_\infty$ norm from measurement noise $m$ to estimation error $e_L$ is equal $2$. With the interconnected observers with $N=2$ connected via an all-to-all connectivity  graph, by formulating the problem according to \eqref{eq:synthesis_model}, the optimization problem in Theorem~\ref{thm:optimization_Hinf_fixedstructure} is solved and the gain matrix is found as 
     {
     $M_u = [v_1\ v_2]$, where $v_1 = [7.9503\ -9.9554\ -5.9424\  9.0014]^\top$ and $v_2 = [-5.9324\ 9.1143\ 7.9605\ -9.8426]^\top$.
     }
Its corresponding global $H_\infty$ norm from $m$ to $\overline{e}$ is $\!\approx\!1.4142$ and its local $H_\infty$ norm from $m$ to $\bar e_1$ (or $\bar e_2$) is $\approx\!1$. Comparing to the Luenberger observer, the global $H_\infty$ norm is decreased by $\!\approx\!29.3\%$ and the local $H_\infty$ norm is decreased by $\!\approx\!50.0\%$.  
\end{example}
The improvements on the local $H_\infty$ gain guaranteed by the proposed interconnected observers in the examples above are justified by the fact that the sufficient condition given in the upcoming Section~\ref{sec:sufficientcondHinf} are satisfied; see Theorem~\ref{thm:smaller_Hgain_guaranteed} and below it, where these examples are revisited.
\subsubsection{Minimization of $H_\infty$ norm under rate of convergence constraint with optimized connectivity graph}
For interconnected observers whose digraph has not yet been specified, a natural question to ask is whether there exists a digraph that minimizes the number of links between agents for the given specifications. 
\IfJtwoC{}{More precisely, given a rate of convergence $\sigma$ and a desired $H_\infty$ gain $\gamma^\star$, find a digraph with minimum number of edges.}
In applications, such minimizations could potentially lower the cost of a distributed system as it could reduce the number of agents and communication links.  The following  result  provides a sufficient and necessary condition for such optimization problem.
\begin{theorem}
\label{thm:minimizinglinks}
For the error system \eqref{eq:general_error_compact_graph}, the rate of convergence is larger than or equal to $\sigma$ and the global $H_\infty$ norm (respectively, the local $H_\infty$ norm) from noise $m$ to estimation error $\bar{e}$ in \eqref{eq:general_error_compact_graph} (respectively, $\bar{e}_i$ in \eqref{eq:error_individual}) is less than or equal to $\gamma^\star$ over a digraph $\Gamma$ with minimized number of edges if and only if there exist matrices ${\cal K}$, $G$,  $P_S$, and $P_H$ such that the following optimization problem has a solution:  
\begin{subequations}
\begin{align}
&\inf \mbox{\rm tr}(D) \label{eq:synthesis_obj}\\[0.001mm]
 \mbox{s.t.}\ &\mbox{\rm He}({\cal A}, P_S) + 2 \sigma P_S <0, \label{eq:synthesis_cons_eigen} \\[0.001mm]
&\left[\begin{array}{ccc}
\mbox{\rm He}({\cal A}, P_H) & P_H {\cal B} & {\cal X}^\top \\[0.001mm]
{\cal B}^\top P_H             & -\gamma^\star I         & 0 \\[0.001mm]
{\cal X}                             & 0                      & -\gamma^\star I
\end{array}\right] <0, \label{eq:synthesis_cons_hinf} \\[0.001mm]
& P_S = P_S^\top >0,\ P_H = P_H^\top >0,
\end{align}
\label{eq:synthesis_cons_nonlinear}%
\end{subequations} 
where ${\cal X} = {\cal C}$ (respectively, ${\cal X} = {\cal C}_i$).
\end{theorem}
\begin{proof}
Following the proof of Theorem~\ref{thm:optimization_Hinf_fixedstructure}, the global $H_\infty$ gain over a digraph $\Gamma$ is less than or equal to $\gamma^\star$ if and only if \eqref{eq:general_Hinf} holds with $A_1={\cal A}$, $B_1 = {\cal B}$, $C_1 = {\cal C}$, $\gamma = \gamma^\star$ and $P_H=P_H^\top>0$. The same argument applies to the local $H_\infty$ gain. Moreover, the rate of convergence condition is satisfied if and only if \eqref{eq:synthesis_cons_eigen} holds. Since $\mbox{\rm tr}(D)=\sum_{i=1}^{N} \sum_{j=1}^N g_{ij}$, where $g_{ij} =1$ indicates there is an edge from node $j$ to node $i$, then the number of edges of the graph is minimized if and only if $\mbox{\rm tr}(D)$ is minimized. 
\end{proof} 
The constraints in \eqref{eq:synthesis_cons_eigen} and \eqref{eq:synthesis_cons_hinf} are nonlinear and not jointly convex. By changing variables, the nonlinear constraints in \eqref{eq:synthesis_cons_eigen} and \eqref{eq:synthesis_cons_hinf} can be linearized as a function of $Q$ and $P$.
\begin{theorem} 
For the error system \eqref{eq:general_error_compact_graph}, the rate of convergence is larger than or equal to $\sigma$ and the global $H_\infty$ norm (respectively, the local $H_\infty$ norm) from noise $m$ to estimation error $\bar{e}$ in \eqref{eq:general_error_compact_graph} (respectively, $\bar{e}_i$ in \eqref{eq:error_individual}) is less than or equal to $\gamma^\star$ over a digraph $\Gamma$ with minimized number of communication links if  there exist matrices ${\cal K}$, $G$, and $P$ such that the following optimization problem is feasible:  
\IfJtwoC{
\begin{subequations}
\begin{align}
&\inf \mbox{\rm tr}(D) \label{eq:synthesis_obj_linearized}\\
 s.t.\ &\mbox{\rm He}(I_N \otimes A, P) \startmodif + \stopmodif Z +  2 \sigma P <0, \label{} \\
&\left[\begin{array}{ccc}
\mbox{\rm He}(I_N \otimes A, P) \startmodif+\stopmodif Z & Q & {\cal X}^\top \\
Q^\top                             & -\gamma^\star I         & 0 \\
{\cal X}                             & 0                      & -\gamma^\star I
\end{array}\right] <0, \label{} \\
& P = P^\top >0,
\end{align}
\label{eq:synthesis_linearized}%
\end{subequations}
where $Q = P ({\cal K} * G^\top)$, $Z = - Q (I_N \otimes C) - (I_N \otimes C)^\top Q^\top$,  and ${\cal X} = {\cal C}$ (respectively, ${\cal X} = {\cal C}_i$).
}{
\begin{subequations}
\begin{align}
&\inf \mbox{\rm tr}(D) \label{eq:synthesis_obj_linearized}\\
 s.t.\ &\mbox{\rm He}(I_N \otimes A, P) - Q (I_N \otimes C) - (I_N \otimes C)^\top Q^\top +  2 \sigma P <0, \label{} \\
&\left[\begin{array}{ccc}
\mbox{\rm He}(I_N \otimes A, P) - Q (I_N \otimes C) - (I_N \otimes C)^\top Q^\top & Q & {\cal X}^\top \\
Q^\top                             & -\gamma^\star I         & 0 \\
{\cal X}                             & 0                      & -\gamma^\star I
\end{array}\right] <0, \label{} \\
& P = P^\top >0,
\end{align}%
\label{eq:synthesis_linearized}%
\end{subequations}%
where $Q = P ({\cal K} * G^\top)$, and ${\cal X} = {\cal C}$ (respectively, ${\cal X} = {\cal C}_i$). }
\end{theorem}
\begin{proof}
Let ${\cal K}$, $G$ and $P$ be solutions of the optimization problem \eqref{eq:synthesis_linearized}.   Since the matrix ${\cal K} * G^\top$ is such that $Q = P ({\cal K} * G^\top)$, using $P = P^\top$ and the definition of ${\cal A}$ in \eqref{eq:matrix_graph}, we have
\IfJtwoC{
\begin{align*}
&\mbox{\rm He}(I_N \otimes A, P) - Q (I_N \otimes C) - (I_N \otimes C)^\top Q^\top \\
&\quad = (I_N\otimes A)^\top P + P(I_N\otimes A)\\
&\qquad -(I_N\otimes C)^\top({\cal K} * G^\top)^\top P^\top - P ({\cal K} * G^\top)(I_N\otimes C)\\
&\quad=\mbox{\rm He}({\cal A},P).
\end{align*}
}{
\begin{align*}
&\mbox{\rm He}(I_N \otimes A, P) - Q (I_N \otimes C) - (I_N \otimes C)^\top Q^\top \\
&\quad = (I_N\otimes A)^\top P + P(I_N\otimes A)-(I_N\otimes C)^\top({\cal K} * G^\top)^\top P^\top - P ({\cal K} * G^\top)(I_N\otimes C)=\mbox{\rm He}({\cal A},P).
\end{align*} }
Then, ${\cal K}$, $G$, $P_S=P$ and $P_H = P$ satisfy \eqref{eq:synthesis_cons_nonlinear}. 
\end{proof}
\begin{remark}
The results above define the graph via the resulting $G$. The resulting ${\cal K}$ and $G$ from\IfJtwoC{}{ the optimization problem} \eqref{eq:synthesis_linearized} satisfies ${\cal K}* G^\top = P^{-1} Q$, which may not be unique.
\end{remark}
\begin{example}
\label{ex:number_internal_obs}
Consider the scalar plant in \eqref{eq:scalarplant} with $a \!=\! -0.5$ as in Example~\ref{ex:saturation}, which can represent the dynamics of a mobile agent whose state is to be estimated using multiple sensors either fixed or mobile (in relative coordinates). \stopmodif  Suppose that the rate of convergence specification is $\sigma = 2.5$.  When using the graph that is all-to-all as shown in Figure~\ref{fig:alltoallconnectedgraph}, it is natural to ask the effect that the number of agents has on the improvement of the global $H_\infty$ norm. As shown in Figure~\ref{fig:saturation_global}, the resulting global $H_\infty$ gain is reduced as the number of agents $N$ grows. These results are obtained following Theorem~\ref{thm:minimizinglinks}. The improvement is summarized in Table~\ref{tab:global_numberofagent}. 
\IfJtwoC{
\begin{figure}[!h]
         \centering
         \subfigure[Graph structures with all to all connections (self connection links are not shown).]{
         \label{fig:alltoallconnectedgraph}
         \psfragfig[trim=0cm 0cm 0cm 0cm,clip,width=0.21\textwidth]{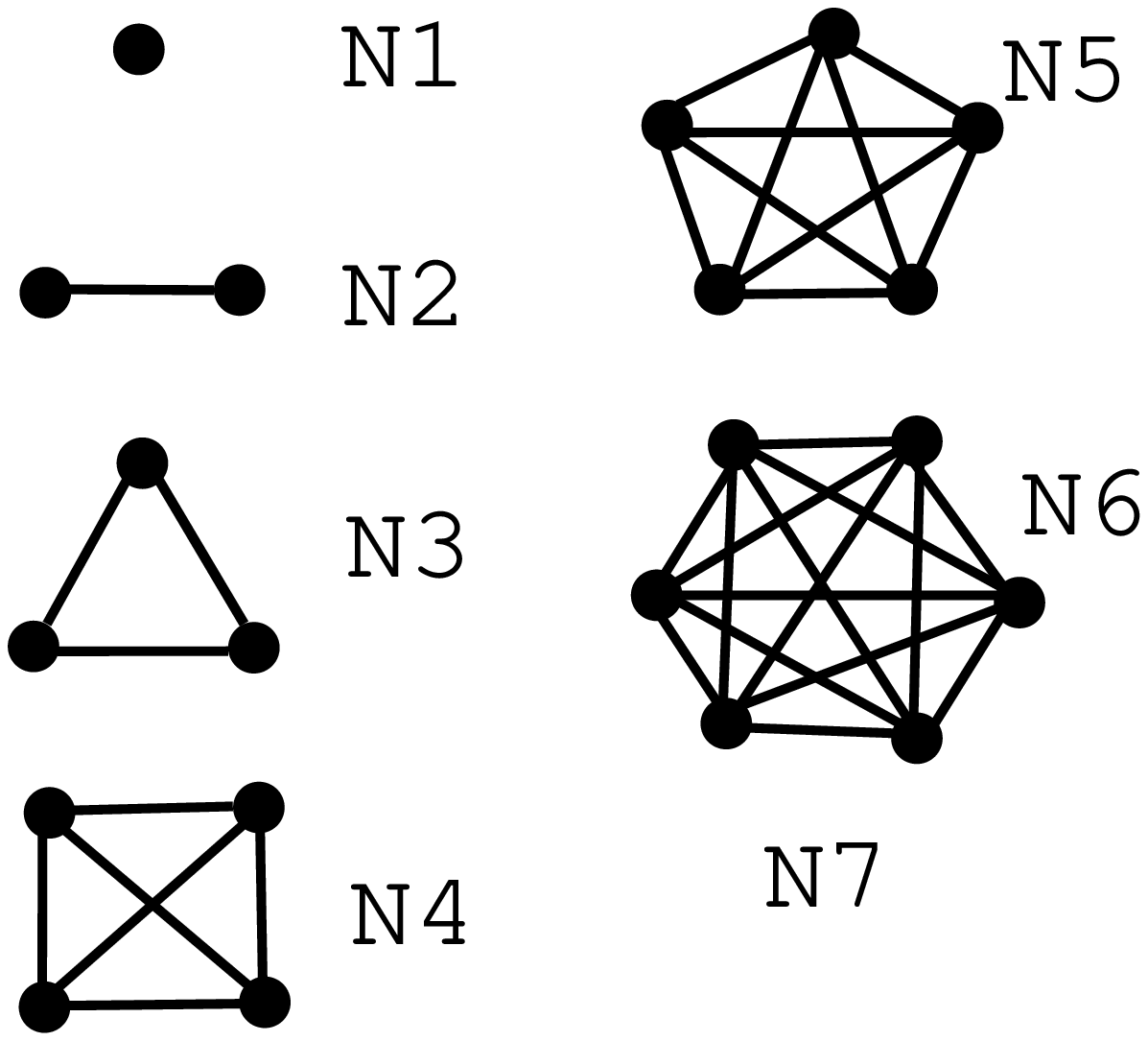}
         {
         \psfrag{N1}[][][0.6]{$N=1$}
         \psfrag{N2}[][][0.6]{$N=2$}
         \psfrag{N3}[][][0.6]{$N=3$}
         \psfrag{N4}[][][0.6]{$N=4$}
         \psfrag{N5}[][][0.6]{$N=5$}
         \psfrag{N6}[][][0.6]{$N=6$}
         \psfrag{N7}[][][1]{\ \ \ \ \quad $\vdots$}          
         }
         }
         \
         \subfigure[The global $H_\infty$ norm from noise $m$ to estimation error $\bar{e}$ with respect to the number of agents.]{
         \label{fig:saturation_global}
         \centering
         \psfragfig[trim=4.6cm 0.8cm 2.4cm 0.6cm,clip,width=0.22\textwidth]{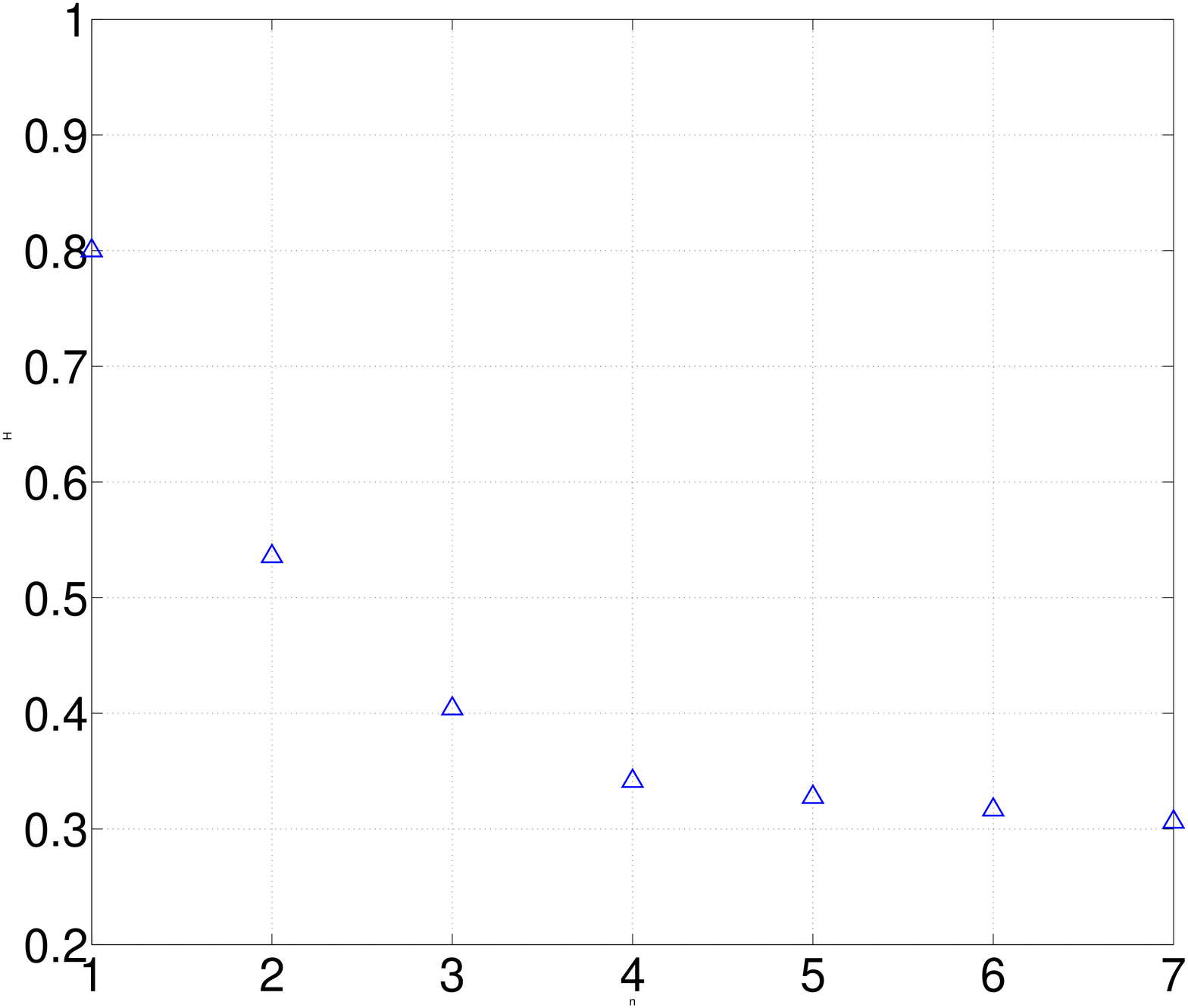}
         {
         \psfrag{n}[][][0.6]{\hspace{-0.4in}$N$}
         \psfrag{H}[][][0.6][-90]{}
         }
	\put(-100,52){\scalebox{0.7}{\color{red}\bf big improvement}}
	\put(-63,28){\scalebox{0.7}{\color{mygreen}\bf little improvement}}
	\put(-117,58){\scalebox{0.7}{{$\gamma$}}}
         }         
         \caption{The influence of the number of agents on the $H_\infty$ gain from noise $m$ to estimation error $\bar{e}$.}
\end{figure}  
}{
\begin{figure}[!h]
         \centering
         \subfigure[Graph structures with all to all connections (self connection links are not shown).]{
         \label{fig:alltoallconnectedgraph}
         \psfragfig[trim=0cm 0cm 0cm 0cm,clip,width=0.38\textwidth]{Jalltoall16graph}
         {
         \psfrag{N1}[][][0.6]{$N=1$}
         \psfrag{N2}[][][0.6]{$N=2$}
         \psfrag{N3}[][][0.6]{$N=3$}
         \psfrag{N4}[][][0.6]{$N=4$}
         \psfrag{N5}[][][0.6]{\qquad\ \ $N=5$}
         \psfrag{N6}[][][0.6]{\qquad\ \ $N=6$}
         \psfrag{N7}[][][1]{\ \ \ \ \quad $\vdots$}          
         }
         }
         \quad
         \subfigure[The global $H_\infty$ norm from noise $m$ to estimation error $\bar{e}$ with respect to the number of agents.]{
         \label{fig:saturation_global}
         \psfragfig[trim=1.5cm 0.8cm 1.3cm 0.7cm,clip,width=0.45\textwidth]{JFigsaturationglobal}
         {
         \psfrag{n}[][][0.6]{\hspace{-0.4in}$N$}
         \psfrag{H}[][][0.6][-90]{\hspace{-0.17in}$\gamma$}
         }
	\put(-170,70){\tiny\color{red}\bf big improvement}
	\put(-85,43){\tiny\color{mygreen}\bf little improvement}
         }         
         \caption{The influence of the number of agents on the $H_\infty$ gain from noise $m$ to estimation error $\bar{e}$.}
\end{figure}    }
\IfJtwoC{
\begin{table}[!h]
\setlength\tabcolsep{3pt}
\centering
\caption{Comparison of global $H_\infty$ norms from noise $m$ to $\bar{e}$ with different number of agents under all-to-all connection.}
\begin{tabular}{|l|c|c|c|c|c|c|c|}
\hline
\multirow{2}{*}{} & \multicolumn{7}{|c|} {number of agents ($N$)} \\ \cline{2-8}
& 1 & 2 & 3 & 4 & 5 & 6 & 7  \\ \hline 
global $H_\infty$      & 0.80 & 0.54 & 0.40 & 0.34 & 0.33 & 0.32 & 0.31  \\ \hline
improv. (\%)           & 0.00 & 32.5 & 50.0 & 57.5 & 58.8 & 60.0 & 61.3  \\ \hline
local $H_\infty$       & 0.80 & 0.38 & 0.24 & 0.23 & 0.21 & 0.20 & 0.19  \\ \hline
improv. (\%)           & 0.00 & 52.5 & 70.0 & 71.3 & 73.8 & 75.0 & 76.3  \\ \hline
\end{tabular}
\label{tab:global_numberofagent}
\end{table}
}{
\begin{table}[!h]
\tabcolsep=3.8pt
\centering
\caption{Comparison of global $H_\infty$ norms from noise $m$ to $\bar{e}$ with different number of agents under all-to-all connection.}
\begin{tabular}{|l|c|c|c|c|c|c|c|}
\hline
\multirow{2}{*}{} & \multicolumn{7}{|c|} {number of agents ($N$)} \\ \cline{2-8}
& 1 & 2 & 3 & 4 & 5 & 6 & 7  \\ \hline 
global $H_\infty$      & 0.80 & 0.54 & 0.40 & 0.34 & 0.33 & 0.32 & 0.31  \\ \hline
improv. (\%)           & 0.00 & 32.5 & 50.0 & 57.5 & 58.8 & 60.0 & 61.3  \\ \hline
local $H_\infty$       & 0.80 & 0.38 & 0.24 & 0.23 & 0.21 & 0.20 & 0.19  \\ \hline
improv. (\%)           & 0.00 & 52.5 & 70.0 & 71.3 & 73.8 & 75.0 & 76.3  \\ \hline
\end{tabular}
\label{tab:global_numberofagent}
\end{table} }
Note that the improvement is less significant for $N > 6$. In particular, the table indicates that if the global $H_\infty$ gain is required to be less than or equal to $0.40$, then, as shown in Figure~\ref{fig:saturation_global}, the least number of agents needed is three\footnote{The optimization problem related to the examples shown in this paper are solved by PENBMI \cite{03Michal.PENNON}.}. For the same scalar plant with three interconnected observers, according to Theorem~\ref{thm:minimizinglinks}, we establish a relationship between $\mbox{\rm tr}(D)$ and the global $H_\infty$ gain from $m$ to estimation error $\bar e$ in Table~\ref{tab:global_minimizing_trace}. In particular, for $\mbox{\rm tr}(D)$ smaller than six, there is no improvement in the $H_\infty$ gain when compared to that of Luenberger observers. Moreover, the table indicates that, with three interconnected observers, if the global $H_\infty$ gain is required to be less than or equal to $0.6$, then the minimum number of links required in the connectivity graph $\Gamma$ is seven.
\begin{table}[!h]
\tabcolsep=3.8pt
\centering
\caption{Comparison of global $H_\infty$ norms from noise $m$ to $\bar{e}$ with different connectivity graph with $N=3$.}
\begin{tabular}{|l|c|c|c|c|c|c|c|}
\hline
\multirow{2}{*}{} & \multicolumn{4}{|c|} {tr(D)} \\ \cline{2-5}
                                     & 6 & 7 & 8 & 9   \\ \hline 
global $H_\infty$          & 0.64   & 0.53    & 0.43    & 0.40    \\ \hline
improv. (\%)                  & 20.0   & 33.8    & 46.3    & 50.0   \\ \hline
\end{tabular}
\label{tab:global_minimizing_trace}
\end{table} 
\end{example}
\IfJtwoC{\vspace{-14pt}}{}
\subsection{A sufficient condition guaranteeing smaller local $H_\infty$ gain}
\label{sec:sufficientcondHinf}
In this section, we are interested in conditions on the plant \eqref{eq:plant} for which it is possible to design interconnected observers  that, for a given rate of convergence $\sigma$, have local $H_\infty$ gains  smaller than when a single Luenberger observer is used at each agent. Note that
the local $H_\infty$ gain affects the quality of the estimates obtained at each node.
These estimates can be computed efficiently and in a decentralized manner using local information, while computing the global estimate requires additional algorithms -- see \IfJtwoC{{\color{black}\cite[Section~\uppercase\expandafter{\romannumeral4}.B]{Li.Sanfelice.13.TR.Interconnected}}}{Section~\ref{subsec:consensus}}.
The following result provides one such condition.
\begin{theorem}
\label{thm:smaller_Hgain_guaranteed}
Given $\sigma \geq 0$, suppose $K_L$ is such that the eigenvalues of the error system \eqref{eq:singleob_error} of the Luenberger observer \eqref{eq:singleob} for the plant \eqref{eq:plant} are located in the region ${\cal D}=\{ s\in {\cal C}_0:\, Re(s) < -\sigma \}$, and the $H_\infty$ gain from $m$ to ${e}_L$ is $\gamma_L>0$.
If there exist $\tilde{\alpha} \in \mathbb{R}$ and $P=P^\top>0$ such that 
\begin{align}
\left[\begin{array}{ccc}
\mbox{\rm He}(A \!-\! K_L C,P)  & P K_L C & - \tilde{\alpha} I_n\\
C^\top K_L^\top P & - I_n & (1 + \tilde{\alpha}) I_n \\
-\tilde{\alpha} I_n & (1+ \tilde{\alpha})I_n & - I_n
\end{array}\right]
<0,
\label{eq:LMI_smaller_H_gain}
\end{align}
then, for every $N\in\mathbb{N},\ N>1$, there exist a digraph $\Gamma$ and a gain ${\cal K}$ for $N$  interconnected observers  in \eqref{eq:graph_individual} such that the error system \eqref{eq:general_error_compact_graph} has its eigenvalues in ${\cal D}$ and the local $H_\infty$ gain from $m$ to associated $\bar{e}_i$ for all agents are less than or equal to $\gamma_L$. 
Moreover, for at least $N-1$ agents, the local $H_\infty$ gain from $m$ to associated $\bar{e}_i$ is strictly less than $\gamma_L$.
\end{theorem}
\begin{proof}
For any $N>1$, let the digraph $\Gamma$ have adjacency matrix 
\begin{align}
G_N =\left[\begin{array}{cc} 
1 & 1_{N-1}^\top \\ 
0 &  I_{N-1}
\end{array}\right].  
\end{align}
This choice of $G$ indicates that agent $1$ can share information with all other agents. Moreover, for each $i\in {\cal V}$,
let $T_i$ be the transfer function from $m$ to $\bar{e}_i$.
Take $N= 2$ and $K_{11}= K_{22} = K_L$, $K_{12} = 0$, and $K_{21}$ to be determined later. Then, the interconnected  observers in \eqref{eq:graph_individual} reduce to
\IfJtwoC{
      \begin{align}
      \begin{split}
              \dot{\hat x}_1 &\!=\! A {\hat x}_1 - K_L({\hat y}_1 - y_1),\\
              \dot{\hat x}_2 & \!=\! A {\hat x}_2- K_L({\hat y}_2 - y_2) \!-\! K_{21}({\hat y}_1 \!-\! y_1), \\ 
              {\hat y}_1 &\!=\! C {\hat x}_1,\
              {\hat y}_2 \!=\! C {\hat x}_2 ,\quad 
              \bar{x}_1 \!=\! \hat x_1,\ \bar{x}_2 \!=\! \frac{{\hat x}_1+{\hat x}_2}{2}, 
      \end{split}\label{eq:multi-observer-reduced}
      \end{align}
}{
      \begin{align}
      \begin{split}
              \dot{\hat x}_1 &\!=\! A {\hat x}_1 -  K_L({\hat y}_1 \!-\! y_1), \quad 
              \dot{\hat x}_2  \!=\! A {\hat x}_2- K_L({\hat y}_2 - y_2) \!-\! K_{21}({\hat y}_1 \!-\! y_1), \\ 
              {\hat y}_1 &\!=\! C {\hat x}_1,\
              {\hat y}_2 \!=\! C {\hat x}_2 ,\quad 
              \bar{x}_1 \!=\! \hat x_1,\ \bar{x}_2 \!=\! \frac{{\hat x}_1+{\hat x}_2}{2}, 
      \end{split}\label{eq:multi-observer-reduced}
      \end{align} }
with associated error system as in \eqref{eq:general_error_compact_graph} with 
      \begin{flalign*}
      \begin{split}
              {\cal A}=\left[
              \begin{array}{cc}
              A-K_L C    &  0 \\
              -K_{21} C  &  A-K_L C
              \end{array}
              \right], \
              {\cal B}=\left[
              \begin{array}{cc}
               K_L     & 0  \\
               K_{21} & K_L
              \end{array}
              \right].            
      \end{split}
      \end{flalign*}
If $K_L$ is such that \eqref{eq:singleob} has its eigenvalues in ${\cal D} \!\!=\!\! \{ s \!\in\! {\cal C}_0:\, Re(s)\!<\! -\sigma \}$, then, due to the block matrix form of ${\cal A}$, the eigenvalues of ${\cal A}$ are also in ${\cal D}$.
Now, suppose \eqref{eq:LMI_smaller_H_gain} holds with $\alpha \in \mathbb{R}$ and $P=P^\top>0$. Then, if \eqref{eq:LMI_smaller_H_gain} is treated as an $H_\infty$ constraint, equivalently, we have 
\begin{align}
\left|\left| -\tilde\alpha(sI-\tilde{A}_L)^{-1}K_L C + (1+\tilde\alpha) I \right|\right|_\infty < 1.
\label{eq:H_infinity_norm_LMI_proof}
\end{align}
Therefore, the transfer function $T_2(s) \!=\! {\cal C}_2 (sI \!-\! {\cal A})^{\!-\!1} {\cal B}$ satisfies 
\IfJtwoC{
\begin{align*}
T_2 
 &= 
\frac{1}{2} \! \left[\begin{array}{cc}
\!\!\!I &\!\!\! I\!\!\!
\end{array}\right] 
\! \left[\!\! \begin{array}{cc}
sI - \tilde{A}_L \!\!& 0 \\
K_{21}C & sI-\tilde{A}_L
\end{array} \!\! \right]^{\!-\!1}
\!\! \left[ \!\!\! \begin{array}{cc}
K_L      & 0\\
K_{21} &K_L 
\end{array}\!\!\right] .
\end{align*}
}{
\begin{align*}
T_2 & =  {\cal C} (sI - {\cal A})^{-1} {\cal B} 
 = 
\frac{1}{2} \! \left[\begin{array}{cc}
\!\!\!I &\!\!\! I\!\!\!
\end{array}\right] 
\! \left[\!\! \begin{array}{cc}
sI - \tilde{A}_L \!\!& 0 \\
K_{21}C & sI-\tilde{A}_L
\end{array} \!\! \right]^{\!-\!1}
\!\! \left[ \!\!\! \begin{array}{cc}
K_L      & 0\\
K_{21} &K_L 
\end{array}\!\!\right] .
\end{align*}  }
By using the inversion identity for a block matrix (inversion lemma), it follows that 
\IfJtwoC{
\begin{align*}
 \left[\!\! \begin{array}{cc}
sI \!\!-\!\! \tilde{A}_L \!\!& 0 \\
K_{21}C & sI-\tilde{A}_L
\end{array} \!\! \right]^{\!-\!1}\!\! =\!\! 
\left[ \!\! \begin{array}{cc}
(sI \!\!-\!\! \tilde{A}_L)^{-1} & 0\\
F & (sI \!\!-\!\! \tilde{A}_L)^{-1}
\end{array} \!\! \right],
\end{align*}
}{
\begin{align*}
 \left[\!\! \begin{array}{cc}
sI - \tilde{A}_L \!\!& 0 \\
K_{21}C & sI-\tilde{A}_L
\end{array} \!\! \right]^{\!-\!1} = 
\left[ \!\! \begin{array}{cc}
(sI - \tilde{A}_L)^{-1} & 0\\
F & (sI-\tilde{A}_L)^{-1}
\end{array} \!\! \right],
\end{align*} }
where $F = - (sI - \tilde{A}_L)^{-1}K_{21}C(sI - \tilde{A}_L)^{-1}$
Then, by assigning $K_{21} = \tilde\alpha K_L$,  $T_2$ can be simplified as 
\begin{align*}
T_2  =  
 \left[\begin{array}{cc}
 \frac{1}{2} T_L  - \frac{1}{2}\tilde\alpha T_L C T_L +  \frac{1}{2}\tilde\alpha T_L & \frac{1}{2}T_L
 \end{array} \right],
\end{align*}
where $T_L(s) = (sI-\tilde{A}_L)^{-1} K_L$. Therefore, we obtain 
\begin{align*}
||T_2||_\infty & \leq \frac{1}{2} \left|\left| (1 + \tilde\alpha) T_L - \tilde\alpha T_L C T_L \right|\right|_\infty +\frac{1}{2} ||T_L||_\infty.
\end{align*}
Using \eqref{eq:H_infinity_norm_LMI_proof}, it follows that
$
||T_ 2||_\infty < ||T_L||_\infty = \gamma_L. $
Now consider for any $N >1,N\in\mathbb{N}$, with digraph whose adjacency matrix is $G_N$, it follows that the transfer function $T_{i}$ from noise $m$ to $\bar{e}_i$ satisfies $T_{i} = T_2$ for all $i \in {\cal V}, i\neq 1$. 
Therefore, $||T_{i}||_\infty < \gamma_L$ for all $i \in {\cal V}, i\neq 1$. 
\end{proof}
\IfJtwoC{\vspace{-6pt}}{}
Note that condition \eqref{eq:LMI_smaller_H_gain} is a property on the plant for a given $K_L$; basically, an $H_\infty$ inequality as in \eqref{eq:synthesis_cons_hinf_fixed}.  Next, we illustrate this condition in the examples throughout the paper.
\begin{example}
For the scalar plant \eqref{eq:scalarplant} with the Luenberger observer \eqref{eg:scalarplantob}, the transfer function in the $s$-domain from $m$ to $e_L$ is given by $T_L(s) = \frac{{K}_L}{s-a+{K}_L}$. 
{Since \eqref{eq:LMI_smaller_H_gain}} is an LMI with respect to $P$ and $\tilde\alpha$, its feasibility can be easily verified, {\it e.g.}, for $a=-0.5$ and $K_L=2$, $P=0.47$ and $\tilde\alpha = -0.5$ solve {\eqref{eq:LMI_smaller_H_gain}}. Therefore, for the plant \eqref{eq:scalarplant}, there exist interconnected observers such that at least $N-1$ local $H_\infty$ gains are smaller than $\gamma_L=0.8$ with $K_L = 2$. This justifies the improvement shown in the motivational example as in Table~\ref{tab:frequency}. 
\end{example}
\IfJtwoC{\vspace{-6pt}}{}
\begin{example}
We revisit the systems in Example~\ref{ex:second_order_plants}. For the first system discussed in Example~\ref{ex:second_order_plants}, 
the improvement is justified by the fact that condition \eqref{eq:LMI_smaller_H_gain} in Theorem~\ref{thm:smaller_Hgain_guaranteed} holds with $\tilde\alpha = -0.3241$ and $P = 0.1I$.
While for the other second-order plant with oscillatory behavior, the improvement is justified by the fact that condition \eqref{eq:LMI_smaller_H_gain} in Theorem~\ref{thm:smaller_Hgain_guaranteed} holds with $\tilde\alpha = 0.8631$ and $P = [w_1\ w_2]$ with $w_1 = [0.1839\ 0.0117]^\top$ and $w_2 = [0.0117\ 0.1722]^\top$. 
\end{example}
\IfJtwoC{While it may be possible to get further improvement by designing the gains of the interconnected observers as in the design of Kalman filters, it should be noted that the tradeoff between performance and robustness affects general Kalman filters; {\color{black} see \cite[Section \uppercase\expandafter{\romannumeral4}.C]{Li.Sanfelice.13.TR.Interconnected}} for a discussion on this.}{}
\NotForJtwoC{
\section{Discussion}
\label{sec:discussion}     
\subsection{Optimizing the number of internal observers $N$}
\label{sec:optimization_on_number_observers}
The parameters of the interconnected observers are the 
gains $K_{ij}$'s and the number of internal observers $N$.
The optimization problems formulated in Section~\ref{sec:Design} optimize the
gains for a prespecified value of $N$.
Clearly, the larger the value of $N$ is, the larger the size 
of the optimization problem to solve becomes.  
Moreover, it is expected that performance and robustness may only be improved
up to a certain value of $N$, and increasing it further would not lead to a significant
improvement.
For a fixed rate of convergence constraint, Example~\ref{ex:number_internal_obs} shows the relationship between the optimal $H_\infty$ gain from $m$ to $\bar{e}$ as a function of $N$ for the scalar plant \eqref{eq:scalarplant} with the interconnected observer via an all-to-all connectivity. Figure~\ref{fig:saturation_global} suggests that, when $N$ is larger than four, the improvement on the $H_\infty$ gain is not significant as its value tends to settle around a constant ($\approx 0.3$).
The objective function in the optimization problem in \eqref{eq:synthesis_cons_hinf_fixed} can be modified to include the number of internal observers $N$ as an optimization variable. The objective function is given by $c_1 \gamma_g + c_2 N$, where $c_1>0, c_2\geq 0$ are constant weights. With this new objective function, an optimization problem that optimizes the number of observers can be written as
\begin{align}
\begin{split}
& \inf c_1 \gamma_g + c_2 N \\
\mbox{s.t.:\ } & \mbox{\rm He}({\cal A},P_{S}) + 2\sigma P_{S} < 0,\\
& \left[\begin{array}{cccc}
     \mbox{\rm He}({\cal A},P_{S}) &    P_{H} {\cal B} & {\cal X}\\
     {\cal B}^\top P_{H} & -\gamma_g I & 0\\
     {\cal X}^\top & 0 & -\gamma_g I
    \end{array}\right] <0,\\
& P_S =P_S^\top >0, P_H= P_H^\top >0, N \in {\cal N},    
\label{eq:optimization_multi_observers_ob_n}
\end{split}
\end{align}
where ${\cal X} = {\cal C}$ (or ${\cal C}_i$), ${\cal N}\subset \mathbb{N}$ and ${\cal N}$ is bounded.
It is worth to note that when $c_2=0$, the optimization reduces to the original one in \eqref{eq:synthesis_cons}. 
A way to solve problem \eqref{eq:optimization_multi_observers_ob_n}, perhaps not efficiently, is by generating a collection of problems with $N$ taking values from the finite set ${\cal N}$ and then determining the one(s) with smallest $H_\infty$ gain.
\subsection{Consensus of the estimates in the nominal case}
\label{subsec:consensus}
The results in the previous section enable the design of interconnected observers as in \eqref{eq:graph_individual} that meet specifications involving the rate of convergence, $H_\infty$ gains, and connectivity graphs. The local estimate could further be employed to determine a global estimate over the connectivity graph.  Such an estimate can be obtained using consensus algorithms, in which case it will consists of a consensus problem of time-varying signals.
When measurement noise is zero, the algorithm in \cite{13.Cortes.Consensus} can already be employed when generalized to the case of vector inputs. To this end, we attach to each agent an agreement vector $\xi_i$ and employ the following distributed algorithm to guarantee that each $\xi_i$ asymptotically approaches the average of the local estimates, namely, $\frac{1}{N}\sum_{j=1}^N \hat x_j(t)$: 
\IfJtwoC{
\begin{align}
\begin{split}
\dot \xi_i^k & = - \beta_1 (\xi_i^k \!-\! \hat x_i^k) \!-\! \beta_2 \sum_{j=1}^N \!\! \ell_{ij} \xi_j^k \!-\! v_i^k \!+\! \dot{\hat x}_i^k,\\
\dot v_i^k   &= \beta_1 \beta_2 \sum_{j=1}^N \ell_{ij} \xi_i^k,
\end{split}
\label{eq:consensus_algorithm}
\end{align} 
}{ 
\begin{align}
\dot \xi_i^k & = - \beta_1 (\xi_i^k - \hat x_i^k) - \beta_2 \sum_{j=1}^N \ell_{ij} \xi_j^k - v_i^k + \dot{\hat x}_i^k,\qquad
\dot v_i^k   = \beta_1 \beta_2 \sum_{j=1}^N \ell_{ij} \xi_i^k,
\label{eq:consensus_algorithm}
\end{align}  }
for $i\in{\cal V}$, $1\leq k\leq n$, where $\xi_i =(\xi_i^1,\dots,\xi_i^k,\dots,\xi_i^n)$; $\hat x_i$'s are the estimates generated by agent $i$ using the local observer in \eqref{eq:graph_individual}, $v_i$ is the auxiliary variable, and $\ell_{ij}$'s are elements of the Laplacian $\cal{L}$ associated with the digraph $\Gamma$. The constants $\beta_1, \beta_2 \in \mathbb{R}$ are parameters to be determined.

To analyze the convergence and stability of algorithm \eqref{eq:consensus_algorithm}, following \cite{13.Cortes.Consensus}, it is rewritten as 
\IfJtwoC{
\begin{align}
\begin{split}
\dot \delta & = -\beta_1 \delta - \beta_2 ({\cal L}\otimes I_n) \delta - w,\\
\dot w  & = \beta_1 \beta_2 ({\cal L}\otimes I_n) \delta - \Pi_{nN}(\ddot{\hat x} + \beta_1 \dot{\hat x}), 
\end{split}
\label{eq:consensus_algorithm_equiv}
\end{align}
}{
\begin{align}
\dot \delta & = -\beta_1 \delta - \beta_2 ({\cal L}\otimes I_n) \delta - w,\qquad
\dot w   = \beta_1 \beta_2 ({\cal L}\otimes I_n) \delta - \Pi_{nN}(\ddot{\hat x} + \beta_1 \dot{\hat x}), \label{eq:consensus_algorithm_equiv}
\end{align} }
where $\delta = (\delta_1,\delta_2,\dots, \delta_N)$,  
$
\delta_i = \xi_i -\frac{1}{N} \sum_{j=1}^{N} \hat{x}_j, i\in{\cal V}$, and 
$w  = v - \Pi_{nN}(\dot{\hat x} + \beta_1 \hat x)$.  
Following \cite[Lemma 4.3]{13.Cortes.Consensus}, we obtain the following property.
\begin{lemma}
\label{lem:consensus}
For the plant in \eqref{eq:plant}, assume the digraph $\Gamma$ is strongly connected and weight balanced, where $\hat x_i$ has the dynamics given in \eqref{eq:graph_individual} with $m_i\equiv 0$. Moreover, assume there exists ${\cal K}$ in \eqref{eq:matrix_graph} such that ${\cal A}$ is Hurwitz.
Then, for any $x(0), \hat{x}_i(0), \xi_i(0)\in \mathbb{R}^n$, $\beta_1 > 0$, $\beta_2 > 0$, and $v_i(0)\in \mathbb{R}^n$ such that $\sum_{i=1}^N v_i(0) = 0$, we have $ \lim_{t\to \infty}\left( \xi_i(t) -\frac{1}{N}\sum_{j=1}^N \hat x_j(t) \right) =0$ for all $i\in {\cal V}$. 
\end{lemma}
\begin{proof}
The proof can be found in \IfJournal{\cite[Appendix  C]{Li.Sanfelice.13.TR.Interconnected}}{{Appendix~\ref{proof_lem_consensus}}}.
\end{proof}
When the measurement noise $m$ is not zero, due to the linear dynamics, we conjecture that the algorithm in \eqref{eq:consensus_algorithm} has an ISS like property with respect to $m$, similar to the ${\cal KL}$ bound in \eqref{neq:KLbounds_graph}. 
Along with Theorem~\ref{thm:smaller_Hgain_guaranteed},
such a property could potentially be used to 
characterize the improvement of the global $H_\infty$ guaranteed by the interconnected observers. 

\subsection{Comparison between interconnected observers and the optimal observer/Kalman-Bucy filter}
\label{subsec:comparison_obs}
It is well known that the Kalman-Bucy filter is the optimal observer that minimizes the mean square estimation error \cite{Kalman1961}.
For the plant \eqref{eq:plant} with $x\in \mathbb{R}^n$, it is given by
\begin{align}
\dot {\hat x}_K = A {\hat x}_K - K(t) (\hat y_K - y),
\label{eq:kalman_bucy_filter}
\end{align}
where $\hat x_K\in \mathbb{R}^n$ and $t\mapsto K(t)$ is the time-varying gain.
Defining the estimation error as $e_K:= {\hat x}_K - x$, the error system for the optimal observer/Kalman-Bucy filter is\footnote{If \eqref{eq:kalman_bucy_filter} is initialized with $\hat x_K(0)= \bar{x}_0:= \mathbb{E}\{ x(0) \}$, where $\mathbb{E}$
{denotes the expected value function},
then, according to
\cite[Theorem~$4.5$]{72tradeoff},
for any {positive definite symmetric weighting matrix function} $t\mapsto W(t)$, the gain that minimizes the objective function $\mathbb{E}\{ e_K(t)^\top W(t) e_K(t) \}$ for all $t\geq 0$ is $K(t) = Q(t) C^\top V_d^{-1}(t)$, where $t \mapsto Q(t)$ is the solution of
\begin{align}
\dot Q(t) = AQ(t) +Q(t) A^\top - Q(t)C^\top V_d^{-1}(t)CQ(t),
\end{align}
from the initial condition $Q(0):= \mathbb{E}\{ (x(0)-\bar{x}_0)(x(0)-\bar{x}_0)^\top \}$,
where
{$t \mapsto V_d(t)$}
is the {(positive definite for each $t\geq0$)} covariance {matrix} of the measurement noise $m$.}
\begin{align}
\dot e_K = (A-K(t)C)e_K + K(t)m.
\end{align}
While the expected value of the (weighted) norm of $e_K$ is minimized, the same trade off pointed out in Section~\ref{sec:introduction} for Luenberger observers plays a key role in the design of \eqref{eq:kalman_bucy_filter}. In fact, as  \cite[page 346]{72tradeoff} correctly points out, ``The optimal observer provides a compromise between the speed of state reconstruction and the immunity to observation noise.''
Moreover, the design of \eqref{eq:kalman_bucy_filter} does not permit incorporating other performance indexes, such as the rate of convergence. On the other hand, the interconnected observers proposed here exploit the connections among agents over a graph to relax the constraints imposed by the said trade off and permits the incorporation of multiple objectives in the design.
In fact,
the results in this paper show that \eqref{eq:kalman_bucy_filter} is not the optimal observer when performance specifications formulated in terms of eigenvalue constraints (relative to the optimal observer) are added. In this way, our results yield observers living in dimensions that are larger than that of the plant, leading to an approach in which the state estimation of systems in $\mathbb{R}^n$ is performed using algorithms in $\mathbb{R}^{nN}$. }
\IfJtwoC{\vspace{-9pt}}{}
\section{Conclusion}  
In contrast to standard observers for linear time-invariant systems,
interconnected observers have the
capability of attaining fast rate of convergence rate without necessarily jeopardizing robustness to measurement noise in the $H_\infty$ sense.
The comparison between ${\cal KL}$ bounds between interconnected and Luenberger observers 
leads to checkable conditions that can be used for design -- though potentially conservative.
When solved for specific systems, 
the stated feasibility and optimization problems lead to significant improvements,
when compared to single Luenberger observers.
Such improvement is guaranteed by the satisfaction of an LMI condition.
While the optimization of the number of internal observers and the connectivity graph are 
not necessarily linear and convex, numerical results for a particular plant indicate that 
the improvement obtained in robustness is significant only up to a finite number of such internal observers.
\IfJtwoC{\vspace{-9pt}}{}
\balance
\bibliography{AME549ProjReport,long,Biblio,RGS}
\bibliographystyle{IEEEtran}
\IfJournal{}{

\appendix
\section{Proof of Proposition~\ref{prop:issbound}}
\label{app:proof_prop_KL_bounds}
For a Hurwitz matrix ${\cal A}$ with distinct eigenvalues, we have the following properties \cite{03.book.Corless.Brazho}:
          (P2.1) $|\exp({\cal A}t)| \leq \kappa({\cal A})\exp({\alpha({\cal A})}t),\ \forall t \geq 0 $;
          (P2.2) Let $\Phi(t) = \exp({\cal A}t) $ for all $t\in \mathbb{R}_{\geq 0}$, then $|| \Phi ||_1 \leq \frac{ \kappa({\cal A})}{|\alpha({\cal A})|}$.
Then, the solution of system \eqref{eq:general_error_compact_graph}, given by
$       e(t) = \exp({\cal A} t) e(0) + \int_{0}^t \exp({\cal A}(t-\tau)){\cal B} m(\tau) d \tau$,
can be bounded for all $t\geq 0$ as
\IfJournal{
      \begin{flalign}
      \begin{split}
           |e(t)| & \!\leq\!  \left|\exp({\cal A} t)\right|\! |e(0)| \!\!+\!\!
            \left| \int_{0}^t \!\!\!\exp({\cal A}(t\!-\!\tau)) {\cal B} m(\tau) d \tau \!\right|\!. \\
      \end{split}\label{neq:bound_multiobs_1}
      \end{flalign}
}{
      \begin{flalign}
      \begin{split}
           |e(t)| & \leq  \left|\exp({\cal A} t)\right| \, |e(0)| +
            \left| \int_{0}^t \exp({\cal A}(t-\tau)) {\cal B} m(\tau) d \tau \right|. \\
      \end{split}\label{neq:bound_multiobs_1}
      \end{flalign} }
Let $\left|\phi(t)\right| = \left|\int_{0}^t \exp({\cal A}(t-\tau)) {\cal B} m(\tau) d \tau \right|$, then
\IfJournal{
      \begin{flalign}
      \begin{split}
         |\phi(t)|  &\!\leq\! | {\cal B} | \!\!\int_{0}^t  \left| \exp({\cal A}(t-\tau)) \right| d \tau \, |m|_\infty \\
                    &\!\leq\! | {\cal B} | \!\!\int_{0}^t  \left| \exp({\cal A}(\tau)) \right| d\tau \! |m|_\infty 
                     \!\leq\! | {\cal B} | \! || \Phi ||_1 \! |m|_\infty.
      \end{split}
      \end{flalign}
}{
      \begin{flalign}
      \begin{split}
         |\phi(t)|  &\leq | {\cal B} | \int_{0}^t  \left| \exp({\cal A}(t-\tau)) \right| d \tau \, |m|_\infty 
                         \leq | {\cal B} | \int_{0}^t  \left| \exp({\cal A}(\tau)) \right| d\tau \, |m|_\infty 
                         \leq | {\cal B} |\, || \Phi ||_1 \, |m|_\infty.
      \end{split}
      \end{flalign} }
Therefore,
by using properties P2.1 and P2.2, inequality \eqref{neq:bound_multiobs_1} can be simplified as
\IfJtwoC{
      \begin{flalign}
      \begin{split}
           |e(t)| & \!\leq\! \kappa({\cal A}) \exp({\alpha}({\cal A}) t) \! |e(0)|  \!+\!  \kappa({\cal A}) \frac{ |{\cal B}| }{| {\alpha}({\cal A}) |}|m|_\infty.
       \end{split}\label{neq:bound_multiobs_2}
       \end{flalign}
}{
      \begin{flalign}
      \begin{split}
           |e(t)| & \leq \kappa({\cal A}) \exp({\alpha}({\cal A}) t) \, |e(0)|  +  \kappa({\cal A}) \frac{ |{\cal B}| }{| {\alpha}({\cal A}) |}|m|_\infty.
       \end{split}\label{neq:bound_multiobs_2}
       \end{flalign} }
Furthermore, 
$           |{\bar e}(t)| = |{\cal C} e(t)| \leq |{\cal C}| |e(t)| 
            \leq \kappa({\cal A}) |{\cal C}| \exp({\alpha}({\cal A}) t) \, |e(0)|  +  \kappa({\cal A}) \frac{ |{\cal B}| |{\cal C}|}{| {\alpha}({\cal A}) |}|m|_\infty. $   
Pick for $s,\ t\in \mathbb{R}_{\geq 0}$,
$         \beta (s, t) = \kappa({\cal A}) |{\cal C}| \exp( {\alpha}({\cal A}) t) s,\
         \varphi (s)=  \kappa({\cal A}) \frac{|{\cal B}||{\cal C}|}{| {\alpha} ({\cal A}) |}s. $
It follows that \eqref{neq:KLbounds_graph} holds.  

When ${\cal A}$ is dissipative such that ${\cal A}^\top + {\cal A} \leq - 2 \overline \alpha I$ for some $\overline \alpha > 0$, following \cite[Section 3.2]{03.book.Corless.Brazho}, we have 
          (P2.3) ${\cal A} + {\cal A}^\top \leq 2\mu({\cal A}) I$;
          (P2.4) $|\exp({\cal A}t)| \leq \exp(\mu ({\cal A})t)$ for all $t \geq 0 $;
          (P2.5) Let $\Phi(t) = \exp({\cal A}t) $ for all $t \geq 0$. Then, $|| \Phi ||_1 \leq \frac{1}{|\mu ({\cal A})|}$.
Using properties P2.4 and P2.5, inequality \eqref{neq:bound_multiobs_1} can be simplified as
$           |e(t)|  \leq \exp(\mu({\cal A}) t) \, |e(0)|  +   \frac{ |{\cal B}| }{| \mu({\cal A}) |}|m|_\infty. $
Then,  if follows that 
$          |{\bar e}(t)| = |{\cal C} e(t)| \leq |{\cal C}| |e(t)| 
           \leq |{\cal C}|\exp(\mu({\cal A}) t) \, |e(0)|  +   \frac{ |{\cal B}| |{\cal C}|}{| \mu({\cal A}) |}|m|_\infty.$
For each $s\in \mathbb{R}_{\geq 0}$ and $t\in \mathbb{R}_{\geq 0}$, define
$         \beta (s, t)=  |{\cal C}| \exp( \mu({\cal A}) t) s,\ 
         \varphi(s)= \frac{|{\cal B}| |{\cal C}| }{| \mu ({\cal A}) |}s. $
It follows that \eqref{neq:KLbounds_graph} holds.    

When there exists $P=P^\top>0$ such that ${\cal A}^\top P + P {\cal A} \leq - 2 \overline\alpha P$ for some $\overline\alpha > 0$,   
consider the Lyapunov function $V(e) = e^\top P e$. Then, $V$ has the following properties:
(P2.6) $\lambda_{\min}(P)|e|^2 \leq V(e) \leq \lambda_{\max}(P)|e|^2$;
(P2.7) $\langle \nabla V (e), {\cal A}e \rangle \leq -2 {\overline\alpha} \lambda_{\min} (P) |e|^2$;
(P2.8) $|\nabla V (e)| \leq  2\lambda_{\max}(P) |e| $.
Moreover, the derivative of the function $V(e) = e^\top P e$ with respect to time is, for each $e\in\mathbb{R}^n$,
\IfJtwoC{
      \begin{flalign}
      \begin{split}
            \dot V (e) & = \langle \nabla V (e), \dot {e} \rangle \\
                       & = \langle \nabla V (e), {\cal A}e + {\cal B}m \rangle \\
                      & = \langle \nabla V (e), {\cal A}e \rangle + \langle \nabla V (e), {\cal B}m \rangle.
      \end{split}\label{eqn:timederivative_Lya}
      \end{flalign}
}{
      \begin{flalign}
      \begin{split}
            \dot V (e) & = \langle \nabla V (e), \dot {e} \rangle 
                        = \langle \nabla V (e), {\cal A}e + {\cal B}m \rangle 
                       = \langle \nabla V (e), {\cal A}e \rangle + \langle \nabla V (e), {\cal B}m \rangle.
      \end{split}\label{eqn:timederivative_Lya}
      \end{flalign} }
Using properties {P2.7} and {P2.8} as well as the Cauchy-Schwarz inequality, we get that for each solution $t\mapsto e(t)$ to  \eqref{eq:general_error_compact_graph} and each $t\mapsto m(t)$
\IfJtwoC{
      \begin{flalign}
      \begin{split}
            \dot V(e(t)) & \leq -2 {\overline\alpha} \lambda_{\min} (P) |e(t)|^2\\
             & \quad + 2\lambda_{\max}(P) |e(t)| |{\cal B}| |m(t)|
      \end{split}\label{eq:bound_lyapunov_intermidiate}
      \end{flalign}
}{
      \begin{flalign}
      \begin{split}
            \dot V(e(t)) & \leq -2 {\overline\alpha} \lambda_{\min} (P) |e(t)|^2 + 
             2\lambda_{\max}(P) |e(t)| |{\cal B}| |m(t)|
      \end{split}\label{eq:bound_lyapunov_intermidiate}
      \end{flalign} }
      for all $t \geq 0$.
To claim the desired bound on $e(t)$ from \eqref{eq:bound_lyapunov_intermidiate}, using similar steps as those in the proof of \cite[Theorem $5.1$]{Khalil.2002.nonlinearsystems}, we define $W(t) = \sqrt{V(e(t))}$. 
It can be shown that,  for all values of $V(e(t))$,
      \begin{flalign*}
      \begin{split}
            D^+ W(t) \leq - \frac{ {\overline\alpha} \lambda_{\min} (P)}{\lambda_{\max}(P)} W(t) + \frac{\lambda_{\max}(P)|{\cal B}|}{ \sqrt{\lambda_{\min}(P)}}|m(t)|.
      \end{split}
      \end{flalign*}
Then, by a comparison lemma (see, {\it e.g.}, \cite[Lemma 3.4]{Khalil.2002.nonlinearsystems}), for all $t\geq 0$, $W(t)$ satisfies the inequality
\IfJtwoC{
      \begin{flalign*}
      \begin{split}
             W(t) \leq & \exp(- \lambda t) W(0)\\ 
             & + \frac{\lambda_{\max}(P)|{\cal B}|}{ \sqrt{\lambda_{\min}(P)}} \!\!\int_0^t \!\! \exp(-\lambda(t-\tau))|m(\tau)| d\tau,
      \end{split}
      \end{flalign*}
}{
      \begin{flalign*}
      \begin{split}
             W(t) \leq & \exp(- \lambda t) W(0) + 
              \frac{\lambda_{\max}(P)|{\cal B}|}{ \sqrt{\lambda_{\min}(P)}}\int_0^t \exp(-\lambda(t-\tau))|m(\tau)| d\tau,
      \end{split}
      \end{flalign*} }
where $\lambda = \frac{{\overline\alpha} \lambda_{\min} (P)}{\lambda_{\max}(P)}$ . Using property {P2.6} and $W(t) = \sqrt{V(e(t))}$, it follows that for all $t > 0$,
\IfJtwoC{
      \begin{flalign}
      \begin{split}
             |e(t)| & \!\leq\!\! \sqrt{\!\frac{\lambda_{\max}(\!P\!)}{\lambda_{\min}(\!P\!)}}|e(0)|\!\exp( \!-\! \lambda t)  \!+\! 
                     \frac{\lambda_{\max}(\!P\!)|{\cal B}|}{ \lambda_{\min}(\!P\!) |\lambda|} |m|_\infty.
      \end{split}\label{neq:bounds_multiob_prop3}
      \end{flalign}
}{
      \begin{flalign}
      \begin{split}
             |e(t)| 
                     & \leq \sqrt{\frac{\lambda_{\max}(P)}{\lambda_{\min}(P)}}|e(0)|\exp(- \lambda t)  + 
                     \frac{\lambda_{\max}(P)|{\cal B}|}{ \lambda_{\min}(P) |\lambda|} |m|_\infty.
      \end{split}\label{neq:bounds_multiob_prop3}
      \end{flalign} }
Then, it follows that 
\IfJtwoC{
      \begin{flalign}
      \begin{split}
           & |{\bar e}(t)| = |{\cal C} e(t)| \leq |{\cal C}| |e(t)| \\
                &     \!\!\!\leq\! \sqrt{\!\frac{\lambda_{\max}(\!P\!)}{\lambda_{\min}(\!P\!)}}|{\cal C}||e(0)|\exp(\!-\! \lambda t)  \!+\! 
                     \frac{\lambda_{\max}(P)|{\cal B}| \! |{\cal C}|}{ \lambda_{\min}(P) |\lambda|} |m|_\infty.
      \end{split}\label{neq:bounds_multiob_prop3_bare}
      \end{flalign} 
}{
      \begin{flalign}
      \begin{split}
            |{\bar e}(t)| = |{\cal C} e(t)| \leq |{\cal C}| |e(t)| 
                     & \leq \sqrt{\frac{\lambda_{\max}(P)}{\lambda_{\min}(P)}}|{\cal C}||e(0)|\exp(- \lambda t)  + 
                     \frac{\lambda_{\max}(P)|{\cal B}||{\cal C}|}{ \lambda_{\min}(P) |\lambda|} |m|_\infty.
      \end{split}\label{neq:bounds_multiob_prop3_bare}
      \end{flalign}      }
For every $s,\ t\in \mathbb{R}_{\geq 0}$, define
$         \beta (s, t)= \sqrt{\frac{\lambda_{\max}(P)}{\lambda_{\min}(P)}}|{\cal C}|\exp(- \lambda t) s, 
         \varphi(s) =  \frac{\lambda_{\max}(P)|{\cal B}||{\cal C}|}{ \lambda_{\min}(P) |\lambda|} s.$
Then, \eqref{neq:KLbounds_graph} holds.

\section{Proof of Proposition~\ref{prop:example_Hinf}}
\label{app:proof_prop_exmaple_Hinf}
Let $K_{11} = K_{22} = K_L$. Define the function $\rho_1: [0,\infty) \to [0,\infty)$ as $\rho_1(x) = 4||T(j\omega)||^2$, where $x=\omega^2$ for all $\omega\in\mathbb{R}$. Then, 
\begin{align}
\begin{split}
\rho_1(x)&= \frac{(2K_L+K_{12}+K_{21})^2\omega^2 + [2K_L(K_L-a)-a(K_{12} + K_{21})-2K_{12}K_{21}]^2}{\omega^4+[2(a-K_L)^2+2 K_{12}K_{21}]\omega^2 + [(a-K_L)^2 - K_{12}K_{21}]^2}
:= \frac{\tilde{a}x+e}{bx^2+cx+d}, 
\label{eq:functionrho}
\end{split}
\end{align}
where $\tilde{a} = (2K_L+\alpha)^2$, $b = 1$, $c=2(a-K_L)^2+2\beta$, $d=[(a-K_L)^2-\beta]^2$, $e=[2K_L(K_L-a)-a\alpha-2\beta]^2$, where $\alpha = K_{12}+K_{21}$ and $\beta = K_{12}K_{21}$. 
It can be proved that the function $\omega \mapsto \rho_1(x(\omega))$ is pseudo-concave on $[0,\infty)$ for $(K_{12},K_{21})\in{\cal D}_1$,  where the nonempty set ${\cal D}_1 = \{(K_{12},K_{21})\in\mathbb{R}^2: (\alpha,\beta)\in{\cal S}_1\bigcap{\cal S}_2\}$ with 
${\cal S}_1 = \{(\alpha,\beta)\in \mathbb{R}^2: ec-\tilde{a} d >  0\}$
and ${\cal S}_2 = \{(\alpha,\beta)\in\mathbb{R}^2:\beta=0\}$;  see Lemma~\ref{lemma_pseudoconvex} (for convenience, in the definition of ${\cal D}_1$, we write the condition on $K_{12}$ and $K_{21}$ in terms of $\alpha$ and $\beta$). Moreover, $\rho_1$ is an even function on $\mathbb{R}$, with maximum attained at each $\omega$ such that $\nabla\rho_1(\omega^2)2\omega=0$. Therefore, by properties of extrema of pseudo-concave functions (see, e.g., \cite[page106]{Bazaraa2006}), its maximum is at $\omega = 0$ or at $\omega$ such that 
\begin{align}
\begin{split}
\nabla\rho_1(\omega^2) &= \frac{-\tilde{a}\omega^4-2e\omega^2+\tilde{a}d-ec}{(\omega^4+c\omega^2+d)^2} =0. \label{derivative_omega}
\end{split}
\end{align}
Since, using the assumption $a-K_L<0$, $c$ and $d$ are positive on ${\cal D}_1$, equation \eqref{derivative_omega} is equivalent to $\tilde{a}\omega^4+2e\omega^2-\tilde{a}d+ec=0$, which for $\alpha > -2K_L$ has roots at 
\begin{align}
\begin{split}
\omega_{1,2}^\star &=  \pm \sqrt{\frac{-e+\sqrt{e^2-\tilde{a}(ec-\tilde{a}d)}}{\tilde{a}}}, \quad
\omega_{3,4}^\star = \pm \sqrt{\frac{-e-\sqrt{e^2-\tilde{a}(ec-\tilde{a}d)}}{\tilde{a}}}. 
\label{eq:derivative_omega_roots}
\end{split}
\end{align}
Recall that $\tilde{a} >0$ and that on the set ${\cal D}_1$ we have $ec - \tilde{a}d >0$, which, since $d > 0$ due to $a - K_L <0$, implies $e > 0$. Therefore, the roots in \eqref{eq:derivative_omega_roots} are complex conjugate. Then, for $(K_{12},K_{21})\in{\cal D}_1 \cap \{(K_{12},K_{21})\in\mathbb{R}^2: \alpha>-2K_L\}$, $\rho_1$ attains maximum at $\omega=0$, and so does $\omega\mapsto ||T(j\omega)||^2$, $i.e.$, $||T||_\infty^2 = ||T(0)||^2 =  \frac{1}{4}  \frac{e}{d}$. 

Now, we show the existence of parameters $K_{11}, K_{22}, K_{12}, K_{21}$ for which the property $||T||_\infty < ||T_0||_\infty$ holds. We claim that it holds for $K_{11} = K_{22} = K_L$, $(K_{12},K_{21})\in {\cal D}:={\cal D}_1 \cap {\cal D}_2$, where ${\cal D}_2:=  \{(K_{12}, K_{21})\in\mathbb{R}^2: 0>K_{12}> \max\{-2K_L, \frac{4K_L(K_L-a)}{a}\}, K_{21} = 0 \} $. 
Take $(K_{12},K_{21})\in {\cal D}$. 
We have that $|| T(j\omega) ||^2$ is given by $\frac{1}{4}$ of the right-hand side of \eqref{eq:functionrho} with 
$\tilde{a} = (2K_L + K_{12})^2$, $b = 1$, $c=2(a-K_L)^2$, $d=(a-K_L)^4$, $e=[2 K_L (K_L-a)-a K_{12}]^2$. 
Then, 
\begin{align}
||T||_\infty^2 = \frac{1}{4}  \frac{e}{d} = \frac{1}{4}\frac{[2 K_L( K_L-a)-a K_{12}]^2}{(a- K_L)^4}.
\end{align} 
Furthermore\footnote{Since $T_L(s)= \frac{K_L}{s-a+K_L}$ and $||T_L(j\omega)||=\frac{K_L}{\sqrt{\omega^2+(K_L-a)^2}}$, it follows that 
$||T_L||_\infty=\sup_{\omega\in\mathbb{R}}||T_L(j\omega)||=\frac{K_L}{K_L-a}$.}, the inequality $||T||_\infty^2 < ||T_L||_\infty^2$ leads to 
\begin{align}
\frac{1}{4}\frac{[2 K_L (K_L-a)-a K_{12} ]^2}{(a-K_L)^4} < \frac{K_L^2}{(a-K_L)^2},
\label{eq:normsquareneq}
\end{align} 
which holds on ${\cal D}$. 

On the other hand, since $K_{11}=K_{22} =K_L$, when  $(K_{12},K_{21})\in {\cal D}$, the rate of convergence of the observer in \eqref{eq:coupled_obs} is $|a - K_L|$ by substituting these parameters in \eqref{eq:coupled_obs_full}, which coincides with that of the Luenberger observer in \eqref{eg:scalarplantob}.

Note that the case where $a>0$ and $a-K_L<0$ can be proved similarly. 

\begin{lemma}
\label{lemma_pseudoconvex}
For system \eqref{eq:coupled_obs} with $K_1=K_2=K_L$ such that $a-K_L<0$, whose transfer function from $m$ to $\bar{e}$ is $T(s) = \tilde{C}(sI - \tilde{A})^{-1} \tilde{B}$, for each $(K_{12},K_{21})\in{\cal D}_1:= \{(K_{12},K_{21})\in\mathbb{R}^2: [2a^2(a-K_L)^2-(a-K_L)^4](K_{12} + K_{21})^2+[8a K_L(a-K_L)^3-4 K_L(a-K_L)^4](K_{12}+K_{21})+4 K_L^2(a-K_L)^4 >  0, K_{12}K_{21}=0 \}$, the function $\omega\mapsto||T(j\omega)||^2$ is pseudo-concave on $[0,\infty)$.  
\end{lemma}
\begin{proof}
With the definition of $\tilde{A},\tilde{B},\tilde{C}$ and $K_1=K_2=K_L$, the transfer function $T$ can be computed as 
\begin{align*}
T(s) &= \tilde{C}(sI - \tilde{A})^{-1} \tilde{B}\\
       &= \frac{1}{2} \frac{(2K_L+K_{12}+K_{21})s+2 K_L (K_L-a) - a(K_{12} + K_{21}) - 2K_{12} K_{21}}{s^2-2(a - K_L)s+(a - K_L)^2 -K_{12} K_{21}}.
\end{align*}
Let $\alpha = K_{12} + K_{21}$ and $\beta = K_{12} K_{21}$, then we have 
\begin{align*}
T(s) = \frac{1}{2} \frac{(2 K_L+\alpha)s + 2 K_L (K_L-a)-a\alpha-2\beta}{s^2-2(a-K_L)s+(a-K_L)^2-\beta}.
\end{align*}
It follows that 
\begin{align*}
||T(j\omega)||^2= \frac{1}{4} \frac{(2 K_L+\alpha)^2\omega^2+[2 K_L(K_L-a)-a\alpha-2\beta]^2}{\omega^4+[2(a-K_L)^2+2\beta]\omega^2+[(a-K_L)^2-\beta]^2}.
\end{align*}
Note that pseudo-concavity of $\omega \mapsto ||T(j\omega)||^2$ is equivalent to pseudo-convexity of $\omega \mapsto -||T(j\omega)||^2$. Define the function $\rho: [0,\infty) \mapsto (-\infty,0]$ as $\rho(x) = -4||T(j\omega)||^2$, where $x=\omega^2$ for all $\omega\in\mathbb{R}_{\geq0}$. Note that since $\omega \mapsto ||T(j\omega)||^2$ is an even function on $\mathbb{R}$, it is enough to discuss the case where $\omega \geq 0$. 
Furthermore, let $\tilde{a} = (2K_L+\alpha)^2$, $b = 1$, $c=2(a-K_L)^2+2\beta$, $d=[(a-K_L)^2-\beta]^2$, $e=[2K_L(K_L-a)-a\alpha-2\beta]^2$, which will be treated as functions of $\alpha$ and $\beta$, but for simplicity of notation, we do not explicitly write that dependency. Then, 
\begin{align*}
\rho(x)= - \frac{\tilde{a}x+e}{bx^2+cx+d}.
\end{align*}
For any two points $x_1,x_2\in\mathbb{R}_{\geq 0}$, $\nabla \rho(x_1)(x_2-x_1)\geq 0 $ can be written as 
\begin{align*}
-\frac{-\tilde{a}x_1^2-2ex_1+\tilde{a}d-ec}{(bx_1^2+cx_1+d)^2}(x_2-x_1)\geq 0, 
\end{align*}
which implies that
\begin{align}
(\tilde{a}x_1^2+2ex_1+ ec-\tilde{a}d)(x_2-x_1)\geq 0.  \label{assumption_pconvex}
\end{align}
Since $\tilde{a} \geq 0$ and $e \geq 0$, then for each $(\alpha,\beta) \in {\cal S}_1:=\{(\alpha,\beta)\in\mathbb{R}^2: ec - \tilde{a}d > 0\}$, \eqref{assumption_pconvex} implies that $x_2 \geq x_1$. Now, $\rho(x_2)-\rho(x_1)$ can be evaluated as 
\begin{flalign}
\rho(x_2)-\rho(x_1) &= -\frac{\tilde{a}x_2+e}{x_2^2+cx_2+d}+\frac{\tilde{a}x_1+e}{x_1^2+cx_1+d}\\
                               & =\frac{\tilde{a}x_1x_2+e(x_1+x_2)+ce-\tilde{a}d}{(x_2^2+cx_2+d)(x_1^2+cx_1+d)}(x_2-x_1)
\end{flalign}
Recall that the minimum of quadratic function $x_1^2+cx_1+d$ on $\mathbb{R}$ is attained at the point $x_1 =-\frac{c}{2}$ and that the actual value of the function is $-\frac{c^2}{4}+d$. Therefore, for $(\alpha,\beta)\in \tilde{\cal S}_2:=\{(\alpha,\beta)\in\mathbb{R}^2: -\frac{c^2}{4}+d > 0\}$, $(x_2^2+cx_2+d)(x_1^2+cx_1+d) > 0$. Moreover, if $(\alpha,\beta)\in {\cal S}_1$, using the property $x_1\geq 0$, $\tilde{a}x_1x_2+e(x_1+x_2)+ce-\tilde{a}d>0$ since it is lower bounded by $\tilde{a}x_1^2+ex_1+ce-\tilde{a}d$. Then, for any $(\alpha,\beta)\in {\cal S}_1\bigcap \tilde{\cal S}_2$, $\nabla \rho(x_1)(x_2-x_1)\geq 0$ implies $\rho(x_2)-\rho(x_1) \geq 0$. Therefore, by definition of pseudo-convexity, the function $\rho$ is pseudo-convex on $[0,\infty)$ for each $(\alpha,\beta)\in {\cal S}_1\bigcap \tilde{\cal S}_2$. 

To show that the set ${\cal S}_1\bigcap \tilde{\cal S}_2$ is nonempty, consider the special case where $\beta = 0$. Using the definitions of $c$ and $d$, the set $\tilde{\cal S}_2$ leads to the smaller set ${\cal S}_2: = \{(\alpha,\beta)\in\mathbb{R}^2:\beta=0\}$. By using the definitions of $\tilde{a}, c,d,e$, $(\alpha,\beta)\in S_1\bigcap S_2$ implies that $\alpha$ should satisfy 
$[2a^2(a-K_L)^2-(a-K_L)^4]\alpha^2+[8a K_L(a-K_L)^3-4 K_L(a-K_L)^4]\alpha+4 K_L^2(a-K_L)^4 > 0$. This condition can always be satisfied for some $\alpha$ since $4 K_L^2(a-K_L)^4 \geq 0$. Thus, ${\cal S}_1 \bigcap {\cal S}_2$ is nonempty, which implies that $S_1\bigcap \tilde{S}_2$ is nonempty. Note that ${\cal D}_1= \{(K_{12},K_{21}): (\alpha,\beta)\in S_1 \bigcap S_2 \}$. Then, for each $(K_{12},K_{21})\in {\cal D}_1$,  $-\rho$ is pseudo-concave on $[0,\infty)$. Moreover, it can be easily verified that the composition $\rho(x(\omega))$  is pseudo-convex on the set $\{\omega\in\mathbb{R}: \omega \geq 0\}$, where $x(\omega) = \omega^2$. In fact, since 
\begin{align}
\frac{d}{d\omega}\rho(x(\omega)) = 2 \nabla \rho(x)\omega,
\end{align}
and for $\omega_1,\omega_2\in \mathbb{R}_{\geq 0}$,  $\rho(x(\omega_2))-\rho(x(\omega_1))=\rho(x_2)-\rho(x_1)$ with $x_i = \omega_i^2$ for $i\in\{1,2\}$, by similar arguments as above, we have that
\begin{align}
\frac{d}{d\omega}\rho(x(\omega_1))(\omega_2-\omega_1)\geq0 
\end{align}
implies $\rho(x(\omega_2))-\rho(x(\omega_1))\geq0$. Thus, for each $(K_{12},K_{21})\in {\cal D}_1$, $\omega \mapsto \rho(x(\omega))$ is pseudo-convex on $[0,\infty)$; hence, for each $(K_{12},K_{21})\in {\cal D}_1$,
 $\omega \mapsto ||T(j\omega)||^2$ is pseudo-concave on $[0,\infty)$.
\end{proof}

\section{Proof of Lemma~\ref{lem:consensus}}
\label{proof_lem_consensus}
Consider the $k$-th element of $\hat x_i$ as $\hat x_i^k$ with the algorithm in \eqref{eq:consensus_algorithm}. Let $\delta^k = (\delta_1^k, \delta_2^k,\dots,\delta_N^k)$, $w^k = (w_1^k, w_2^k,\dots,w_N^k)$, $v^k = (v_1^k, v_2^k,\dots,v_N^k)$ and $\hat x^k = (\hat x_1^k, \hat x_2^k,\dots,\hat x_N^k)$. We can rewrite \eqref{eq:consensus_algorithm} as
\begin{align}
\dot \delta^k & = -\beta_1 \delta^k - \beta_2 {\cal L} \delta^k - w^k,\\
\dot w^k  & = \beta_1 \beta_2 {\cal L} \delta^k - \Pi_{N}(\ddot{\hat x}^k + \beta_1 \dot{\hat x}^k),
\end{align}
where 
\begin{align}
\delta_i^k &= \xi_i^k -\frac{1}{N} \sum_{j=1}^{N} \hat{x}_j^k, i\in{\cal V},\quad 
w^k  = v^k - \Pi_{N}(\dot{\hat x}^k + \beta_1 \hat x^k).
\end{align}
Then, we obtain 
\begin{align}
\left[\begin{array}{l}
\dot \delta^k\\
\dot w
\end{array}\right] = A_k 
\left[\begin{array}{l}
\delta^k\\
w
\end{array}\right] - 
\left[\begin{array}{l}
0\\
\Pi_{N}(\ddot{\hat x}^k + \beta_1 \dot{\hat x}^k)
\end{array}\right], 
\end{align}
where 
\begin{align}
A_k = 
\left[\begin{array}{cc}
-\beta_1 I_N - \beta_2 {\cal L} & -I_N \\
\beta_1\beta_2 {\cal L} & 0
\end{array}\right].
\end{align}
Applying \cite[Lemma 4.3]{13.Cortes.Consensus}, we have $\lim_{t\to \infty}\left( \xi_i^k(t) -\frac{1}{N}\sum_{j=1}^N \hat x_j^k(t) \right) =0$ for all $i\in {\cal V}$. Furthermore, this is true for all $k\in\{1,2,\dots, n\}$. Therefore, the claim in Lemma is proved.
\section{On two uncoupled Luenberger observers}
\label{app:argumenton2uncoupledobs}
It should be noted that simply using two Luenberger observers without any coupling and taking the average of their estimates will not lead to both faster convergence rate and smaller steady state error. 
In fact, when $K_{12} = K_{21} =0$, the 
system in \eqref{eq:coupled_obs} proposes a structure of two Luenberger observers without coupling.
A direct calculation shows that for constant noise, the estimation error $\bar e$ satisfies
\IfConf{${\bar e}^\star = \frac{1}{2}\left(\frac{k_1}{k_1-a}+\frac{k_2}{k_2-a}\right)m.$}
\NotForConf{$${\bar e}^\star = \frac{1}{2}\left(\frac{K_1}{K_1-a}+\frac{K_2}{K_2-a}\right)m.$$}
Suppose the gain for the Luenberger observer in \eqref{eq:singleob_error} is $K_L\geq 0$ and $a-K_L<0$ for stability, where $a-K_L$ denotes the rate of convergence of the Luenberger observer. To guarantee stability and that the rate of convergence of the proposed observer is no worse than that of a  Luenberger observer, it is necessary to have $K_1\geq K_L$ and $K_2\geq K_L$. It can be easily verified that in such a case
     \begin{flalign}
     \begin{split}
         \frac{1}{2}\left(\frac{K_1}{K_1-a}+\frac{K_2}{K_2-a}\right)- \frac{K_L}{K_L-a} \geq 0.
     \end{split}
     \end{flalign}
Thus, $\left| {\bar e}^\star \right|\geq \left| e_L^\star \right|$ as $\frac{1}{2}\left(\frac{K_1}{K_1-a}+\frac{K_2}{K_2-a}\right)>0$ and $\frac{K_L}{K_L-a}>0$.

\section{Bound of $H_\infty$ gain as an inequality constraint}
To guarantee that the $H_\infty$ from $m$ to $\bar{e}$ is as small as possible, the bound of the transfer function $T$ in the $s$-domain should be minimized, namely, we look for the minimum $\gamma> 0$ such that $|T(j\omega)| < \gamma$ for all $\omega \in \mathbb{R}$, namely, we minimize the ${\cal L}_2$ gain.  \footnote{Such a bound guarantees that
$\int_0^\infty|z_\infty(t)|^2 dt  < \gamma^2 \int_0^\infty|m(t)|^2 dt$,
and $\gamma$ is the  ${\cal L}_2$ gain, where $m \in {\cal L}_2$, the so-called $H_\infty$ gain \cite{Tamer.1995.Hoptimal.book}.  }

\begin{lemma}
\label{Hinf_lemma}
\cite{anderson.1973.network}\cite[Theorem 2.41]{90.scherer.dssertation} 
For the transfer function \eqref{eq:tf_closedsys} defined  by $({\cal A},\ {\cal B},\ {\cal C},\ {\cal D})$, the following statements are equivalent.
\begin{enumerate}
\item [a)] The system is stable and the $H_\infty$ gain of the  system is less than $\gamma$ for some $\gamma>0$, \itshape{i.e.}, $||T||_\infty<\gamma$,
\item [b)] There exists $P_H =P_H^\top > 0$ such that
      \begin{align}
              \left[
              \begin{array}{ccc}
               \mbox{\rm He}({\cal A},P_H)     & P_H {\cal B} & {\cal C}^\top \\
               {\cal B}^\top P_H   & -\gamma I & {\cal D}^\top \\
               {\cal C}            & {\cal D}  & -\gamma I
              \end{array}
              \right]<0. 
              \label{eq:bounded_real_lemma_Hinfty}
      \end{align}
\end{enumerate}
\end{lemma}
\begin{remark}
The condition in item $b)$ of Lemma~\ref{Hinf_lemma} is the so-called Bounded Real Lemma condition; see, {\it e.g.}, \cite{anderson.1973.network, 90.scherer.dssertation}. 
\end{remark}

\section{Dilated LMI formulation for the interconnected observers}
\label{app:dilatedLMI}
\begin{proposition}
\label{prop:alt_LMI_equivalenceprop}
Given $\sigma\geq 0$, the rate of convergence of error system \eqref{eq:synthesis_model} is greater than or equal to $\sigma$ and the $H_\infty$ gain from $m$ to $\bar{e}$ is less than or equal to $\gamma$ if there exist real matrices $P_S$, $P_H$, $Q_D$, $Q_H$ and real numbers $r_D>0$, $r_H>0$ such that the following optimization problem (LMI) is feasible:
\begin{subequations}
\begin{align}
& \inf \gamma \nonumber \\
\mbox{s.t.\ } & \left[\begin{array}{cc}
     \mbox{\rm He}({\cal A}, Q_D) + 2 \sigma P_S & P_S - Q_D^\top + r_D {\cal A}^\top Q_D\\
     P_S - Q_D + r_D Q_D^\top {\cal A}  & -r_D(Q_D + Q_D^\top)\\
    \end{array}\right] < 0, \label{neq:eigen_nobs}\\
& \left[\begin{array}{cccc}
     \mbox{\rm He}({\cal A}, Q_H) &  P_S - Q_H^\top + r_H {\cal A}^\top Q_H &  Q_H^\top {\cal B}  & {\cal C}^\top \\
     P_S - Q_H + r_H Q_H^\top {\cal A}  & - r_H(Q_H+Q_H^\top) & r_H Q_H^\top {\cal B} & 0\\
     {\cal B}^\top Q_H & r_H {\cal B}^\top Q_H & -\gamma I & 0\\
     {\cal C}  & 0 & 0 & -\gamma I
    \end{array}\right] <0, \label{eq:synthesis_multi_equivalent_Hinfty} \\
& P_S =P_S^\top >0, P_H = P_H^\top >0.  
\end{align}
\label{eq:synthesis_multi_equivalent}
\end{subequations}
\end{proposition}
\begin{proof}
The proof follows from \cite[Theorem $1$ and Theorem $2$]{08.Xie.dilatedLMI}, see also \cite{04.Ebihara.DilatedLMI,05.Ebihara.dilatedLMI}.
\end{proof}
}

\IfJtwoC{
\begin{IEEEbiography}
[{\includegraphics[width=1.in,keepaspectratio]{Yuchun2015biopic1.eps}}]
{Yuchun Li} 
received his B.S. and M.S. degree in Mechanical Engineering from Zhejiang University, Hangzhou, China, in 2007 and 2010, respectively. Currently, he is pursuing a Ph.D. in the Hybrid Systems Laboratory in the Department of Computer Engineering at the University of California, Santa Cruz. 
His research interests include modeling, stability, observer design, control and robustness analysis of hybrid systems. 
\end{IEEEbiography}
\begin{IEEEbiography}
[{\includegraphics[width=1.in,keepaspectratio]{Sanfelice5x7a2012smallIEEE.eps}}]
{Ricardo G. Sanfelice} 
is an Associate Professor of Computer Engineering, University of California at Santa Cruz, CA, USA. He received his M.S. and Ph.D. degrees in 2004 and 2007, respectively, from the University of California, Santa Barbara. Prof. Sanfelice is the recipient of the 2013 SIAM Control and Systems Theory Prize, the National Science Foundation CAREER award, the Air Force Young Investigator Research Award, and the 2010 IEEE Control Systems Magazine Outstanding Paper Award. His research interests are in modeling, stability, robust control, observer design, and simulation of nonlinear and hybrid systems with applications to power systems, aerospace, and biology.
\end{IEEEbiography}
}

\end{document}